\UseRawInputEncoding
\documentclass{amsart}

\usepackage{amssymb,amsmath,amsthm}
\usepackage{amsaddr}
\usepackage{amsfonts}
\usepackage{authblk}
\usepackage{graphicx}
\usepackage{epstopdf}
\usepackage{caption}
\usepackage{subcaption}
\usepackage{hyperref}
\usepackage{color}
\usepackage{marginnote}
\usepackage{float}
\usepackage{graphicx}
\usepackage{comment}
\usepackage{tikz-cd}
\usetikzlibrary{arrows}
\usetikzlibrary{intersections,positioning}
\tikzset{>=latex}
\usepackage{lipsum}%
\usepackage{a4wide}
\allowdisplaybreaks[4]

\newcommand{\C}{\mathbb{C}}
\newcommand{\D}{\mathbb{D}}
\newcommand{\R}{\mathbb{R}}
\newcommand{\Z}{\mathbb{Z}}

\newcommand{\N}{\mathbb{N}}
\newcommand{\M}{\mathcal{M}}

\newcommand{\J}{\mathcal{J}}

\renewcommand{\P}{\mathcal{P}}
\newcommand{\F}{\mathcal{F}}
\newcommand{\V}{\mathcal{V}}

\renewcommand{\sl}{\text{sl}}
\newcommand{\link}{\mathrm{link}}
\newcommand{\wind}{\mathrm{wind}}

\newcommand{\U}{\mathcal U}
\newcommand{\lk}{\mathrm{link}}
\renewcommand{\S}{\mathcal S}
\newcommand{\ind}{\mathrm{ind}}

\newcounter{newcounter}[section]
\numberwithin{equation}{section}
\numberwithin{newcounter}{section}
\numberwithin{figure}{section}
\numberwithin{footnote}{section}

\newcommand{\authorfootnotes}{\renewcommand\thefootnote{\@fnsymbol\c@footnote}}%

\newtheorem{thm}[newcounter]{Theorem}
\newtheorem{defi}[newcounter]{Definition}
\newtheorem{prop}[newcounter]{Proposition}
\newtheorem{lem}[newcounter]{Lemma}
\newtheorem{cor}[newcounter]{Corollary}
\newtheorem{rem}[newcounter]{Remark}

\title[Finite energy foliations in the restricted three-body problem]{Finite energy foliations and global dynamics in the restricted three-body problem}

\begin{document}

\email{liulei30@email.sdu.edu.cn}

\email{psalomao@sustech.edu.cn}

\maketitle

\begin{center}

\normalsize
\authorfootnotes
Lei Liu\textsuperscript{1} and
Pedro A. S. Salom\~ao\textsuperscript{2}  \par \bigskip

\textsuperscript{1}School of Mathematics, Shandong University  \par

\textsuperscript{2} Shenzhen International Center for Mathematics, SUSTech

\end{center}

\begin{abstract}
We establish a general criterion for the existence of finite energy foliations on contact three-manifolds with boundary, imposing strong global constraints on the associated Reeb flows.
Our main abstract result shows that a configuration of periodic orbits, consisting of hyperbolic boundary orbits of Conley-Zehnder index $2$ and an interior orbit of index at least $3$, gives rise to a finite energy foliation, provided that no additional contractible periodic orbit satisfies a specific rotation, linking, and action condition. This identifies the precise dynamical obstruction to the existence of such foliations. These foliations organize the flow in a regime where the dynamics is typically chaotic, and imply the existence of infinitely many periodic orbits and infinitely many homoclinic/heteroclinic orbits to the Lyapunov orbits.

We apply this result to the circular planar restricted three-body problem. For mass ratios sufficiently close to $\frac{1}{2}$ and energies slightly above the first Lagrange value, the regularized energy surface $\mathbb{R}P^3 \# \mathbb{R}P^3$ admits a finite energy foliation whose binding consists of the retrograde orbits around the primaries together with the Lyapunov orbit near the first Lagrange point. Moreover, the convexity of the critical energy surface provides a proof of Birkhoff's retrograde orbit conjecture for mass ratios sufficiently close to $\frac{1}{2}$ and all energies below the first Lagrange value.

\end{abstract}

\tableofcontents
\section{Introduction}

Finite energy foliations, introduced by Hofer, Wysocki, and Zehnder, provide a geometric mechanism for organizing global dynamics of Reeb flows via pseudo-holomorphic curves. In this paper, we identify a configuration of periodic orbits consisting of an interior orbit together with distinguished separating orbits that, under a sharp nonexistence condition for additional contractible periodic orbits, gives rise to a finite energy foliation with minimal topology. The resulting foliations impose strong global constraints on the dynamics: they organize the flow by planes and cylinders and force the existence of infinitely many periodic orbits and homoclinic or heteroclinic connections to the separating orbits.

To describe the abstract setting, consider a contact three-manifold $\mathcal M$ with non-empty boundary, equipped with a Reeb vector field. Assume that each boundary component is a two-sphere $\mathcal S$ containing a hyperbolic periodic orbit $P_2$ with Conley-Zehnder index $2$, referred to as a {\em Lyapunov orbit}. The vector field is transverse to $\mathcal S$ and points in opposite directions along the two hemispheres of $\mathcal{S}\setminus P_2$. Assume further that the interior of $\mathcal M$ contains a periodic orbit, referred to as a {\em retrograde orbit}, such that its contractible cover has  Conley-Zehnder index at least $3$.

Our main abstract result characterizes when the Lyapunov orbits in $\partial \mathcal{M}$ together with the retrograde orbit in $\mathcal{M} \setminus \partial \mathcal{M}$ form the binding of a finite energy foliation with minimal topology. More precisely, the result shows that the only obstruction to the existence of such foliations is the presence of contractible periodic orbits satisfying specific rotation, linking, and action conditions.

For energies slightly above a critical value in typical Hamiltonian systems with saddle-center equilibria, a subset $\mathcal{M}$ of the energy surface, called a {\em chamber}, has boundary components in the neck regions and contains an interior orbit of the type described above.

In the particular case of the circular planar restricted three-body problem, the regularized component near the primaries is diffeomorphic to $\R P^3 \# \R P^3$. Each chamber $\mathcal{M}$ around a primary is diffeomorphic to $\R P^3 \equiv L(2,1)$ with an open ball removed, and contains a retrograde orbit around the corresponding primary. Its boundary $\partial \mathcal M$ consists of a single component containing the Lyapunov orbit near the first Lagrange point.

The applicability of the abstract result requires a detailed convexity analysis of the critical energy surface. We establish a general criterion for convexity and verify it for the mass ratio $\mu=\frac{1}{2}$ in elliptic-hyperbolic coordinates. This condition excludes the existence of the obstructing periodic orbits. Hence, each chamber admits a finite energy foliation, and the regularized component $\R P^3 \# \R P^3$ carries a $3-2-3$ foliation.

These foliations have a rigid structure: their leaves are planes asymptotic to the retrograde orbits, and cylinders connecting the retrograde orbits to the Lyapunov orbit. Their existence has strong dynamical consequences, including the existence of infinitely many periodic orbits and infinitely many homoclinic connections to the Lyapunov orbit.

As a further consequence of convexity, we also prove Birkhoff's retrograde orbit conjecture for mass ratios sufficiently close to $\frac{1}{2}$ and all energies below the first Lagrange value.

Our approach to the main abstract result is based on the study of pseudo-holomorphic curves in the symplectization of the energy surface. The construction starts from a Bishop family of pseudo-holomorphic disks with suitable boundary conditions, following \cite{Hofer1993}. The theory of pseudo-holomorphic curves in symplectizations was developed by Hofer, Wysocki, and Zehnder \cite{Hofer1993, Hofer1998, char1, char2, HWZ1996, HWZII, HWZIII, convex, fols}, and was first applied to the restricted three-body problem in \cite{AFKP2012}.

The mechanism developed in this paper is expected to extend beyond the restricted three-body problem to a broad class of Hamiltonian systems exhibiting saddle-center type dynamics. In many such systems, neck regions are organized by Lyapunov orbits, suggesting the existence of chambers bounded by distinguished periodic orbits as in the present work. Potential applications include Hill¡¯s lunar problem, H¨¦non-Heiles type Hamiltonians, magnetic-mechanical systems, and transition state models in molecular dynamics. In these settings, finite energy foliations may provide a geometric framework for understanding global transport, homoclinic and heteroclinic phenomena, entropy, and the organization of chaotic dynamics through pseudo-holomorphic curves.

\subsection{The setting}

Let $\M=\mathcal{M}^3$ be a smooth three-manifold diffeomorphic to a lens space $L(p,q)$ with finitely many disjoint regular open balls removed. Recall that $p\geq q\geq 1$ are relatively prime and that $L(p,q)$ is the quotient of
\[
S^3=\{(z_1,z_2)\in \C^2 : |z_1|^2+|z_2|^2=1\}
\]
by the free $\Z_p$-action generated by
\[
(z_1,z_2)\mapsto (e^{2\pi i/p}z_1,e^{2\pi i q/p}z_2).
\]

The boundary components $\mathcal{S}_i\subset \partial \M$, $i=1,\ldots,l$, are diffeomorphic to $S^2$. Let $\alpha$ be a contact form on $\M$ defining the standard contact structure of $L(p,q)$, and denote by $\mathcal{R}$ its Reeb vector field.

Assume that each $\mathcal{S}_i$ contains a hyperbolic periodic orbit $P_{2,i}\subset \mathcal{S}_i$ with Conley--Zehnder index $2$, and that $\mathcal{R}$ is transverse to $\mathcal{S}_i$ and points in opposite directions along the two hemispheres of $\mathcal{S}_i\setminus P_{2,i}$.

Assume further that there exists a periodic orbit $P_3\subset \M\setminus \partial \M$ which is $p$-unknotted with self-linking number $-1/p$. That is, there exists an immersed disk $\mathcal{D}\hookrightarrow \M\setminus \partial \M$ whose interior is embedded and whose boundary is a $p$-fold cover of $P_3$. The disk $\mathcal{D}$ is called a $p$-disk for $P_3$, and its characteristic foliation has a unique singular point which is positive and nicely elliptic.

Denote by $P_3^j$ the $j$-th iterate of $P_3$, and set the binding
\[
\P = \{P_3, P_{2,1},\ldots,P_{2,l}\}.
\]

\begin{defi}\label{defi_TF}
A weakly convex foliation $\mathcal{F}$ of $\M$ adapted to $\alpha$ and $\mathcal{P}$ is a singular foliation whose singular set is $\cup_{P\in \mathcal{P}} P$, and whose regular leaves are properly embedded punctured spheres transverse to the Reeb vector field.

Each leaf $\dot \Sigma\subset \M\setminus \cup_{P\in \mathcal{P}} P$ is diffeomorphic to $\C P^1\setminus \Gamma$, $0<\#\Gamma<\infty$, and is asymptotic at each puncture to some orbit $P_z\in\mathcal{P}$. The asymptotic limits are distinct, and exactly one puncture is positive, i.e., the orientation induced by the leaf coincides with the flow orientation.

If $\dot\Sigma$ is asymptotic to $P_3^p$, then the corresponding puncture is positive and all other punctures are negative and asymptotic to distinct $P_{2,j}$. The hemispheres of $\mathcal{S}_j\setminus P_{2,j}$ are also leaves of $\mathcal{F}$.
\end{defi}

\medskip

These foliations arise as projections of finite energy foliations in the symplectization $\R\times \M$. Their leaves are images of embedded finite energy $J$-holomorphic curves, where $J$ is an almost complex structure adapted to $\alpha$.

For generic $J$, the leaves have minimal topology: planes asymptotic to $P_3^p$, cylinders connecting $P_3^p$ to $P_{2,j}$, and planes projecting to the hemispheres of $\mathcal{S}_j\setminus P_{2,j}$. There are exactly $l$ one-parameter families of planes and, for each $j$, a unique rigid cylinder from $P_3^p$ to $P_{2,j}$ separating these families.

Let $J$ be an $\R$-invariant almost complex structure on $\R\times \M$ compatible with $d(e^a\alpha)$ and satisfying $J\partial_a=\mathcal{R}$. The leaves of the foliation are images of maps
\[
\tilde u=(a,u):\C P^1\setminus \Gamma \to \R\times \M,
\]
with finite Hofer energy. Trivial cylinders $\R\times P, P \in \mathcal P,$ over the binding orbits are also included.

\begin{defi}
A finite energy foliation $\tilde{\mathcal{F}}$ adapted to $\alpha$, $\mathcal{P}$, and $J$ is a foliation of $\R\times \M$ whose leaves are trivial cylinders over $\mathcal{P}$ or embedded finite energy $J$-holomorphic curves with uniformly bounded energy. Its projection to $\M$ is a weakly convex foliation.
\end{defi}

\subsection{The main abstract result}

Assume that $\alpha$ is the restriction to $\M$ of a contact form on $L(p,q)$ defining the standard tight contact structure $\xi_{p,q}$ induced from the Liouville form restricted on $S^3$ and the natural projection from $S^3$ to $L(p,q)$. We retain the notation introduced above. Let $\P(\alpha)$ denote the set of periodic orbits of $\alpha$, and for $P\in\P(\alpha)$ denote by $\mathcal{A}(P)= \int_P \alpha = T$ its action or period, and by $\rho(P)$ its rotation number if $P$ is contractible. We always assume that the contractible iterate $P_3^p$ satisfies $\rho(P_3^p)>1$, and that each Lyapunov orbit is hyperbolic and satisfies $\rho(P_{2,j})=1$. We also assume the existence of $J$ on $\R \times \mathcal{M}$ so that the hemispheres of $\S_j \setminus P_{2,j}$ are projections of nicely embedded $J$-holomorphic planes asymptotic to $P_{2,j}$. Such planes are unique and always exist in our concrete applications.

Let $\mathcal{D}\subset \mathcal M$ be a $p$-disk for $P_3$. Its $|d\alpha|$-area is defined as
$$
\mathcal{S}(\mathcal{D},\alpha):=\int_{\mathcal{D}}|d\alpha|.
$$
We consider the set $\P'\subset \P(\alpha)$ of contractible periodic orbits $P'\subset \M\setminus (\partial \M \cup P_3)$ satisfying
\begin{equation}\label{main assumptions}
\rho(P')=1, \qquad \lk(P',P_3^p)=0, \qquad \mathcal{A}(P') \leq \mathcal{S}(\mathcal{D},\alpha).
\end{equation}
The following theorem gives a sufficient condition for the existence of a finite energy foliation with minimal topology and binding $\mathcal P$.

\begin{thm}\label{main1}
Assume that $\mathcal{M}$, $\alpha$, $\P$, and $J$ satisfy the above conditions. If $\P'=\emptyset$, then there exist an almost complex structure $J'$, which is $C^\infty$-arbitrarily close to $J$ and coincides with $J$ on a neighborhood of $\partial \M$, and a finite energy foliation $\tilde \F$ adapted to $\alpha$, $\P$, and $J'$, whose leaves are planes and cylinders. In particular, its projection to $\M$ is a weakly convex foliation.
\end{thm}

The construction is based on a Bishop family of $J$-holomorphic disks with boundary on a $p$-disk $\{0\}\times\mathcal{D}$ for $P_3$ in $\R\times \mathcal M$, whose characteristic foliation has a unique singularity, which is nicely elliptic. This analysis was first introduced by Hofer in \cite{Hofer1993} in his proof of the Weinstein conjecture for overtwisted contact three-manifolds, and later developed in \cite{char1, char2}. The disk boundaries in the Bishop family is radially monotone from the singularity towards $\partial \mathcal{D} \equiv P_3^p$. If no bubbling occurs before the boundary reaches $P_3^p$, the family converges in the SFT sense to a holomorphic building whose top level yields planes or cylinders asymptotic to the binding orbits. These curves generate the desired foliation via standard Fredholm, gluing, and compactness arguments.
If bubbling occurs, the SFT-limit contains a half-cylinder with boundary on the $p$-disk and negative end at some $P_{2,j}$, together with a plane projecting to a hemisphere of $\mathcal{S}_j\setminus P_{2,j}$. Gluing the half-cylinder with the opposite plane to $P_{2,j}$ produces a new Bishop family closer to $P_3^p$, and an iteration of this procedure leads to the previous case.

For this strategy to work, two ingredients are crucial. First, $J$ may need to be perturbed to avoid non-regular curves in the compactness argument, allowing the application of Fredholm theory and the gluing theorem. Second, the uniqueness of disks in the Bishop family, as discussed in \cite{AH}, ensures that the Bishop family obtained by gluing the half-cylinder with the opposite rigid plane gets closer to $P_3^p$.

As an immediate consequence of the foliation, one obtains the existence of connecting orbits between the Lyapunov orbits.

\begin{cor}\label{cor_homoclinic}
Under the assumptions of Theorem \ref{main1}, let $P_{2,j}$ be the Lyapunov orbit with maximal action. Then $P_{2,j}$ admits a homoclinic orbit or a heteroclinic orbit to some $P_{2,k}$, $k\neq j$. In particular, if $l=1$, then $P_{2,1}$ admits a homoclinic orbit in $\M$.
\end{cor}

\subsection{The main application}

We explain how Theorem \ref{main1} applies to the circular planar restricted three-body problem for mass ratios near a fixed value and energies slightly above the first Lagrange value. In particular, we obtain $3-2-3$ foliations and deduce the existence of infinitely many periodic orbits and homoclinic orbits to the Lyapunov orbit.

Consider two primaries moving along circular orbits about their center of mass. A massless satellite moves in the same plane under Newtonian attraction. In rotating coordinates, the Hamiltonian is
\[
H_\mu(p,q)=\frac{1}{2}|p+iq|^2-\frac{\mu}{|q-(1-\mu)|}-\frac{1-\mu}{|q+\mu|}-\frac{1}{2}|q|^2,
\]
where $q=q_1+iq_2\in \C\setminus\{-\mu,1-\mu\}$ is its position, $p=p_1+ip_2\in \C$ is its momentum. The mass ratio $0<\mu<1$ is the mass of the primary at $1-\mu\in \C$, and $1-\mu$ is the mass of the primary at $-\mu\in \C$. The primaries are called the moon and the earth, respectively.

The Hamiltonian $H_\mu$ has five critical points $l_1(\mu),\ldots,l_5(\mu)$, called Lagrange points, with corresponding critical values $L_1(\mu)<L_2(\mu)\leq L_3(\mu)<L_4(\mu)=L_5(\mu)$ called Lagrange values. Notice that $L_2(1/2)=L_3(1/2)$ and $l_2,l_3$ are interchanged as $\mu$ crosses $1/2$. The Lagrange points are the rest points of the Hamiltonian flow of $H_\mu$, and $l_1(\mu)$ lies between the primaries. For $E<L_1(\mu)$, the energy surface has three components $\mathcal{M}^e_{\mu,E}, \mathcal{M}^m_{\mu,E}, \mathcal{M}^u_{\mu,E}$, the first two components project to punctured disk-like domains about the primaries, and the third one projects to a neighborhood of $\infty$. For $E=L_1(\mu)$, the components $\mathcal{M}^e_{\mu,E}$ and $\mathcal{M}^m_{\mu,E}$ meet at $l_1(\mu)$. For $L_1(\mu)<E<L_2(\mu)$, they merge into a regular component $\mathcal{M}^{e\# m}_{\mu,E}$, a connected sum of $\mathcal{M}_{\mu,E}^e$ and $\mathcal{M}_{\mu,E}^m$, which projects to a two-punctured disk-like domain about the earth and the moon.

To treat collisions, we regularize the flow using elliptic-hyperbolic coordinates. The regularized Hamiltonian $\hat H_{\mu,E}$ is defined on $\R^2\times (\R\times \R/2\pi\Z)$ and has the form
$$\hat H_{\mu,E}(y,x)= \frac{1}{2}|y+ F(x)|^2 + V_{\mu,E}(x),$$
where $F(x) = (f_1(x_2),f_2(x_1))$ and $V_{\mu,E}(x)$ are smooth functions on $\R\times \R / 2\pi \Z$, with $V_{\mu,E}$ smoothly depending on $(\mu,E)$. Up to time reparametrization, its zero level corresponds to $H^{-1}(E)$. After quotienting by the antipodal symmetry, the regularized components become compact manifolds:
\begin{equation}\label{energy_surfaces}
\begin{aligned}
S^3 &\stackrel{2:1}{\longrightarrow} \mathcal{M}^e_{\mu,E}, \mathcal{M}^m_{\mu,E} \equiv \R P^3, && E<L_1(\mu),\\
S^1\times S^2 &\stackrel{2:1}{\longrightarrow} \mathcal{M}^{e\# m}_{\mu,E} \equiv \R P^3 \# \R P^3, && L_1(\mu)<E<L_2(\mu).
\end{aligned}
\end{equation}

In \cite{AFKP2012}, Albers, Frauenfelder, van Koert, and Paternain observed that for energies up to slightly above $L_1(\mu)$, the regularized energy hypersurfaces are of contact type. Thus the regularized flow is equivalent to a Reeb flow. For $E<L_1(\mu)$, one obtains a Reeb flow on $(\R P^3,\xi_0)$. Birkhoff \cite{Birkhoff1915} used the shooting method to prove the existence of a retrograde orbit, that is, a periodic orbit projecting to a simple closed curve around the primary moving opposite to the rotating system. He raised the question of whether the retrograde orbit bounds a disk-like global surface of section. For $E$ slightly above $L_1(\mu)$, the component $\mathcal{M}^{e\# m}_{\mu,E}$ contains an index-$2$ hyperbolic periodic orbit $P_2=P_{2,E}$ in the neck region, the Lyapunov orbit. It bounds a two-sphere $\S$ which separates $\mathcal{M}^{e\# m}_{\mu,E}$ into two chambers whose closures, still denoted $\mathcal{M}^e_{\mu,E}$ and $\mathcal{M}^m_{\mu,E}$, are contactomorphic to $(\R P^3,\xi_0)$ with an open ball removed.

The following theorem shows that the interiors of $\mathcal{M}_{\mu, E}^e$ and $\mathcal{M}_{\mu, E}^m$ possess retrograde orbits $P_3^e$ and $P_e^m$, respectively.  These orbits are $2$-unknotted with self-linking number $-1/2$ and admit $2$-disks with a unique nicely elliptic singularity.

\begin{thm}\label{thm_retrograde}
Fix $(\mu_0,E_0)$ with $0<\mu_0<1$ and $E_0\leq L_1(\mu_0)$. Then there exist continuous families of retrograde orbits $P_3^e=P^e_{3,\mu,E}$ and $P_3^m=P^m_{3,\mu,E}$, defined for $(\mu,E)$ sufficiently close to $(\mu_0,E_0)$, with uniformly bounded action, and a continuous family of $2$-disks for $P_3^e$ with uniformly bounded $|d\alpha|$-area.
\end{thm}
The existence of a retrograde orbit as in Theorem \ref{thm_retrograde} follows from the usual Birkhoff's shooting method. This method considers trajectories issuing perpendicularly from certain open intervals in the $q_1$-axis containing the primary as an endpoint. Such intervals are shown to be uniformly away from $l_1(\mu)$ as $E \to L_1(\mu)$ and, moreover, they parametrize two non-self-intersecting real-analytic curves in the interior of the rectangle $Q:=[-\pi/2,\pi/2] \times [-M,0]$, for some $M>0$. At the endpoints, these curves tend to $\partial Q$. An intersection point between these curves correspond to a retrograde orbit. A topological crossing of such curves always exists and persists under small perturbations of $(\mu, E)$, thus giving at least one family of retrograde orbits near $(\mu_0, E_0)$ as in the statement.

Our goal is to construct weakly convex foliations on $\mathcal{M}^{e\# m}_{\mu,E}$ whose binding consists of the Lyapunov orbit $P_2$ and the retrograde orbits $P_3^e, P_3^m$. To fix the separating sphere $\S$, we choose an almost complex structure $J$ on the symplectization such that the hemispheres of $\S\setminus P_2$ are projections of $J$-holomorphic planes asymptotic to $P_2$, located in the neck region and converging to $l_1(\mu)$ as $E\to L_1(\mu)^+$. The following result provides such a choice.

\begin{thm}\label{thm_alphaJ}
Let $0<\mu_0<1$. For every $(\mu,E)$ sufficiently close to $(\mu_0,L_1(\mu_0))$ with $E>L_1(\mu)$:
\begin{itemize}
\item[(i)] There exists a contact form $\alpha=\alpha_{\mu,E}=i_{Y}\omega_0$ on $\mathcal{M}^{e\# m}_{\mu,E}\equiv \R P^3 \# \R P^3$ whose Reeb flow is equivalent to the regularized Hamiltonian flow. Here, $Y=Y_{\mu, E}$ is a Liouville vector field, defined on a neighborhood of $\mathcal{M}_{\mu, E}^{e\#m}$ in $\R^4$ and transverse to $\mathcal{M}^{e\# m}_{\mu, E}$, and $\omega_0$ is the canonical symplectic form $\sum_{i=1,2} dp_i \wedge dq_i$.
\item[(ii)] There exists a compatible almost complex structure $J=J_{\mu,E}$ on $\R \times \mathcal{M}^{e\#m}_{\mu,E}$  adapted to $\alpha$ that admits a pair of $J$-holomorphic planes asymptotic to $P_{2}$ through opposite directions. The closure of their projections to $\mathcal{M}^{e\#m}_{\mu,E}$ form a regular two-sphere $\S=\S_{\mu,E}$ containing $P_{2}$. Furthermore, $\text{dist}(\S,l_1(\mu)) \to 0$  as $E \to L_1(\mu)^+$ uniformly in $\mu$.

\item[(iii)] The contact form $\alpha$ in (i) $C^\infty_{\text{loc}}$-converge uniformly in $\mu$ to a limiting contact form $\alpha_{\mu,L_1(\mu)}$ on the critical level away from $l_1(\mu)$ as $E \to L_1(\mu)^+$.
\end{itemize}
\end{thm}

\medskip

Fix $\alpha$ and $J$ as above. The sphere $\S$ separates $\mathcal{M}^{e\# m}_{\mu,E}$ into chambers $\mathcal{M}^e_{\mu,E}$ and $\mathcal{M}^m_{\mu,E}$, each containing a retrograde orbit as mentioned before. We have $P_2=P_{2,E}\subset  \S=\partial \mathcal{M}_{\mu,E}^e = \partial \mathcal{M}_{\mu,E}^m,$ $P_3^e \subset \mathcal{M}_{\mu,E}^e \setminus \mathcal{S}$ and $P_3^m \subset \mathcal{M}_{\mu,E}^m \setminus \mathcal{S}.$ We aim to construct finite energy foliations projecting to special weakly convex foliations of $\mathcal{M}_{\mu, E}^{e \# m}$, called $3-2-3$ foliations.

\begin{defi}
A $2-3$ foliation of $\mathcal{M}^e_{\mu,E}$ adapted to $\alpha$, $J$, and $\{P_3^e,P_2\}$ is a weakly convex foliation whose leaves consist of two hemispheres in $\S\setminus P_{2}$, a one-parameter family of planes asymptotic to $(P_3^e)^2$, and a rigid cylinder connecting $(P_3^e)^2$ to $P_2$. They are transverse to the flow and consist of projections to $\mathcal{M}_{\mu,E}^e$ of a finite energy foliation in the symplectization. A $3-2-3$ foliation of $\mathcal{M}^{e\# m}_{\mu,E}$ is obtained by combining $2-3$ foliations on both chambers.
\end{defi}

\begin{figure}[hbpt]
\centering
  \includegraphics[width=.75\linewidth]{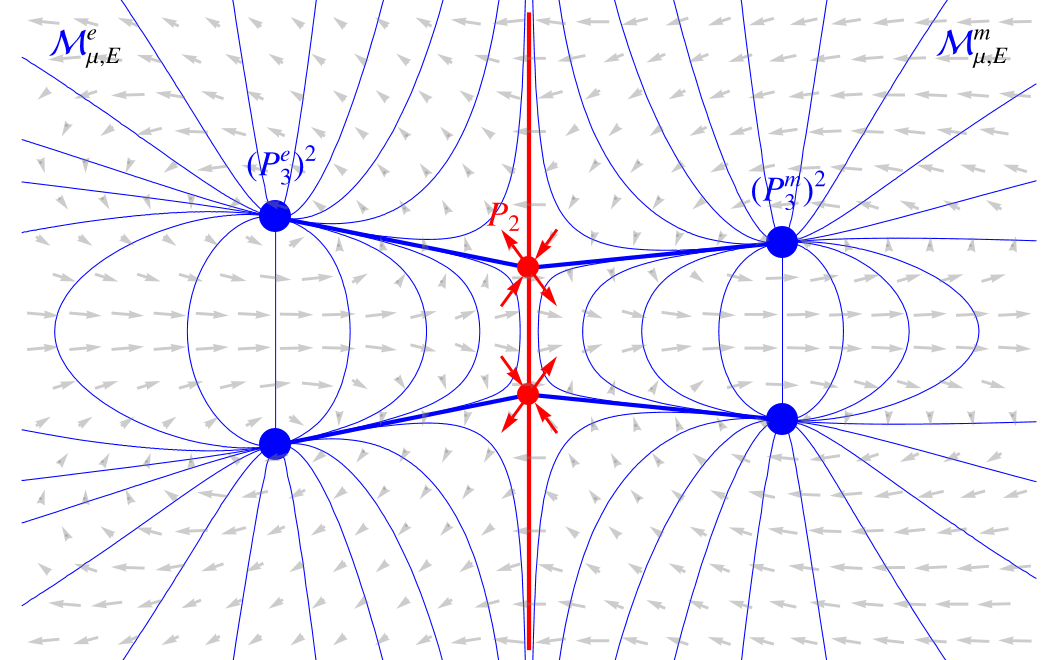}
\caption{The $3-2-3$ foliation on the regularized component $\mathcal{M}_{\mu, E}^{e\# m} \equiv \R P^3 \# \R P^3$ for mass ratios close to $1/2$ and energies slightly above the first Lagrange value. The rigid cylinders (bold blue) connect the double cover of the retrograde orbits $P_3^e$ and $P_3^m$ to the Lyapunov orbit $P_2$ near the first Lagrange point. The rigid planes asymptotic to $P_2$ (bold red) separate $\mathcal{M}_{\mu,E}^e$ and $\mathcal{M}_{\mu,E}^m$.}
\label{fig: 3-2-3 foliation}
\end{figure}

To apply our main abstract theorem, Theorem \ref{main1}, one needs a crucial index estimate of periodic orbits near the critical level, which states that periodic orbits passing sufficiently close to $l_1(\mu)$ must have a high Conley-Zehnder index.

\begin{thm}\label{thm_muRSl1}
Let $0<\mu_0 <1$. Given $N\in \N$, there exists an open neighborhood $\U_{N}\subset \R^2 \times (\R \times \R / 2\pi \Z)$ of the saddle-center singularities $S_\pm(\mu_0),$ corresponding to the first Lagrange point $l_1(\mu_0)$, such that any periodic orbit $P$ entering it, and which is not a cover of the Lyapunov orbit, has Conley-Zehnder index $\mu_{\text{CZ}}(P)>N$.
\end{thm}

We now state the main criterion to the existence of $3-2-3$ foliations on the regularized component $\mathcal{M}^{e \# m}_{\mu, E}$ for $E$ slightly above $L_1(\mu)$.

\begin{thm}\label{thm_main3}
Let $0<\mu_0<1$,  and let $\alpha,J$ and $\S$ be defined for $(\mu,E)$ sufficiently close to $(\mu_0, L_1(\mu_0))$, with $E> L_1(\mu)$, as in Theorem \ref{thm_alphaJ}. Let $P_3^e$ be the continuous family of retrograde orbits and $\mathcal{D}=\mathcal D_{\mu,E}$ be the $2$-disk for $P_3^e$ in Theorem \ref{thm_retrograde} for $(\mu, E)$ sufficiently close to $(\mu_0, L_1(\mu_0))$. Assume the absence of contractible periodic orbits $P'\subset \mathcal{M}^e_{\mu_0, L_1(\mu_0)}\setminus (P_{2} \cup P^e_{3})$ at the critical level  satisfying
$$\rho(P')=1, \quad \link(P', P^e_{3})=0 \quad \mbox{ and } \quad \mathcal{A}(P') \leq \S(\mathcal{D}),$$
and that $\mu_{\mathrm{CZ}}((P_3^e)^2)\geq 3$. Then for $(\mu,E)$ sufficiently close to $(\mu_0,L_1(\mu_0))$, with $E>L_1(\mu)$, the following statements hold:
\begin{itemize}
\item[(i)] $P_2$ is the unique contractible orbit with rotation number $1$, unlinked with $P_3^e$ in $\mathcal{M}_{\mu,E}^e$ and action $\leq \S(\mathcal{D}, \alpha)$;

\item[(ii)] $\mathcal{M}^e_{\mu,E}$ admits a $2-3$ foliation whose binding orbits consist of the retrograde orbit $P_3^e$ and the Lyapunov orbit $P_2$ around the first Lagrange point $l_1(\mu)$;

\item[(iii)] $\mathcal{M}^e_{\mu, E}$ contains infinitely many periodic and homoclinic orbits to the Lyapunov orbit near $l_1(\mu)$. Moreover, if the branches in $\mathcal{M}_{\mu, E}^e$ of the stable and unstable manifolds of the Lyapunov orbit do not coincide, then the topological entropy of the flow on $\mathcal{M}_{\mu, E}^e$ is positive.
\end{itemize}
An analogous statement holds for $\mathcal{M}^m_{\mu,E}$. Moreover, if the conditions above are satisfied for both  $\mathcal{M}_{\mu_0,L_1(\mu_0)}^e$ and $\mathcal{M}_{\mu_0, L_1(\mu_0)}^m$, then the regularized component $\mathcal{M}_{\mu,E}^{e\# m} \equiv \R P^3 \# \R P^3$ admits a $3-2-3$ foliation whose binding orbits are $P_3^e, P_3^m$ and $P_2$, for every $(\mu, E)$ sufficiently close to $(\mu_0, L_1(\mu_0))$, with $E> L_1(\mu)$.
\end{thm}

To establish the estimate $\mu_{\mathrm{CZ}}((P_3^e)^2)\geq 3$, we introduce the following definitions.
\begin{defi}
Let $0<\mu<1$, and let $\alpha_{\mu}=\alpha_{\mu,L_1(\mu)}$ be the contact form on the regularized critical subset $\mathcal{M}^e_{\mu}=\mathcal{M}^e_{\mu,L_1(\mu)}$ as in Theorem \ref{thm_alphaJ}.
\begin{itemize}
\item[(i)] We say that $\mathcal{M}_{\mu}^e$ is dynamically convex if the Conley-Zehnder index of every contractible periodic orbit of $\alpha_\mu$ is at least $3$.
\item[(ii)] Let $\widetilde{\mathcal{M}}_{\mu}^e$ be a component of the lift of the regularized critical subset $\mathcal{M}_{\mu}^e$ to $\R^4$, which double covers $\mathcal{M}_\mu^e$. In particular, $\widetilde{\mathcal{M}}_{\mu}^e$ is a topological three-sphere with antipodal symmetry and two opposite saddle-center singularities $S_\pm=S_\pm(\mu)$ corresponding to the first Lagrange point $l_1(\mu)$.  We say that $\mathcal{M}_{\mu}^e$ is strictly convex if  $\widetilde{\mathcal{M}}_{\mu}^e$ bounds a convex subset of $\R^4$ and all sectional curvatures of $\widetilde {\mathcal{M}}_{\mu}^e\setminus \{S_\pm\}$ are positive.
\end{itemize}
A similar definition holds for $\mathcal{M}_{\mu}^m=\mathcal{M}^m_{\mu,L_1(\mu)}$.
\end{defi}

A key point in the argument is that dynamical convexity of the critical energy surface $\mathcal M^e_\mu$ and $\mathcal M^m_\mu$  provides a direct mechanism to exclude the obstructions in Theorem \ref{main1} for slightly higher energies. In practice, one of the few effective ways to establish dynamical convexity is through geometric convexity properties of $\mathcal M^e_\mu$ and $\mathcal M^m_\mu$. In particular, strict convexity implies dynamical convexity, and thus forces all contractible periodic orbits to have Conley-Zehnder index at least $3$, ruling out the configurations responsible for the obstruction.

\begin{prop}\label{prop_convexity}
Let $0<\mu<1$. Strict convexity of the critical chambers $\mathcal{M}_{\mu}^e$ and $\mathcal{M}_{\mu}^m$ implies that all contractible periodic orbits have Conley--Zehnder index at least $3$.
\end{prop}

This proposition is essentially proved in \cite{convex} and \cite{Long02}. Indeed, consider any contractible periodic orbit $P$ in $\mathcal{M}_{\mu}^e$, and let $\tilde P$ denote a lift of $P$ in $\widetilde {\mathcal{M}}_{\mu}^e\subset\R^4$. The index of $\tilde P$ then depends only on the Hamiltonian near $\tilde P$. One may change the Hamiltonian away from $\tilde P$ so that $\tilde P$ lies in a strictly convex regular hypersurface. Hence, its index is at least $3$, including $P_3^e$.

Our main application in the restricted three-body problem relies on establishing the strict convexity of $\mathcal M^e_\mu$ and $\mathcal M^m_\mu$ for $\mu=1/2$.  This is a delicate geometric property, but it provides a direct route to dynamical convexity and hence to the exclusion of the obstructions in Theorem~\ref{main1}.

\begin{thm}\label{thm_3bp2}
For $\mu=1/2$, the regularized components $\mathcal{M}_{\mu, E}^e$ and $\mathcal{M}_{\mu, E}^m$ are strictly convex for all energies $E\leq L_1(1/2)=-2$.
\end{thm}

The proof of Theorem \ref{thm_3bp2} involves a generalization of the results in  \cite{Sa1} for magnetic-mechanical Hamiltonians. The strategy is to analyze the geometry of the regularized Hamiltonian in elliptic-hyperbolic coordinates and reduce convexity to explicit curvature estimates on the energy hypersurface. This reduction allows us to verify strict convexity by controlling a concrete function on the Hill region and exploiting a monotonicity argument that propagates the convexity from the critical level to lower energies.

Combining the previous results yields the main application.

\begin{thm}\label{thm_main3bp}
For $(\mu,E)$ sufficiently close to $(1/2,L_1(1/2))$ with $E>L_1(\mu)$, the flow on $\mathcal{M}^{e\# m}_{\mu,E}$ admits a $3-2-3$ foliation  whose binding orbits are $P_3^e,P_3^m$, and $P_2$. In particular, each chamber $\mathcal{M}_{\mu, E}^e$ or $\mathcal{M}_{\mu, E}^m$ contains infinitely many periodic orbits and infinitely many homoclinic orbits to the Lyapunov orbit $P_2$ near $l_1(\mu)$. Moreover, if the stable and unstable manifolds of $P_2$ do not coincide, then the topological entropy of the flow on $\mathcal{M}_{\mu, E}^{e \# m}$ is positive.
\end{thm}

The $3-2-3$ foliation in Theorem \ref{thm_main3bp} follows from Theorems \ref{thm_main3} and \ref{thm_3bp2} and Proposition~\ref{prop_convexity}.

\begin{rem}\label{rem_McGehee}For small mass ratios, homoclinic orbits to the Lyapunov orbit were studied in the regularized component $\mathcal{M}_{\mu,E}^e$ by McGehee \cite{McGehee1969}, exploiting the integrability of the Rotating Kepler Problem. In fact, it is simple to check the linking and index conditions that imply $\mathcal{P}'=\emptyset$ in Theorem \ref{main1} and thus $\mathcal{M}_{\mu,E}^e$ admits a $2-3$ foliation for $\mu>0$ sufficiently small and $E$ slightly above $L_1(\mu)$.  McGehee's construction is relatively simpler than finding a $2-3$ foliation but somehow deals with the existence of a family of disks bounded by $P_3^e$ that are transverse to the flow.
\end{rem}
\begin{rem}
See \cite{JvK} for a computer-assisted proof of convexity in a certain range of mass ratio and energies below $L_1(\mu)$. Finally, the critical subsets $\mathcal{M}_{\mu, L_1(\mu)}^e$ and $\mathcal{M}_{\mu, L_1(\mu)}^m$ are expected to be dynamically convex for every mass ratio.
\end{rem}

\subsection{Birkhoff's retrograde orbit conjecture}

In \cite{Birkhoff1915}, Birkhoff asked whether the double cover of a retrograde orbit $P_3^e$ bounds a disk-like global surface of section for the flow on the regularized components $\mathcal{M}_{\mu,E}^e, \mathcal{M}_{\mu,E}^m \equiv \R P^3$ for energies $E<L_1(\mu)$. This question is closely related to the existence of direct orbits around the primaries.

The convexity estimates established in Theorem \ref{thm_3bp2} imply that, for mass ratios sufficiently close to $1/2$ and all energies below the first Lagrange value, the components $\mathcal{M}_{\mu,E}^e$ and $\mathcal{M}_{\mu,E}^m$ are dynamically convex, i.e., all contractible periodic orbits have Conley-Zehnder index at least $3$. As a consequence, Birkhoff's conjecture holds in this regime.

\begin{thm}\label{thm_Birkhoff_conj}
There exists $\epsilon_0>0$ such that for every $|\mu-1/2|<\epsilon_0$ and $E<L_1(\mu)$, the following hold:
\begin{itemize}
\item[(i)] The components $\mathcal{M}_{\mu,E}^e$ and $\mathcal{M}_{\mu,E}^m$ are dynamically convex.

\item[(ii)] Every retrograde orbit $P_3^e\subset \mathcal{M}_{\mu,E}^e$ binds a rational open book decomposition whose pages are disk-like global surfaces of section. More generally, the same holds for any periodic orbit transversely isotopic to a Hopf fiber. An analogous statement holds for $\mathcal{M}_{\mu,E}^m$.

\item[(iii)] If $P'$ is a simple periodic orbit corresponding to a fixed point of the return map associated to the global surface of section bounded by $P$ in (ii), then $P\cup P'$ bounds an annulus-like global surface of section. The same holds for $\mathcal{M}_{\mu,E}^m$.
\end{itemize}
\end{thm}

\subsection{Related works}

The theory of pseudo-holomorphic curves in symplectizations was initiated by Hofer \cite{Hofer1993} and developed by Hofer, Wysocki, and Zehnder \cite{HWZ1996, HWZII, HWZIII}. A first major application to dynamics concerns strictly convex hypersurfaces in $\R^4$.

\begin{thm}[Hofer--Wysocki--Zehnder \cite{convex}]
Let $\alpha=f\alpha_0$ be a dynamically convex contact form on $(S^3,\xi_0)$. Then $\alpha$ admits an index-$3$ periodic orbit bounding an open book decomposition whose pages are disk-like global surfaces of section. This open book arises as the projection of a finite energy foliation.
\end{thm}

The existence of finite energy foliations projecting to open book decompositions, all of whose pages are global surfaces of section, was further studied in \cite{char1, char2, hryn, hryn2, HLS, HS1, HS2}, see the surveys \cite{BH, Hofer1998, SH18}.

A fundamental result by Hofer, Wysocki, and Zehnder establishes more general finite energy foliations for generic star-shaped hypersurfaces in $\R^4$.

\begin{thm}[Hofer-Wysocki-Zehnder \cite{fols}] Let $\alpha=f\alpha_0$ be a nondegenerate contact form on the tight three-sphere $(S^3,\xi_0)$. Then, for $J$ in a generic subset $\J_{\text{reg}}(\alpha)\subset \J(\alpha)$ of $d\alpha$-compatible almost complex structures, there exists a finite energy foliation by $J$-holomorphic curves whose projection is a transverse foliation. The binding is formed by simple periodic orbits with self-linking number $-1$ and index $1$, $2$, or $3$. Each leaf is the projection of a $J$-holomorphic curve, has genus zero and satisfies one of the following conditions:
\begin{itemize}
\item[(i)] It is a trivial cylinder over a binding orbit.
\item[(ii)] It has one positive puncture asymptotic to an index-$3$ orbit and an arbitrary number of negative punctures asymptotic to index-$1$ orbits.
\item[(iii)] It has one positive puncture asymptotic to an index-$3$ orbit, one negative puncture asymptotic to an index-$2$ orbit, and an arbitrary number of negative punctures asymptotic to index-$1$ orbits.
\item[(iv)] It has one positive puncture asymptotic to an index-$2$ orbit and an arbitrary number of negative punctures asymptotic to index-$1$ orbits.
\end{itemize}
\end{thm}

Further developments include the construction of finite energy foliations in overtwisted manifolds and criteria for Fredholm regularity \cite{Wendl08, Wendl08b, Wendl10a, Wendl10}. The results in \cite{Wendl10} will be used in the proof of Theorem \ref{main1}. Connected sums of finite energy foliations were studied in \cite{FS}. Finite energy foliations near critical energy surfaces were considered in \cite{dPS1}. The combination of holomorphic curves and asymptotic cycle techniques yields the so-called broken book structures, which yield results on periodic orbits and global surfaces of section \cite{CDR, CDHR, CM, CM2}. 

In celestial mechanics, classical results of Poincar\'e \cite{Po} and Birkhoff \cite{Birkhoff1915} establish the existence of global surfaces of section and periodic orbits in the circular planar restricted three-body problem for energies below the first Lagrange value. Perturbative methods yield annulus-like global surfaces of section near integrable regimes, while shooting arguments produce retrograde and direct orbits. The dynamics at higher energies, including homoclinic connections to Lyapunov orbits, was further analyzed in \cite{BGG, BGG2, conley, con2, con3, McGehee1969}.

Holomorphic curve methods were introduced in this context in \cite{AFKP2012}, where the regularized energy hypersurfaces were shown to be of contact type. Global surfaces of section for small mass ratios were obtained in \cite{AFFHvK}. Related constructions in spatial or symmetric settings appear in \cite{MvK, HLOSY}. Finite energy foliations in other Hamiltonian systems were constructed in \cite{dPS1, dPS2, dPKSS, Kim1, Kim2, Sa1, Sch}. Methods for prescribing binding orbits were developed in \cite{dPHKS} and applied, for instance, to the H\'enon--Heiles system \cite{dPKSS}.

We refer to \cite{FvK} for background on symplectic methods in the restricted three-body problem.

\hfill\newline
\noindent{\bf Acknowledgment.}
LL is partially supported by the National Natural Science Foundation of China (Grant number: 12401238), the Natural Science Foundation of Shandong Province, China (Grant number: ZR2024QA188). LL thanks the School of Mathematics at Shandong University for its support. PS thanks Umberto Hryniewicz for several conversations on the analysis of $J$-holomorphic curves in symplectizations. PS was partially supported by the National Natural Science Foundation of China (grant number W2431007). LL and PS thank the support of the Shenzhen International Center for Mathematics - SUSTech.

\section{Preliminaries}

\subsection{Reeb flows and pseudo-holomorphic curves}\label{subsec: pseudo-holo curves}

The pair $(S^3,\xi_0=\ker \alpha_0)$ is a contact manifold called the tight three-sphere, where $\alpha_0:=\frac{1}{4i}\sum_{j=1}^2(\bar z_j dz_j - z_jd\bar z_j)$ is the Liouville form restricted to $S^3$. For every smooth function $f: S^3 \to \R^+$,  $\alpha:= f \alpha_0$ is a contact form on $(S^3,\xi_0)$, and its Reeb vector field $R$  is determined by $d\alpha(R,\cdot) =0$ and $\alpha(R)=1$. The flow of $R$, denoted $\psi_t,t\in \R,$ preserves $\alpha$ and thus preserves $\xi_0$.  We say that a periodic orbit $P=(x,T)$ is nondegenerate if the linear map $D\psi_T(x(0)):\xi_0|_{x(0)}\to \xi_0|_{x(0)}$ does not have $1$ as an eigenvalue. We say that $\alpha$ is nondegenerate if all of its periodic orbits are nondegenerate. The iterates of $P$ are denoted $P^k=(x,kT),$ for every $k\in \N.$

Given relatively prime integers $p\geq q\geq 1,$ we consider the lens space $L(p,q)=S^3 / \Z_p$ as before. Since the contact form $\alpha_0$ on $S^3$ is $\Z_p$-invariant, it descends to a contact form on $L(p,q)$, also denoted $\alpha_0$. The contact structure $\xi_0:=\ker \alpha_0$ is called the universally tight contact structure on $L(p,q)$.

A knot $K\subset L(p,q)$ is called $p$-unknotted if there exists an immersion $u:\D \to L(p,q)$ so that $u|_{\D \setminus \partial \D}:\D \setminus \partial \D \to L(p,q) \setminus K$ is an embedding and $u|_{\partial \D}:\partial \D \to K$ is a $p$-covering map. The disk $u$ is called a $p$-disk for $K$. Let $K\subset L(p,q)$ be  $p$-unknotted and transverse to $\xi_0$. Take a $p$-disk $u$ for $K$ and a small non-vanishing section $Y$ of $u^*\xi_0$. Use $Y$ and an exponential map to push $K$ to a knot $K'$ that is disjoint from $K$, close to a $p$-cover of $K$, and transverse to $u$. The (rational) self-linking number of $K$ is defined as the normalized algebraic intersection number between $K'$ and $u$, i.e.
$\sl(K):=\frac{1}{p^2} K' \cdot u.
$ Here, $K$ is oriented by $\alpha_0$, $K'$ inherits the orientation of $K$, $u$ is oriented by $K$ and $L(p,q)$ is oriented by $\alpha_0 \wedge d\alpha_0>0$. As an example, the knot $K:=\pi_{p,q}(S^1 \times 0)$ is $p$-unknotted and transverse to $\xi_0$. In this case, a $p$-disk for $K$ is given by $\pi_{p,q} \circ u$, where $u(z)=(z,\sqrt{1-|z|^2})\in S^3, \forall z\in \D$. One readily checks that a knot $K'$ as above satisfies $K' \cdot u=-p,$  and thus  $\sl(K)= -\frac{1}{p}$.

Let $B_i \subset L(p,q),i=1,\ldots,l,$ be mutually disjoint regular open three-balls, and let $\M:= L(p,q) \setminus \cup_{i=1}^l B_i.$ Then $\partial \M$ is the union of $l$ regular two-spheres  $\S_i := \partial B_i$. The restriction of $\alpha_0$ and $\xi_0$ to $\M$ is still denoted $\alpha_0$ and $\xi_0$, respectively. We only require $\partial \M$ to be  $C^1$.

Let $J$ be an almost complex structure on $\R \times \M$ so that $J \cdot \partial_a = R$ and $J(\xi_0)=\xi_0,$ where $d\alpha(\cdot, J, \cdot)$ is an inner product on $\xi_0$. Here, $a$ is the $\R$-coordinate, and $R$ and $\xi$ are regarded as $\R$-invariant objects on $\R \times \M$. The space of such $J$'s is denoted by $\J(\alpha)$. Let $(\Sigma,j)$ be a connected Riemann surface (possibly with boundary), and let $\Gamma \subset \Sigma \setminus \partial \Sigma$ be a finite set. Let $\dot \Sigma := \Sigma \setminus \Gamma,$ and let $J\in \J(\alpha)$. A map $\tilde u=(a,u):\dot \Sigma \to \R \times \M$ is called a finite energy $J$-holomorphic curve if it satisfies the non-linear Cauchy-Riemann equation
$
\bar \partial_J \tilde u = d\tilde u + J(\tilde u) \circ d\tilde u \circ j =0,
$ and has  finite Hofer energy
$
0< E(\tilde u):=\sup_{\phi\in \mathcal{T}}\int_{\dot \Sigma} \tilde u^*d(\phi(a)\alpha)<+\infty,
$
where $\mathcal{T}:=\{\phi:\R \to [0,1],\ \phi'\geq 0\}$.

We consider totally real boundary conditions. Let $\gamma_1,\ldots,\gamma_{m_0}$ be the components of $\partial \Sigma$. We assume that for each $j=1,\ldots,m_0$, there exists a totally real surface $L_j \subset \{0\} \times \M$, that is $TL_j\oplus JTL_j=T(\R\times \M)$ along $L_j$, so that $\tilde u(\gamma_j) \subset L_j$. We denote by $l_j \to \gamma_j$ the line bundle of $u^*\xi_0|_{\gamma_j}$ given by  $l_j(z) =\xi_0|_{u(z)} \cap T_{u(z)}L_j$ for every $z\in \gamma_j$. Later we shall consider the particular case where $L \subset \{0\} \times (\M\setminus \partial \M)$ is a totally real surface that transversely intersects the contact structure and   $\tilde u=(a,u):\D \to \R \times \M$ is a $J$-holomorphic disk satisfying $\tilde u(\partial \D) \subset L$ and $\partial_{\vec n} a>0$, where $\vec n$ is the outer normal vector along $\partial \D$.

Each non-removable puncture $z_0\in\Gamma$ of $\tilde u$ has a sign $\epsilon(z) \in \{-1,+1\}$ so that $a(z) \to \epsilon(z_0) \infty$ as $z\to z_0$. Furthermore, for suitable polar coordinates $s+it\in [0,+\infty) \times \R / \Z$ on a punctured neighborhood of $z_0$ the following holds: given a sequence $s_n \to +\infty$, there exists a subsequence also denoted $s_n$ and a periodic orbit $P=(x,T)$ so that $u(s_n,\cdot) \to x(\epsilon(z_0)T\cdot)$ in $C^\infty$ uniformly in $t$ as $s \to +\infty$. The periodic orbit $P$ is called an asymptotic limit of $\tilde u$ at $z_0$, and if $P$ is nondegenerate, then $P$ is the unique asymptotic limit of $\tilde u$ at $z_0$, and much can be said about the asymptotic behavior of $\tilde u$ as it approaches $P$. If $\epsilon(z_0)=+1$, we say that $z_0$ is a positive puncture. Otherwise, we say $z_0$ is a negative puncture. The signs of the punctures induce the splitting $\Gamma = \Gamma^+ \cup \Gamma^-.$

\subsection{The asymptotic operator and the Conley-Zehnder index}\label{subsec: asym operator}
Let $P=(x,T)$ be a periodic orbit of $\alpha$ and let $x_T:=x(T\cdot):\R / \Z \to \M$. Let $J\in \J(\alpha)$. The unbounded self-adjoint operator
$A_P:W^{1,2}(\R / \Z, x_T^*\xi_0) \to L^2(\R / \Z, x_T^*\xi_0)$, defined by
$
A_P \cdot \eta : = -J \cdot \mathcal{L}_{\dot x_T} \eta,
$
is called the asymptotic operator of $P$. Here, $\mathcal{L}_{\dot x_T} \eta$ is the Lie derivative of $\eta$ along $x_T$. The spectrum $\sigma(A_P)$ of $A_P$ consists of countably many real eigenvalues accumulating precisely at $\pm \infty$. An eigenvector $e: \R / \Z \to \R^2$ of $\lambda \in \sigma(A_P)$ is smooth and never vanishes. Hence, for a fixed trivialization $\Psi:x_T^*\xi_0 \to \R / \Z \times \R^2$,  $e$ has a well-defined winding number $\wind(\lambda)$, depending only on $\lambda$ and the homotopy class of $\Psi$. We omit the dependence on $\Psi$ in the notation. The function $\sigma(A_P) \ni \lambda \mapsto \wind(\lambda)\in \Z$ is monotone increasing, and given $k\in \Z$, there exist precisely $2$ eigenvalues (counting multiplicities) of $A_P$  with winding number $k$.  It can be directly checked that $P$ is nondegenerate if and only if $0\notin  \sigma(A_P).$  Fix $\delta\in \R$, and let
$$
\begin{aligned}
    \wind^{<\delta}(A_P)& := \max \{\wind(\lambda): \sigma(A_P)\ni\lambda<\delta \},\\ \wind^{\geq \delta}(A_P)& := \min \{\wind(\lambda): \sigma(A_P) \ni  \lambda \geq \delta\}.
\end{aligned}
$$
The weighted Conley-Zehnder index of $P$ is defined as
$$
\mu^\delta(P):=\wind^{<\delta}(A_P) + \wind^{\geq \delta}(A_P).
$$
The weighted index $\mu^\delta(P)$ depends  on $J$ and the homotopy class of $\Psi$. If $\delta=0$, then  $\mu(P):=\mu^0(P)$ depends only on the homotopy class of $\Psi.$ Since the parity of $\mu(P)$ does not depend on $\Psi$, there exists a natural splitting $\Gamma = \Gamma_{\text{even}} \cup \Gamma_{\text{odd}}.$

Having a geometric definition of $\mu(P)$ is convenient. In the frame induced by $\Psi$, the linearized flow along $x(t)$ determines a path of $2\times 2$ symplectic matrices $t\mapsto \Phi(t),t\in [0,T]$. Given $v_0\in \R^2\setminus \{0\}$, let $\theta(t),t\in [0,T],$ be a continuous argument of $\Phi(t) \cdot v_0$, and let $\Delta(v_0):=(\theta(T) - \theta(0))/2\pi$. Then $I_P:=\{\Delta(v_0),0\neq v_0\in \R^2\}$ is an interval of length $<1/2,$ and there exists $k\in \Z$ such that for every $\epsilon>0$ sufficiently small, either $k\in I_P - \epsilon$ or $I_P-\epsilon \subset (k,k+1)$. We then have, respectively, $\mu(P) = 2k$ or $\mu(P) = 2k +1$.  Finally, the rotation number of $P$ in the frame induced by $\Psi$ is defined as
$
\rho(P) := \lim_{k \to \infty} \frac{1}{2k}\mu(P^k),
$
where $P^k$ is the $k$-th iterate of $P$. It is immediate to check that if $\mu(P) =2$, then $\rho(P)=1$. Moreover,  $\mu(P) \geq 3$ if and only if $\rho(P)>1$.

\subsection{\texorpdfstring{Asymptotics of $J$-holomorphic curves}{Asymptotics of J-holomorphic curves}}
Let $z_0\in \Gamma$ be a puncture of a $J$-holomorphic curve $\tilde u=(a,u):  \Sigma\setminus \Gamma \to \R \times M$ and let $\epsilon(z_0)\in \{-1,+1\}$ be the sign of $z_0$. Assume that $P=(x, T)$ is an asymptotic limit of $\tilde u$ at $z_0$.  Denote by $P_0=(x_0,T_0)$ the simple periodic orbit so that $P= P_0^k$ for some integer $k\geq 1$. Let
$(\vartheta,x,y)\in \R / \Z \times B_\delta(0)$ be coordinates on a small tubular neighborhood $\U\subset M$ of $P_0$, so that $\alpha = f (d\vartheta + x dy),$ for a function $f=f(\vartheta,x,y)$ satisfying $f(\vartheta,0,0)=T_0$ and $df(\vartheta,0,0)=0$ for every $\vartheta$. Here, $B_\delta(0)\subset \R^2$ is an open disk of radius $\delta>0$ centered at $0$. Let $z:=(x,y)$. Consider polar coordinates $s+it \in [0,\infty) \times \R / \Z\to e^{-2\pi(s+i t)}\in \D \setminus \{0\}$ on a punctured neighborhood of $z_0\equiv 0$ and write $\tilde u(s,t)=(a(s,t),\vartheta(s,t),z(s,t)),$ whenever defined. Let $A_P$ be the asymptotic operator at $P$.

\begin{thm}[Hofer-Wysocki-Zehnder \cite{HWZ1996}, Siefring \cite{Si1}] \label{thm: half cylinder 1} Assume that $z_0\in \Gamma$ is a positive puncture and that $P=P_0^k$ is nondegenerate. Then $(\vartheta(s,t),z(s,t))\in \R / \Z \times B_\delta(0)$ for every $s$ sufficiently large, and there exist $a_0,\vartheta_0\in \R,$ a $\lambda$-eigenvector $e(t)$ of $A_P$, with $\lambda<0$, so that
\begin{equation}\label{eq:asymptotic_formula}
z(s,t)=e^{\lambda s}(e(t)+r(s,t)),\\
\end{equation}
where
$|r(s,t)| \to 0$ as $s\to +\infty$ uniformly in $t$.
Moreover,
    $|a(s,t) - (a_0+Ts)|\to 0$ and $|\vartheta(s,t) - (kt+\vartheta_0)| \to 0$ as $s\to +\infty$ uniformly in $t$, where $\vartheta$ is lifted to a real-valued function.
\end{thm}

In Theorem \ref{thm: half cylinder 1}, $\lambda$ and $e$ are called the leading eigenvalue and the leading eigenvector of $\tilde u$ at $z_0\in \Gamma$, respectively. The leading eigenvector is determined up to a positive multiple. If $z_0$ is a negative puncture, then we consider polar coordinates $(s,t)\in (-\infty,0]\times \R / \Z\to e^{2\pi(s+it)}\in \D \setminus \{0\}$ on a punctured neighborhood of $z_0\equiv 0$ and the asymptotic formula \eqref{eq:asymptotic_formula} still holds for some positive leading eigenvalue $\lambda$ and a leading eigenvector $e$. The asymptotic properties of $a(s,t)\to -\infty$ and $\vartheta(s,t)$ as $s\to -\infty$ are similar to the case of a positive puncture.

We call a puncture $z_0\in \Gamma$  nondegenerate if $\tilde u$ has an asymptotic formula at $z_0$ as in Theorem \ref{thm: half cylinder 1} for a non-vanishing leading eigenvalue $\lambda$  and a leading eigenvector $e$. With this definition, the puncture $z_0$ may be nondegenerate even if the asymptotic limit $P$ is degenerate. 

\begin{defi}Let $\tilde u_i:\Sigma \setminus \Gamma_i \to \R \times M$, $i=1,2,$ be a pair of finite energy $J$-holomorphic curves asymptotic to the same nondegenerate periodic orbit $P=(x,T)$ at $z_i\in \Gamma_i$. Let $\lambda_i$ and $e_i$ be the leading eigenvalue and leading eigenvector of $\tilde u_i$ at $z_i$, respectively. We say that $\tilde u_1$ and $\tilde u_2$ approach $P$ through the same direction if $\lambda_1 = \lambda_2$ and $e_1 = c e_2$ for some constant $c>0$. If $e_1 = -ce_2$ for some $c>0$ then we say that $\tilde u_1$ and $\tilde u_2$ approach $P$ through opposite directions.
\end{defi}

\subsection{\texorpdfstring{Uniqueness of $J$-holomorphic planes and cylinders}{Uniqueness of J-holomorphic planes and cylinders}} The  following uniqueness results on $J$-holomorphic curves follows from Siefring's intersection theory \cite{Si1, Si2}.

\begin{thm}\label{thm_uniqueness} The following uniqueness statements hold:
\begin{itemize}
   \item[(i)] If $\tilde u=(a,u):\C \to \R \times \mathcal{M}$ is an embedded $J$-holomorphic plane asymptotic to $P_{2,i}$, then up to parametrization and $\R$-translation, $\tilde u$ coincides with one of the $J$-holomorphic planes projecting to the hemispheres of $\S_i \setminus P_{2,i}$.
\item[(ii)] If $\tilde v_1=(b,v), \tilde v_2 = (b_2,v_2):\R \times \R / \Z \to \R \times (\mathcal{M}\setminus \partial \M)$ are embedded $J$-holomorphic cylinders, not intersecting $\R \times P_3$, with a positive end at $P_3^p$ whose leading eigenvalue has winding number $1$, and a negative end at $P_{2,i}$, then up to parametrization and $\R$-translation, $\tilde v_1$ coincides with $\tilde v_2$.
\end{itemize}
\end{thm}

\begin{proof} The proof of (i) is essentially contained in \cite[Proposition C-3]{dPS1}. However, the proof of (ii) needs to be adapted to the current situation since the index of $P_3^p$ might be greater than $3$. In that case, the leading eigenvalue of $\tilde v_i$ at the positive puncture has winding number $1$ and thus does not coincide with the winding of the largest negative eigenvalue of $A_{P_3^p}$.

Let $\tilde v_1, \tilde v_2$ be as in Theorem (ii). Siefring \cite{Si2} introduced the generalized intersection number $[\tilde v_1] * [\tilde v_2],$ which is invariant under homotopies that keep the same asymptotic limits. This number counts actual intersections between the curves as well as intersections at infinity, i.e., intersections related to the respective punctures of $\tilde v_1$ and $\tilde v_2$ whose asymptotic limits are covers of the same simple periodic orbit. In particular, this number includes hidden intersections at the punctures, which correspond to tangencies at infinity or appear only when the curves are suitably perturbed. Since $\tilde v_1$ and $\tilde v_2$ have the same asymptotic limits at their respective positive and negative punctures, and since these curves do not intersect the trivial cylinders $\R \times P_3$ and $\R \times P_{2, i}$, Theorem 5.8 from \cite{Si2} gives
$$
[\tilde v_1]*[\tilde v_2] = pd_0^+ + d_0^-,
$$
where $d_0^+\geq 0$ is the difference between the winding number $\alpha_+:=d+1\geq 1$ of the largest negative eigenvalue of $A_{P_3^p}$ and the winding number of the leading eigenvalue at the positive puncture of $\tilde v_1$, which is equal to $1$. In the same way, $d_0^-\geq 0$ is the difference between the winding number of the leading eigenvalue at the negative puncture, which is equal to $1$, and the winding number $\alpha_-$ of the smallest positive eigenvalue of $A_{P_3^p}$, also equal to $1$. Indeed, at the negative puncture, if the leading eigenvalue does not have winding number $1$, then its projection to $\mathcal{M}$ must wind around $P_{2,i}$ with winding number $\geq 2$, forcing intersections with $\partial \M$, a contradiction. Hence we obtain $d_0^+ = d$ and $d_0^-=0,$ which implies that
\begin{equation}\label{gen_int_number}
[\tilde v_1]*[\tilde v_2] = pd.
\end{equation}
This value accounts for the hidden intersections at $P_3^p$ that arise from the positive puncture after suitably perturbing $\tilde v_1$ and $\tilde v_2$. Since the asymptotic limit is the $p$-cover of a simple periodic orbit, the number of these potential intersections is a multiple of $p$.

The generalized intersection number $[\tilde v_1] *[\tilde v_2]$ can be better expressed once we have more information on the relative asymptotic behavior between $\tilde v_1$ and $\tilde v_2$ at both positive and negative punctures. Assume by contradiction that $\tilde v_1$ does not coincide with $\tilde v_2$, i.e., $\tilde v_1$ is not obtained from $\tilde v_2$ by reparametrization and $\R$-translation. In particular, $v_1(\R \times \R / \Z) \neq v_2(\R \times \R / \Z)$ and one may consider the non-trivial difference between these curves near the ends. It is then proved in \cite[Theorem 4.4]{Si2} that
\begin{equation}\label{generalized}
[\tilde v_1 ]*[\tilde v_2] = \text{int}(\tilde v_1, \tilde v_2) + \delta_\infty(\tilde v_1,\tilde v_2),
\end{equation}
where $\text{int}(\tilde v_1,\tilde v_2)\geq 0$ is the algebraic intersection number of $\tilde v_1$ and $\tilde v_2$ and $\delta_\infty(\tilde v_1, \tilde v_2)\geq 0$ is the total asymptotic intersection index of $\tilde v_1$ and $\tilde v_2$, i.e.,
$$
\delta_\infty(\tilde v_1, \tilde v_2) = \delta_\infty^+(\tilde v_1,\tilde v_2) + \delta_\infty^-(\tilde v_1,\tilde v_2).$$
Here, $\delta_\infty^+(\tilde v_1, \tilde v_2)$ and $\delta_\infty^-(\tilde v_1, \tilde v_2)$ are the asymptotic intersection indices of $\tilde v_1$ and $\tilde v_2$ at the positive
 and negative punctures, respectively, given by
 $$
 \begin{aligned}
 \delta_\infty^+(\tilde v_1, \tilde v_2) & = i_\infty^+(\tilde v_1, \tilde v_2) +p\alpha_+,\\
 \delta_\infty^-(\tilde v_1, \tilde v_2) & = i_\infty^-(\tilde v_1, \tilde v_2) -\alpha_-,
 \end{aligned}
 $$
 see Lemma 3.20 and equation (3-32) in \cite{Si2}. Notice that the sign convention in \cite{Si2} is slightly different from the one in this paper. Indeed, in \cite{Si2}, the winding number at a negative puncture has an opposite sign. Here, $$i_\infty^\pm(\tilde v_1, \tilde v_2) =\mp  \wind_{\text{rel}}^\pm(\tilde v_1, \tilde v_2)$$ is the adjusted winding number between the difference of $\tilde v_1$ and $\tilde v_2$ on suitable coordinates, see \cite{Si1}. In the case of positive punctures, where the curves $p$-cover $P_3$, since the leading eigenvectors of both curves have winding number $1$ and the space of such eigenvectors is two-dimensional, the winding number of the difference between $\tilde v_1$ and $\tilde v_2$ in suitable coordinates is $\leq 1$. The fact that the asymptotic limit is the $p$-cover of $P_3$, this number is multiplied by $p$. Therefore, $$i_\infty^+(\tilde v_1, \tilde v_2) = -\wind_{\text{rel}}^+(\tilde v_1, \tilde v_2) \geq -p,$$ see Corollary 3.21 in \cite{Si2}. This implies that
 $$
 \delta_\infty^+(\tilde v_1, \tilde v_2) \geq -p + p(d+1) =pd.
 $$
 Now, since both leading eigenvalues of $\tilde v_1$ and $\tilde v_2$ at their negative punctures coincide and have winding number $1$, and since the space of such eigenvectors is one-dimensional, we conclude from Siefring's formula \cite{Si1} for the difference between $\tilde v_1$ and $\tilde v_2$ near the negative punctures that
 $$i_\infty^-(\tilde v_1, \tilde v_2) = \wind_{\text{rel}}^-(\tilde v_1, \tilde v_2) \geq 2.$$ This implies that
 $$
 \delta_\infty^-(\tilde v_1, \tilde v_2) \geq 2 - \alpha_-=2-1=1.
 $$
 The above estimates for $\delta_\infty^+(\tilde v_1, \tilde v_2)$ and $\delta_\infty^-(\tilde v_1, \tilde v_2)$ give
 $$
 \delta_\infty(\tilde v_1, \tilde v_2) = \delta_\infty^+(\tilde v_1, \tilde v_2) + \delta_\infty^- (\tilde v_1, \tilde v_2) \geq pd + 1,
$$
and thus \eqref{generalized} implies $[\tilde v_1] *[\tilde v_2] \geq pd +1,$ contradicting \eqref{gen_int_number}. We conclude that $\tilde v_1$ coincides with $\tilde v_2$ up to reparametrization and $\R$-translation and this finishes the proof of Theorem \ref{thm_uniqueness}-(ii).
\end{proof}

\subsection{Automatic Transversality} Fix a compact connected Riemann surface $(\Sigma,j)$ possibly with non-empty boundary, and let $\Gamma \subset \Sigma \setminus \partial \Sigma$ be a finite set. Assume that all punctures of a $J$-holomorphic curve $\tilde u=(a,u):\Sigma\setminus \Gamma \to \R \times \M$ are nondegenerate and denote by $P_z$  the asymptotic limit of $\tilde u$ at $z\in \Gamma$. Assume that the $d\alpha$-area of $\tilde u$ is positive. Fix a symplectic trivialization $\Psi$ of $u^*\xi$, which induces a homotopy class of symplectic trivializations of $\xi$ along the asymptotic limits $P_z, z\in \Gamma$. Denote by $\mu(P_z),z\in \Gamma,$ the index of $P_z$ induced by $\Psi$. Let $\lambda_z$ be the leading eigenvalue of $\tilde u$ at $z$. Recall that $\lambda_z<0$ if $z$ is a positive puncture, and $\lambda_z>0$, otherwise. Let $\delta$ be a collection of real numbers $\delta_z,z\in \Gamma$, called weights, so that $\lambda_z < \delta_z\leq 0$ if $z$ is a positive puncture  and $0\leq \delta_z < \lambda_z$, otherwise. In the following, we shall assume that  $\delta_z$ is not an eigenvalue of $A_{P_z}$. Later, we also assume that no eigenvalue of $A_{P_z}$ exists between $\lambda_z$ and $\delta_z$. Assume that each boundary component $\gamma_j \subset \partial \Sigma$ is mapped under $\tilde u$ into a totally real surface $L_j\subset \{0\} \times M$ and let $l_j\to \gamma_j$ be the line bundle of $u^*\xi|_{\gamma_j}$ given by $l_j(z) = \xi|_{u(z)} \cap T_{u(z)} L_j, \forall z\in \gamma_j$ and $j=1,\cdots,m_0$. The $\delta$-weighted Conley-Zehnder index and the $\delta$-weighted Fredholm index of $\tilde u$ are defined as
\begin{equation}\label{eq_index}
\begin{aligned}
\mu^\delta(\tilde u) & : = \sum_{j=1}^{m_0} \mu(u^*\xi|_{\gamma_j}, l_j) + \sum_{z\in \Gamma^+} \mu^{\delta_z}(P_z) - \sum_{z\in \Gamma^-} \mu^{\delta_z}(P_z),\\
\ind^\delta (\tilde u) & := \mu^\delta(\tilde u) - \chi(\dot \Sigma),
\end{aligned}
\end{equation}
where $\chi(\dot \Sigma)=\chi(\Sigma) - \#\Gamma$ is the Euler characteristic of $\dot \Sigma.$ The first term in the definition of $\mu^\delta(\tilde u)$ consists of Maslov indices of $l_j\subset u^*\xi|_{\gamma_j}$  in the frame  $\Psi,$ and $m_0$ is the number of components of $\partial \Sigma$.
If $\delta_z =0, \forall z\in \Gamma,$ then $\ind^\delta(\tilde u)$ is denoted by $\ind(\tilde u)$ and called the index of $\tilde u$. Since $\mu^{\delta_z}(P_z) \leq \mu(P_z), \forall z\in \Gamma^+,$ and $\mu^{\delta_z}(P_z)\geq \mu(P_z), \forall z\in \Gamma^-$, the inequalities $\mu^\delta(\tilde u) \leq \mu(\tilde u)$ and $\ind^\delta(\tilde u) \leq \ind(\tilde u)$ always hold.

Recall that the $d\lambda$-area $\int_{\dot \Sigma} u^*d\lambda$ of $\tilde u$ is always non-negative and vanishes if and only if $\pi \circ du \equiv 0$, where $\pi: TM \to \xi$ is the projection along the Reeb vector field. We keep assuming that
$
\int_{\dot \Sigma}u^* d\lambda >0.
$
In this case, the leading eigenvalue and leading eigenvector are well-defined at each puncture. Since $\dot \Sigma$ is connected, Theorem \ref{thm: half cylinder 1} implies that $\pi \circ du$ does not vanish near the punctures. Since $\pi \circ du$ satisfies a Cauchy-Riemann equation, each zero is isolated and has a positive local degree. Hence, $\pi \circ du$ has finitely many zeros, and the sum of their local degree is denoted by $\wind_\pi(\tilde u).$ Each zero of $\pi \circ du$  lying in  $\partial \Sigma$  contributes with half of its local degree. Hence $\wind_\pi(\tilde u)$ is a half-integer. According to Theorem \ref{thm: half cylinder 1}, each puncture $z\in\Gamma$ admits a leading eigenvalue and a leading eigenvector. The winding number of the leading eigenvector in the frame $\Psi$ is denoted by $\mathrm{wind}_\infty(z)$. It is proved in \cite{HWZII} that $\mathrm{wind}_\pi(\tilde u)=\sum_{z\in\Gamma^+}\mathrm{wind}_\infty(z)-\sum_{z\in\Gamma^-}\mathrm{wind}_\infty(z)-\chi(\dot \Sigma)$ provided $\partial \Sigma=\emptyset$. The theorem below follows from Wendl's results in \cite{Wendl08b} and the definitions above.

\begin{thm}[Wendl \cite{Wendl08b}]\label{thm_Wendl_windpi}The following inequalities hold
\begin{equation}\label{wind_pi}
0\leq 2\wind_\pi(\tilde u) \leq \ind^\delta(\tilde u) - 2 + 2g + \#\Gamma_{\mathrm{even}}^\delta+m_0,
\end{equation}
where $g$ is the genus of $\Sigma$ and  $\#\Gamma_{\mathrm{even}}^\delta$ is the number of punctures whose asymptotic limit has an even $\delta$-weighted index.
\end{thm}

Assume that the $J$-holomorphic curve $\tilde u=(a,u):\dot \Sigma \to \R \times \M$ is  embedded. Then $\tilde u^* T(\R \times \M)$ splits as $T_{\tilde u} \oplus N_{\tilde u}$, where the fiber of $T_{\tilde u}$ at $z\in \dot \Sigma$ is $d\tilde u(T_z\dot \Sigma)$ and $N_{\tilde u}$ is a complex line bundle complementary to $T_{\tilde u}$.  From the asymptotic behavior of $\tilde u$, see Theorem \ref{thm: half cylinder 1}, we may assume that $N_{\tilde u}$ coincides with $\tilde u^*\xi$ near the punctures.  We also assume that  $N_{\tilde u}$ coincides with $u^*\xi|_{\partial \Sigma}$ along $\partial \Sigma$.

The $\delta$-weighted normal first Chern number $c_N^\delta(\tilde u)$ of $\tilde u$ is  defined as the half-integer
\begin{equation}\label{c_N}
2c_N^\delta(\tilde u) := \ind^\delta(\tilde u) -2 + 2g +\#\Gamma_{\text{even}}^\delta + m_0,
\end{equation}
where $g$ is the genus of $\Sigma$ and $\Gamma_{\text{even}}^\delta$ is the number of punctures whose asymptotic limit has an even $\delta$-weighted index. Notice that \eqref{wind_pi} and \eqref{c_N} imply
\begin{equation}\label{wind_pi_c_N}
0\leq \wind_\pi(\tilde u) \leq c_N^\delta(\tilde u).
\end{equation}
If $\delta_z=0, \forall z\in \Gamma,$ then $c_N^\delta(\tilde u)$ is simply denoted $c_N(\tilde u)$. If $\delta_z$ is not specified for some puncture $z$, then we tacitly assume that $\delta_z=0$.  If $m_0=0$, then $c_N^\delta(\tilde u)$ is an integer, since $\ind^\delta(\tilde u)$ and $\#\Gamma_{\text{even}}^\delta$ have the same parity, see \eqref{eq_index}. The half-integer $c_N^\delta(\tilde u)$ can be regarded as a bound on the algebraic number of zeros of a section $\sigma:\dot \Sigma \to N_{\tilde u}$  representing infinitesimal $J$-holomorphic variations of $\tilde u$ in the normal direction, keeping the same asymptotic limits at the punctures, the same boundary conditions,  and respecting the weight constraints at the punctures. The zeros on the boundary $\partial \Sigma$ contribute half of their local degree. The section $\sigma$ satisfies  $D^N\bar \partial_J (\tilde u) \cdot \sigma =0$, where $D^N\bar \partial_J \tilde u$ is the restriction and projection to the normal bundle $N_{\tilde u}$ of the linearized Fredholm operator  $D\bar \partial_J (\tilde u)$ between suitable weighted Sobolev spaces. The curve $\tilde u$ is said to be {\em regular} if $D\bar \partial_J (\tilde u)$ is surjective. Any such a section of $N_{\tilde u}$ admits an asymptotic formula similar to the one in Theorem \ref{thm: half cylinder 1} so that at each puncture $z\in \Gamma$, $\sigma$ has a leading eigenvalue and a leading eigenvector of the asymptotic operator $A_{P_z}$.

\begin{thm}[Wendl \cite{Wendl10a}]\label{thm_aut_transverse} Assume that $\tilde u$ is embedded. If $\ind^\delta(\tilde u) > c_N^\delta(\tilde u),$ then $\tilde u$ is regular. In particular, the space of $J$-holomorphic curves near $\tilde u$, with the same asymptotic limits, boundary conditions, and $\delta$-weight constraints, has the structure of a smooth manifold with dimension $\ind^\delta(\tilde u)$.
\end{thm}

Theorem \ref{thm_aut_transverse} can be regarded as follows. Suppose by contradiction that an embedded curve $\tilde u$ with $\ind^\delta(\tilde u) > c_N^\delta(\tilde u)$ is not regular. Then the linearized Cauchy-Riemann operator in the normal direction has a non-trivial cokernel, and thus the dimension of the kernel is $d \geq \ind^\delta(\tilde u)+1 > c_N^\delta(\tilde u)+1$. Then it is possible to construct a non-trivial section with more than $c_N^\delta(\tilde u)$ zeros, a contradiction.

The curves satisfying the conditions of Theorem \ref{thm_aut_transverse} are called automatically transverse. The following proposition is a direct consequence of Wendl's results in \cite{ Wendl08, Wendl08b, Wendl10a, Wendl10}.

\begin{prop}\label{prop_automatic_transversality}The following assertions hold:
\begin{itemize}
    \item[(i)] If $\tilde u:\D \to \R \times \M$ is an embedded $J$-holomorphic disk so that $\mu(u^*\xi|_{\partial \D},l_1)=2,$ then $\tilde u$ is regular, $\ind(\tilde u)=1$, $\wind_\pi(\tilde u)=0$ and $c_N(\tilde u)=0.$

    \item[(ii)] If $\tilde u:\D \setminus \{0\}\to \R \times \M$ is an embedded $J$-holomorphic punctured disk so that $\mu(u^*\xi|_{\partial \D},l_1)=2,$ and $0$ is a negative puncture whose asymptotic limit is an index-$2$ hyperbolic orbit, then $\ind(\tilde u)=0,$ $\wind_\pi(\tilde u)=0$ and $c_N(\tilde u) =0.$

    \item[(iii)] If $\tilde u:\C \to \R \times \M$ is an embedded $J$-holomorphic plane whose asymptotic limit $P_3$ at $\infty$ has $\delta$-weighted  index $\mu^\delta(P_3)=3$, then $\tilde u$ is regular, $\ind^\delta(\tilde u)=2$, $\wind_\pi(\tilde u)=0$ and $c_N^\delta(\tilde u)=0$.

    \item[(iv)] If $\tilde u:\C \setminus \{0\}\to \R \times \M$ is an embedded $J$-holomorphic cylinder whose asymptotic limit $P_3$ at the positive puncture $\infty$ has $\delta$-weighted index $\mu^\delta(P_3)=3$, and $0$ is a negative puncture whose asymptotic limit is an index-$2$ hyperbolic orbit, then $\tilde u$ is regular, $\ind^\delta(\tilde u) =1$, $\wind_\pi(\tilde u) =0$ and $c_N^\delta(\tilde u)=0$.
\end{itemize}
\end{prop}
\begin{proof}
We compute in each case
\begin{itemize}
    \item[(i)]  $\ind(\tilde u) = 2-1=1,m_0=1$ and $\#\Gamma_{\text{even}}=0.$
     \item[(ii)] $\ind(\tilde u) = 2 - 2 -0 = 0,m_0=1,$ $\#\Gamma_{\text{even}}=1.$
     \item[(iii)] $\ind^\delta(\tilde u)=3-1=2, m_0=0$ and $\#\Gamma_{\text{even}}^\delta=0$.
     \item[(iv)] $\ind^\delta(\tilde u) =3-2=1,$ $m_0=0$ and $\#\Gamma_{\text{even}}^\delta=1.$
\end{itemize}
From \eqref{c_N}, we obtain $c_N^\delta(\tilde u)=c_N(\tilde u)=0$ in all cases. A direct application of inequality \eqref{wind_pi_c_N} gives $\wind_\pi(\tilde u)=0$. In cases (i), (iii), and (iv), we see that $c_N^\delta(\tilde u) < \ind^\delta(\tilde u)$. By Theorem~\ref{thm_aut_transverse}, $\tilde u$ is regular in those cases.
\end{proof}

\begin{rem}  The curve $\tilde u:\D \setminus \{0\} \to \R \times \M$ in Proposition \ref{prop_automatic_transversality}-(ii) is not automatically transverse. Theorem 1 in \cite{Wendl10a} implies that $\dim \ker D \bar \partial_J (\tilde u) \leq 1$ and it is not difficult to construct examples so that $\dim \ker D \bar \partial_J(\tilde u) =1.$ In that case, $\dim {\rm coker} D \bar \partial_J(\tilde u) = 1$ and thus $\tilde u$ is not regular. 
\end{rem}

\section{Proof of Theorem \ref{main1}}

Let $\mathcal{M}$, $\alpha$, $\P=\{P_3,P_{2,1},\ldots,P_{2,l}\}\subset \P(\alpha)$ and $J\in \J(\alpha)$ be as in Theorem \ref{main1}. Recall that each $P_{2,i}$ lies in the sphere-like boundary component $\S_i\subset \partial \mathcal{M}$ and the hemispheres of $\S_i \setminus P_{2,i}$ are projections of a pair of nicely embedded $J$-holomorphic planes $\tilde u_{i,1}=(a_{i,1},u_{i,1})$ and $\tilde u_{i,2}=(a_{i,2},u_{i,2})$ asymptotic to $P_{2,i}$ and satisfying $u_{i,1}(\C) \cup u_{i,2}(\C) = \S_i \setminus P_{2,i}.$

The planes $\tilde u_{i,j}, j=1,\ldots,l,$ prevent families of $J$-holomorphic curves in $\mathcal{M}$ from escaping through $\partial \mathcal{M}$. They form a barrier so that the bubbling-off analysis is similar to that of a boundaryless contact three-manifold. For this reason, we call $\S_j$ a {\bf spherical shield}.

\begin{prop}\label{prop_inside}
    Let $\tilde u =(a,u):\dot \Sigma \to \R \times \mathcal{M}$ be a connected $J$-holomorphic curve with boundary conditions in $\mathcal{M} \setminus \partial \mathcal{M}$, and whose image is not contained in $\R \times \partial \mathcal{M}$. Then
    \begin{itemize}
        \item[(i)] $\tilde u$ does not have a positive end at any orbit in $\partial \M$.

        \item[(ii)]  $u(\dot \Sigma) \subset \mathcal{M} \setminus \partial \mathcal{M}$.
    \end{itemize}
 \end{prop}

 \begin{proof}
By contradiction, assume that $\tilde u=(a,u)$ has a positive end at $P_{2, i}^p$ for some $i\in \{1,\ldots,l\}$ and $p>0$. Consider the asymptotic operator $A_{P_{2,i}^p}$ associated with the positive hyperbolic orbit $P_{2,i}^p$. The leading eigenvalue of $\tilde u$ at the corresponding puncture must be the maximum negative eigenvalue $p\lambda_2<0$, where $\lambda_2$ is the maximum negative eigenvalue of $A_{P_{2,i}}$. In fact, otherwise the curve $u$ winds around $P_{2,i}$ slower that the hemispheres of $\mathcal{S}_i \setminus P_{2,i}$, forcing the existence of points in $u(\dot \Sigma)$ outside $\mathcal{M}$, a contradiction. The eigenspace associated with $p\lambda_2$ is one-dimensional and corresponds to the $p$-cover of the $\lambda_2$-eigenspace $\R e_2$ of $A_{P_{2,i}}$. In particular,  $e_2^p$ is tangent to $\mathcal{S}_i$ along $P_{2,i}^p$. Consider the holomorphic plane $\tilde v$ that $p$-covers the hemisphere of $\mathcal{S}_i\setminus P_{2,i}$ and approaches $P_{2,i}^p$ through the same direction as $\tilde u$. After reparametrizing $\tilde u$ if necessary, both $\tilde u$ and $\tilde v$ have the same leading eigenvalues and eigenvectors. The maps $u$ and $v$ do not have the same image close to the corresponding punctures by assumption. Hence, Siefring's formula for the difference of $J$-holomorphic curves asymptotic to the same periodic orbit, see Theorem 2.2 in \cite{Si1}, implies that the difference between $u$ and $v$ in suitable coordinates is ruled by a leading eigenvalue $\lambda_d<0$ strictly smaller than $p\lambda_2$. The winding number of $\lambda_d$ is strictly less than the winding number of $p\lambda_2$. This implies that $\tilde u$ and $\tilde v$ wind around each other and, as before, $u(\dot \Sigma)$ contains points outside $\mathcal{M}$, again a contradiction. This proves (i).

 Now assume by contradiction that $u(\dot \Sigma) \cap \mathcal{S}_i \neq \emptyset$ for some $i$. Since every point in $\partial \mathcal{M}$ lies in the projection of a $J$-holomorphic curve, the assumptions on $\tilde u$ imply that the intersection of $\tilde u$ with the curves $\tilde u_{i,1}$, $\tilde u_{i,2}$ and the trivial cylinder $\R \times P_{2, i}$ are isolated. Moreover, $\int_{\dot \Sigma} u^* d\alpha >0$ and a point $z\in \dot \Sigma$ so that $u(z)\in \mathcal{S}_i$ satisfies $z\in \dot \Sigma \setminus \partial \Sigma$. If $\tilde u(z) \in \R \times P_{2, i}$, then since $\tilde u$ satisfies a Cauchy-Riemann equation on the contact structure, a circle around $z$ and close to $z$ is mapped under $u$ onto a closed curve that locally surrounds $P_{2, i}$, contradicting $u(\dot \Sigma) \subset \mathcal{M}$.  By the similarity principle, this holds true even in the case $\pi \circ d u(z)=0$. Hence $\tilde u$ does not intersect $\R \times P_{2,i}$ and we may assume that $\tilde u$ intersects $\tilde u_{i,1}$. Again, these intersection points are isolated and $\tilde u$ has positive $d\lambda$-area. By positivity and stability of intersections, we may shift $\tilde u_{i,1}$ in the $\R$-direction, and $\tilde u$ intersects $r+ \tilde u_{i,1}$ for all small values of $r>0$. We keep increasing $r$ and notice that the intersecting point in the domain of $\tilde u$ cannot escape through $\partial \Sigma$ by the standing assumption $u(\partial  \Sigma) \subset \mathcal{M} \setminus \partial \mathcal{M}$. Also, it cannot escape through a negative puncture of $\tilde u$ since $r+a_{i,1}(\C)$ is uniformly bounded from below for every $r>0$. By the first statement of this proposition, the intersection point in the domain of $\tilde u$ cannot escape through a positive puncture. So there exists a compact subset $K \subset \dot \Sigma \setminus \partial \dot \Sigma$ that contains all points in the domain of $\tilde u$ corresponding to intersections between $\tilde u$ and $r+\tilde u_{i,1},$ for every $r>0$. In the same way, the points in the domain of $r+\tilde u_{i,1}$ corresponding to intersections between $\tilde u$ and $r+\tilde u_{i,1}$ must be contained in a compact subset $K'\subset \C$ for every $r>0$ since $a_{i,1}\to +\infty$ at $\infty$ uniformly. However, for $r>0$ sufficiently large, we have $r+ \min a_{i,1}(\C)> \max a(K),$ meaning that $\tilde u$ does not intersect $r+ \tilde u_{i,1}$ for every $r>0$ sufficiently large. This contradicts the assumption that $\tilde u$ intersects $\tilde u_{i,1}$, and (ii) is proved.
 \end{proof}

\subsection{\texorpdfstring{The Bishop family of $J$-holomorphic disks}{The Bishop family of J-holomorphic disks}}  Let $P\subset \mathcal{M} \setminus \partial \mathcal{M}$ be a $p$-unknotted periodic orbit with self-linking number $-1/p$ and Conley-Zehnder index $\mu(P^p)\geq 3$. Assume that $P$ bounds a $p$-disk $\mathcal{D}'\hookrightarrow \mathcal{M} \setminus \partial \mathcal{M}$. Consider the characteristic foliation $(T\mathcal{D}'\cap \xi)^\perp$, where $\perp$ is the symplectic complement with respect to any area form on $\mathcal{D}'$. We call $e_0'\in\mathcal{D}'$ a singular point of the characteristic foliation if $T_{e_0'}\mathcal{D}'=\xi_{e_0'}$. We say that a singular point  $e_0'\in \mathcal{D}' \setminus \partial \mathcal{D}'$ of the characteristic foliation is nicely elliptic if there exists a vector field $V:\mathcal{D}'\rightarrow (T\mathcal{D}' \cap \xi)^\perp$  so that  $DV(e_0')$ has positive eigenvalues. In particular, $e_0'$ is a source of $V$. Using Giroux's elimination procedure, it is always possible to $C^0$-perturb $\mathcal{D}'$ away from the boundary so that its characteristic foliation contains a unique singularity $e_0$ that is nicely elliptic. We may assume that the interior of $\mathcal{D}'$ contains no periodic orbits. We may also assume that $\mathcal{D}'$ is transverse to the Reeb vector field near $P_3^p$. Indeed, we may $C^0$-perturb $\mathcal{D}'$ in a tubular neighborhood of $P_3$ so that it approaches $P_3^p$ as an eigenvector of the asymptotic operator $A_{P_3^p}$ with winding number $1$. Such an approach implies that $\mathcal{D}'$ is transverse to the Reeb vector field near $P_3^p$ and for each cross section at $x_0\in P_3$, the $p$-disk $\mathcal{D}'$ is formed by $p$ local branches approaching $x_0$, whose tangent at $x_0$ are not a positive multiple of each other.

\begin{thm}[Hofer \cite{Hofer1993}]\label{thm_Bishop}  We may further $C^0$-perturb $\mathcal{D}'$ away from the boundary $\partial \mathcal{D}'$, keeping the previous properties, and $C^\infty$-perturb $J$ in a small neighborhood of $e_0$ so that the new disk $\mathcal{D}$ and the new almost complex structure, still denoted by $J$, admits a family of $J$-holomorphic disks $\tilde u_\tau=(a_\tau,u_\tau): (\D,i) \to \R \times \M, \tau >0$ small, satisfying

\begin{equation}\label{bishop_family}
\left\{
\begin{aligned}
& \tilde u_\tau \mbox{ is an embedding}, a_\tau|_{\partial \D} \equiv 0, \\
& u_\tau(\partial \D)\subset \mathcal{D} \setminus (\partial \mathcal{D}\cup \{e_0\}), \\  & \wind(u_\tau|_{\partial \D}, e_0) = +1.\\
\end{aligned}
\right.
\end{equation}

Finally, $\displaystyle \lim_{\tau \to 0^+} u_\tau(z) = e_0$ uniformly in $z\in \D$.
\end{thm}

One of the key properties of the $J$-holomorphic disk $\tilde u= \tilde u_\tau$ given in Theorem \ref{thm_Bishop} is that $u|_{\partial \D}$ is transverse to the characteristic foliation. Indeed, take polar coordinates $(\theta,r)$  on $\D$. Since $\Delta a  = |\pi u_s|^2 \geq 0$,  the strong maximum principle implies that $(\partial_r a)|_{\partial \D}>0$. Hence $\alpha(\partial_\theta u|_{\partial \D}) >0$ and thus $u|_{\partial \D}$ is positively transverse to the leaves of $(T\mathcal{D} \cap \xi)^\perp$. Since $\tilde u$ is an embedding and $a|_{\partial \D} \equiv 0$, we conclude from $\wind(u_\tau|_{\partial \D}, e_0) = +1$ that $u|_{\partial \D}$ is an embedding that intersects each leaf of the characteristic foliation precisely once.

Another important property of the family $u_\tau$ is that if $\tau<\tau'$, then $u_{\tau}(\partial \D)$ is contained in the interior of the disk in $\mathcal{D}$ bounded by $u_{\tau'}(\partial \D)$. Indeed, the three-dimensional group of bi-holomorphisms of $(\D, i)$ acts freely on the space of $J$-holomorphic disks satisfying \eqref{bishop_family}. For each $\tau$, the linearization  $D\bar \partial_{\widetilde J}(\tilde u_\tau)$ is a surjective Fredholm operator with Fredholm index $4$. Since the group of automorphisms of the disk has dimension $3$, the unparametrized space $\mathcal{M}_0$ of $J$-holomorphic disks satisfying the boundary conditions in \eqref{bishop_family} is one-dimensional, see Proposition \ref{prop_automatic_transversality}-(i).

Let $V\subset T\mathcal{D}$ be a vector field generating $(T\mathcal{D} \cap \xi)^\perp$, and let $N_{\tilde u_\tau}\subset \tilde u_\tau^*T(\R \times \M)$ be a normal complex line bundle  satisfying $N_{\tilde u_\tau} \oplus T_{\tilde u_\tau} = \tilde u_\tau^*T(\R \times \M)$, so that $N_{\tilde u_\tau}$ coincides with the contact structure and thus contains $0 \oplus \R V$ along $\partial \D$.  Infinitesimally, the curves near $\tilde u_\tau \in \mathcal{M}_0$ are modeled by non-trivial sections $\sigma: \D\to N_{\tilde u_\tau}$ satisfying $\sigma|_{\partial \D} \subset \R V.$ Indeed, Proposition \ref{prop_automatic_transversality}-(i) gives $c_N^\delta(\tilde u_\tau)=0$ and thus $\sigma$ never vanishes.  Hence, the nearby disks in the Bishop family do not intersect each other, see \cite{Hofer1993}, and  $u_\tau(\partial \D)$ is strictly monotone towards $P_3^p = \partial \mathcal{D}$ as $\tau$ increases. Moreover, $\pi \circ du_\tau$ never vanishes since $0\leq \wind_\pi(\tilde u_\tau) \leq c_N(\tilde u_\tau)=0$ and thus $u_\tau$ is transverse to the Reeb vector field.

\subsection{Compactness properties}

Before studying the compactness properties of the Bishop family, we assume that the contact form $\alpha$ and the almost complex structure $J$ satisfy the following $C^\infty$-generic conditions:

\begin{itemize}
    \item[H1.] $\alpha$ is nondegenerate up to action $\S(\mathcal{D},\alpha)=\int_\mathcal{D} |d\alpha|$. Notice that the energy of any $\tilde u_\tau$ satisfying \eqref{bishop_family} is bounded by $\S(\mathcal{D},\alpha).$

    \item[H2.]  Let $\tilde u=(a,u): \D \setminus \Gamma \to \R \times (\M\setminus \partial \M)$, $\emptyset \neq \Gamma \subset \D \setminus \partial \D,$ be a somewhere injective $J$-holomorphic punctured disk satisfying the analogous conditions in    \eqref{bishop_family}, and so that at the negative punctures in $\Gamma$, $\tilde u$ is asymptotic to orbits with index $\geq 2$. Then $\# \Gamma=1$ and $\tilde u$ is asymptotic to an index-$2$ orbit at its negative puncture. Moreover, $\tilde u$ is Fredholm regular.

    \item[H3.] Let $\tilde u=(a,u):\C \setminus \Gamma \to \R \times (\M\setminus \partial \M)$, $\Gamma\neq \emptyset$, be a somewhere injective $J$-holomorphic curve with non-vanishing $d\alpha$-area, asymptotic to $P_3^p$ at $\infty$, whose leading eigenvalue has winding number $+1$, and to orbits in $\partial \M$ at the negative punctures in $\Gamma$. Then $\# \Gamma=1$ and $\tilde u$ is asymptotic to some index-$2$ orbit at its negative puncture.
\end{itemize}

Condition H1 is achieved under $C^\infty$-small perturbations of $\alpha$, see Proposition 6.1 in \cite{convex}. Fixing $\mathcal{D}$, condition H2 is achieved for $C^\infty$-generic $J\in \mathcal{J}(\alpha)$, see \cite[Theorem 1.15]{HWZIII}, \cite[Theorem 5.8]{HS1} and \cite{Drago2004}. Indeed, let $\tilde u:\D \setminus \Gamma \to \R \times (\M\setminus \partial \M)$ be as in H2.  If $\#\Gamma \geq 2$ or there exists a negative puncture where $\tilde u$ is asymptotic to an orbit with index $\geq 3$, then $\text{ind}(\tilde u) < 2 - 2\#\Gamma - \chi(\D) + \#\Gamma = 1-\#\Gamma \leq 0$. Since $\tilde u$ is somewhere injective,  such a curve does not exist for a generic $J$. We conclude that $\tilde u$ has precisely one puncture whose asymptotic limit has index $2$. Moreover,  $\tilde u$ is Fredholm regular for a $C^\infty$-generic $J$. Condition H3 also holds for a $C^\infty$-generic $J$. In fact, such a curve is necessarily somewhere injective and as in H2, for a generic $J$, its weighted index is at least $1$. Hence, as before, it has at most $1$ negative puncture, and the asymptotic limit at this negative puncture has index $2$.

From now on, we assume that $\alpha$ and $J$ satisfy conditions H1-H3. The set of such $J$'s for a fixed $\alpha$ is denoted by $\J_{\text{reg}}(\alpha) \subset \J(\alpha)$. Notice that the assumptions of Theorem \ref{main1} still hold if the perturbation of $\alpha$ is sufficiently small in the $C^\infty$-topology. Indeed, we restrict to perturbations that keep $P_3$ as a periodic orbit, and the condition $\mu(P_3^p) \geq 3$ still holds. Moreover, if $P'\subset \mathcal{M} \setminus (\partial \mathcal{M} \cup P_3)$ is contractible and satisfies $\rho(P')=1$ and $\link(P', P_3^p) =0$, then its action must be $> \mathcal{S}(\mathcal{D},\alpha)$. Otherwise, we find a sequence of such nondegenerate perturbations $\alpha_n \to \alpha$ in $C^\infty$ and periodic orbits $P_n$ of $\alpha_n$, satisfying $\rho(P_n)=1$, $\link(P_n,P_3^p)=0$ and $\mathcal{A}(P_n) \leq \mathcal{S}(\mathcal{D},\alpha_n)$,
and converging to a periodic orbit $P\subset \M \setminus (\partial \M \cup P_3)$ of $\alpha$ which is contractible and satisfies $\rho(P)=1$, $\link(P,P_3^p)=0$ and $\mathcal{A}(P) \leq \mathcal{S}(\mathcal{D},\alpha),$ a contradiction. Notice that $P$ cannot coincide with a cover of $P_3$ since $P$ is contractible and its rotation number is $1$.

Since the disks in the Bishop family are automatically transverse, the family $\tilde u_\tau$ persists under $C^\infty$-small perturbations of $J$, and thus we may assume they exist for some $J\in \J_{\text{reg}}(\alpha)$.

We choose three distinct leaves $l_{1}, l_i$ and $l_{-1}$ of the characteristic foliation $(T\mathcal{D}\cap \xi)^\perp$ issuing from $e_0$ to $\partial \mathcal{D}$, in the counterclockwise direction, and reparametrize the disks $\tilde u_\tau$ in the Bishop family so that
\begin{equation}\label{param_Bishop}
u_\tau(z)\in l_z,\quad \forall z\in \{1,i,-1\},\quad \forall \tau \in (0,1).
\end{equation}
This parametrization, introduced in \cite{Hofer1993}, controls the compactness properties of the Bishop family. Take the maximal family $\tilde u_\tau=(a_\tau,u_\tau), \tau \in (0,1)$ of such $J$-holomorphic disks. By Proposition \ref{prop_inside},  $u_\tau(\D) \subset \mathcal{M} \setminus \partial \mathcal{M}$ for every $\tau \in (0,1)$.

\begin{prop}\label{prop_buildings_boundary} Under the conditions H1-H3, the maximal family  $\tilde u_\tau, \tau \in (0,1),$ SFT-converges as $\tau \to 1$ to one of the following  buildings $\mathcal{B}$:
\begin{itemize}

     \item[(i)] $\mathcal{B}=(\tilde v_1,\tilde v_2)$ has two levels. The top level $\tilde v_1$ is a nicely embedded punctured disk $\tilde v_1=(a_1,v_1):\D \setminus \{0\} \to \R \times \M$ satisfying the boundary conditions in \eqref{bishop_family}, with a negative puncture at $0$ asymptotic to some $P_{2,i} \subset \partial \M$. The lowest level $v_2$ is an embedded plane asymptotic to $P_{2, i}$ that projects to one of the hemispheres of $\S_i \setminus P_{2, i}$.

     \item[(ii)] $\mathcal{B}=(\tilde v_1,\tilde v_2)$ has two levels. The top level is a punctured disk $\tilde v_1:\D \setminus \{0\} \to \R \times \M$ whose image is  $(-\infty,0] \times P_3.$ The lowest level is a nicely embedded plane $\tilde v_2$ asymptotic to $P_3^p$, and projecting to $\M\setminus \partial \M$.

     \item[(iii)] $\mathcal{B}=(\tilde v_1,\tilde v_2,\tilde v_3)$ has three levels. The top level $\tilde v_1$ is a punctured disk $\tilde v_1:\D \setminus \{0\}$ whose image is  $(-\infty,0] \times P_3.$ The second level is a nicely embedded cylinder $\tilde v_2:\R \times \R / \Z \to \R \times \M$ with a positive end at $P_3^p$ and a negative end at some $P_{2,i}\subset \partial \M.$ The lowest level consists of an embedded plane $\tilde v_3:\C \to \R \times \M$ asymptotic to $P_{2,i}$ that projects to one of the hemispheres of $\S_i \setminus P_{2,i}$.
\end{itemize}
Moreover, the building in (i) occurs if and only if $\inf_\tau \text{dist} (u_\tau(\partial \D), \partial \mathcal{D}) >0.$
\end{prop}

\begin{proof}
The parametrization \eqref{param_Bishop} prevents bubbling-off points at the boundary $\partial \D$. In fact, it follows from \cite[Theorem 2.1]{char1}, see also \cite{Hofer1993}, that there exists $\epsilon>0$ small such that
$
\sup\{|\nabla \tilde u_\tau(z)|:1-\epsilon<|z|\leq 1, \tau\in(0,1)\}<+\infty.
$
Hence $\tilde u_\tau$ only admits bubbling-off points in the interior of $\D$. Consider a sequence $\tau_n\to 1^-$ as $n \to \infty$ and denote $\tilde u_n=(a_n,u_n):=\tilde u_{\tau_n}$. Since every bubbling-off point takes away at least $\gamma/2>0$ of the $d\alpha$-area ($\gamma>0$ is the shortest period of a periodic orbit) and $\int_\D \tilde u_n ^* d\alpha$ is uniformly bounded by $\S(\mathcal{D}, \alpha)$, we find a finite set  $\Gamma \subset \D \setminus \partial \D$ so that, up to a subsequence, $\tilde u_n \to \tilde v_1=(b_1,v_1):\D \setminus \Gamma \to \R \times \M$ in $C^\infty_{\text{loc}}$ as $n\to +\infty$. Every puncture of $\tilde v_1$ is necessarily negative.

Assume first the $\inf_\tau \text{dist}(u_\tau(\partial \D),\partial \mathcal{D})>0.$ Since $\mathcal{D}\setminus \partial \mathcal{D}$ contains no periodic orbits, we conclude that $\tilde v_1$ satisfies the boundary conditions in \eqref{bishop_family} and its $d\alpha$-area is positive. If $\Gamma = \emptyset,$ then $\tilde v_1$ is an embedded disk since it is the limit of embeddings and $\tilde v_1$ is embedded near the boundary. Indeed, recall that $\partial_r b_1 >0$ and $v_1|_{\partial \D}$ transversely intersects $(T\mathcal{D} \cap \xi)^\perp$. Since $\wind(v_1|_{\partial \D},e_0)=\wind(u_n|_{\partial \D},e_0)=+1$ for every $n$, it follows that  $v_1|_{\partial \D}$ and $\tilde v_1$ are embedded near the boundary.  By the automatic transversality of the disks in the Bishop family, see Proposition \ref{prop_automatic_transversality}-(i), we conclude that $\tilde v_1$ lies in the interior of a unique local family of such disks containing $\tilde u_n$, $n$ large, contradicting the maximality of the family $\tilde u_\tau, \tau \in (0,1)$. Hence $\Gamma \neq \emptyset.$  Since the energy of $\tilde u_n$ is uniformly bounded by $\S(\mathcal{D},\alpha)$ and $\alpha$ is nondegenerate up to action $\S(\mathcal{D}, \alpha)$, see condition H1, we can invoke the SFT compactness theorem to obtain a building $\mathcal{B}$ which is the SFT-limit of $\tilde u_n$, up to a subsequence, so that the top level of $\mathcal{B}$ is $\tilde v_1$. We follow \cite{HS1}, and $\mathcal B$ has the structure of a rooted tree with $v_1$ as the root. Moreover, $\mathcal B$ has no trivial cylinder. Notice that $\tilde v_1$ is somewhere injective since it is embedded near the boundary. The curves below $\tilde v_1$ have precisely one positive puncture and an arbitrary number of negative punctures. Also, every asymptotic limit of a curve in $\mathcal{B}$ is contractible and thus its index is well-defined. Moreover, for every curve in $\mathcal{B}$, the action of every asymptotic limit is uniformly bounded by $\mathcal{S}(\mathcal{D},\alpha)$. If necessary, we may see the disks in the Bishop family lifted to the symplectization $\R \times S^3$ whose almost complex structure is the one lifted from $\R \times L(p,q)$ under $\pi_{p,q}:S^3 \to L(p,q)$. The same holds for the limiting curves in $\mathcal{B}$ after taking a subsequence. In particular, the asymptotic limits have a well-defined Conley-Zehnder index computed in a symplectic global frame of $\xi \to S^3$.

First, we prove that all asymptotic limits of $\tilde {v} _1$ have index  $\geq 2$. This argument is found in \cite{HS1}. Suppose by contradiction that $\tilde v_1$ has an asymptotic limit $\gamma$ with index $\leq 1$ at a negative puncture in $\Gamma$. Then $\mathcal{B}$ contains a curve $\tilde v_2:\C \setminus \Gamma_2\to \R \times \M$ below $\tilde v_1$, and positively asymptotic to $\gamma$ at  $\infty$.  All punctures in $\Gamma_2$ are negative, and their asymptotic limits are contractible. We claim that $\tilde v_2$ has an asymptotic limit at a negative puncture whose index is $\leq 1$. If the $d\alpha$-area of $\tilde v_2$ vanishes, then $\Gamma_2 \neq \emptyset$ and $\tilde v_2$ is a $k$-branched cover of the trivial cylinder over a contractible Reeb orbit $\hat \gamma$. We may assume that $\#\Gamma_2 \geq 2$, otherwise $\tilde v_2$ is a trivial cylinder and the claim trivially holds. Since the index of $\gamma=\hat \gamma^k$ is $\leq 1$, the index of $\hat \gamma$ is necessarily $\leq 1$. Thus, the asymptotic limit at a negative puncture of $\tilde v_2$ has the form $\hat \gamma^m$ for some $m< k$. Hence, $\hat \gamma^m$ has index $\leq 1.$
Now, assume that the $d\alpha$-area of $\tilde v_2$ is positive. Let $\wind_\infty(z)$ be the winding number of the leading eigenvalue of $\tilde v_2$ at $z\in\Gamma\cup\{\infty\}$. Since the index of $\gamma$ is $\leq 1$, we have $\wind_\infty(\infty) \leq 0$. This implies that $0\leq \wind_\pi (\tilde v_2)= \wind_\infty(\infty) - \sum_{z\in \Gamma_2}\wind_\infty(z)-(1-\#\Gamma_2) \leq -1 - \sum_{z\in \Gamma_2} (\wind_\infty(z)-1)$. It follows that there exists $z\in \Gamma_2$ such that $\wind_\infty(z)\leq 0$. This implies that $\gamma_z$ has index $\leq 0$.   We conclude in both cases that $\tilde v_2$ has an asymptotic limit at a negative puncture with index $\leq 1$. Similarly, we find a curve $\tilde v_3 \in \mathcal{B}$ below $\tilde v_2$ with a negative puncture whose asymptotic limit has index $\leq 1$. Continuing this procedure, we eventually find a $J$-holomorphic plane in a leaf of $\mathcal{B}$, which is asymptotic to a Reeb orbit with index $\leq 1$. However, such a plane cannot exist, and we conclude that every asymptotic limit of $\tilde v_1$ has index $\geq 2$.

Next, we claim that $\#\Gamma = 1$ and $\tilde v_1$ is asymptotic to an index-$2$ orbit $P_2=P_{2,i}\subset \partial \mathcal M$ for some $i$.
Since $\tilde v_1$ is somewhere injective, it follows from condition H2, that $\tilde v_1$ has precisely one negative puncture whose asymptotic limit $P_2$ has index $2$. Hence $\tilde v_1$ is a punctured $J$-holomorphic disk, and we may assume, after parametrization, that $\Gamma = \{0\}$. Since $P_2$ is contractible and has rotation number $1$, it is geometrically distinct from $P_3$. This implies that the $d\alpha$-area of $\tilde v_1$ is positive. Also, $P_2$ is unlinked with $P_3^p$ since the image of any circle $S\subset \D$ under $u_n$ is unlinked with $P_3^p$. Hence, the same holds for $\tilde v _1$. Moreover, since the action of $P_2$ is $\leq \mathcal{S}(\mathcal{D},\alpha)$, we conclude from our linking hypothesis that $P_2$ coincides with one of the Lyapunov orbits, say $P_2 = P_{2,i_0}$ for some $i_0\in \{1,\ldots,l\}$. In particular, $\tilde v_1$ is an embedded punctured disk asymptotic to $P_{2,i_0}$ at its negative puncture.

The  curve $\tilde v_2=(v_2,b_2):\C \setminus \Gamma_2 \to \R \times \mathcal{M}$ below $\tilde v_1$ is asymptotic to $P_{2,i_0}$ at its positive puncture $\infty$. By Proposition \ref{prop_inside}-(i), $\tilde v_{2}$ must coincide with one of the curves projecting to a hemisphere of $\mathcal{S}_{i_0} \setminus P_{2,i_0}.$ We may assume that $\tilde v_2 = \tilde u_{i_0,1}$ and the building $\mathcal{B}$ has no other levels besides the ones formed by $\tilde v_1$ and $\tilde v_2$. Thus $\mathcal{B}$ is as in (i).

Now assume that  $\inf_\tau \text{dist}(u_\tau(\partial \D), \partial \mathcal{D}) =0$, and pick $\tilde u_n := \tilde u_{\tau_n}$ with $\tau_n \to 1$ as $n\to \infty$. As before, no bubbling-off point occurs at the boundary $\partial \D$. Moreover, the  first level of $\mathcal{B}$ is a punctured disk $\tilde v_1=(b_1,v_1):\D \setminus \Gamma_1 \to \R \times \mathcal M$, $\Gamma_1\subset \D \setminus \partial \D,$ so that $v_1$ intersects $P_3$. If $\tilde v_1$ has positive $d\alpha$-area, then an important fact proved in \cite[Theorem 4.4]{char1} using a degree argument implies that $\tilde u_n$ intersects $\R \times P_3$ for $n$ sufficiently large, a contradiction. Hence the $d\alpha$-area of $\tilde v_1$ vanishes and $\tilde v_1$ is a trivial half-cylinder over $P_3^p$ whose image is $(-\infty,0] \times P_3$. This follows from the fact that $\mathcal{D}$ is a $p$-disk for $P_3$ and thus $u_n(\partial \D)$ converges to $P_3^p$ as $n\to \infty$. The curve $\tilde v_2=(b_2,v_2):\C \setminus \Gamma_2 \to \R \times \mathcal M$ below $\tilde v_1$ is asymptotic to $P_3^p$ at its positive puncture $\infty$. We may assume that $\tilde v_2$ is not a trivial cylinder over $P_3^p$. Also, it cannot be a branched cover of a trivial cylinder over $P_3^{p_0},$ $p_0$ divides $p$, since $P_3^j$ is non-contractible for every $0<j=mp_0 <p, m\in \N,$ and all asymptotic limits of $\tilde v_2$ are contractible. Hence, the $d\alpha$-area of $\tilde v_2$ is positive, and the leading eigenvalue of $\tilde v_2$ at $\infty$ has winding number $1$ with respect to a global symplectic frame of $\xi \to S^3$. This follows from the fact that the lift of $P_3$ to $S^3$ is a $\Z_p$-symmetric trivial knot $\hat P_3\equiv P_3^p$ and the lift of the image of any circle $S\subset \D$ under $\tilde u_n$ is unlinked with $\hat P_3$. In particular, the leading eigenvector, when projected to $P_3^p$ is simple in the sense that it does not cover an eigenvector along a smaller iterate of $P_3$. This implies that $\tilde v_2$ is somewhere injective. Arguing as before, the index of every asymptotic limit of $\tilde v_2$ at a negative puncture is $\geq 2$. The inequality $0 \leq \wind_\pi(\tilde v_2) = \wind_\infty(\infty) - \sum_{z\in \Gamma_2} \wind_\infty(z) -(1-\#\Gamma_2)=-\sum_{z\in \Gamma_2}( \wind_\infty(z) - 1)$ implies that  $P_z$ has index $2$ for every $z\in \Gamma_2$. Indeed, notice that $\wind_\infty(z) \geq 2$ if $P_z$ has index $\geq 3$. Since $P_z$ is unlinked with $P_3$ and the only orbits with action $\leq \S(\mathcal{D},\alpha)$ that are unlinked with $P_3$ are the orbits in $\partial \mathcal M$, we conclude that $P_z \subset \partial \M, \forall z\in \Gamma_2$.  Our generic choice of $J$ implies that $\#\Gamma_2\in \{0,1\}$, see condition H3.

If $\Gamma_2=\emptyset$, then $\tilde v_2$ is an embedded plane asymptotic to $P_3^p$, and not intersecting $\R \times P_3$. Its projection $v_2$ is an embedding, see Theorem 5.20 in \cite{Si2}. By Proposition \ref{prop_inside}, $v_2(\C) \subset \mathcal{M} \setminus \partial \mathcal{M}$. The building $\mathcal{B}=(\tilde v_1,\tilde v_2)$ is as in (ii). If $\#\Gamma_2=1$,  let $P_{2,i_0}$ be the asymptotic limit at the negative puncture in $\Gamma_2$.  The level $\tilde v_3$ below $\tilde v_2$ must then be a plane asymptotic to $P_{2,i_0}$ projecting to a hemisphere of $\mathcal{S}_{i_0} \setminus P_{2,i_0}$, see Proposition \ref{prop_inside}  As limits of embedded curves, both $\tilde v_2$ and $\tilde v_3$ are embedded and do not intersect any of its asymptotic limits. Therefore, their projections are also embedded, see Theorem 5.20 in \cite{Si2}. This gives the building in (iii). \end{proof}

\begin{prop}\label{prop_no_boundary} There exists a sequence of $J$-holomorphic disks $\tilde u_n=(a_n,u_n):\D \to \R \times
\mathcal{M}$ satisfying the conditions in \eqref{bishop_family}, and so that $\inf_n \text{dist}(u_n(\partial \D),\partial \mathcal{D}) =0$. Moreover, $\tilde u_n$ admits an SFT-limit as in Proposition \ref{prop_buildings_boundary}-(ii) or (iii).
\end{prop}

\begin{proof}
Let us assume that the maximal family $\tilde u_\tau, \tau\in (0,1),$ issuing from the nicely elliptic singularity of $(T\mathcal{D} \cap \xi)^\perp$ for $\tau=0$, satisfies $\inf_\tau \text{dist}(u_\tau(\partial \D),\partial \mathcal{D}) >0$. Then it breaks into a building $\mathcal{B}=(\tilde u_1,\tilde v_1)$ as in Proposition \ref{prop_buildings_boundary}-(i). Moreover, $\tilde u_1:\D \setminus \{0\} \to \R \times \mathcal{M}$ is an embedded $J$-holomorphic punctured disk satisfying the boundary conditions in \eqref{bishop_family}, it is asymptotic at $0\in \D$ to $P_{2,i}\subset \S_i\subset \partial \M$ for some $i$, and $\tilde v_1:\C \to \R \times \mathcal{M}$ is one of the planes $\tilde u_{i,1}$ or $\tilde u_{i,2}$. Let us assume without loss of generality that  $\tilde v_1=\tilde u_{i,1}$.  Our generic choice of $J$ implies that $\tilde u_1$ is a regular curve. Since $\tilde u_{i,2}$ is also an embedded regular curve, we can glue the embedded curves $\tilde u_1$ and $\tilde u_{i,2}$, see \cite{HT1, HT2}, to obtain a new family of embedded $J$-holomorphic disks, denoted $\tilde u_\tau, \tau\in (1,1+\epsilon), \epsilon>0$ small.

\begin{lem} Every disk $\tilde u_\tau=(a_\tau,u_\tau),\tau\in (1,1+\epsilon),$ satisfies  \eqref{bishop_family}.  Moreover, $u_\tau(\D) \subset \mathcal{M} \setminus \partial \mathcal{M}$, and $u_\tau(\partial \D)$ lies in the exterior of  $u_1(\partial \D)\subset \mathcal{D}$.
\end{lem}

\begin{proof}
 Since $\tilde u_1$ and $\tilde u_{i,2}$ are embedded and their generalized intersection number vanishes, the glued curves $\tilde u_\tau$ are also embedded. The conditions $a_\tau|_{\partial \D} \equiv 0$, $u_\tau(\partial \D) \subset \mathcal{D} \setminus (\partial \mathcal{D} \cup \{e_0\})$ and $\wind(u_\tau|_{\partial \D},e_0) =+1$ follow from similar properties of $\tilde u_1$. We also know that $\wind_\pi(\tilde u_\tau)=0$ and thus $u_\tau$ is transverse to the flow.

 To prove that $u_\tau(\D) \subset \mathcal{M} \setminus \partial \mathcal{M},\forall \tau -1>0$ sufficiently small, we assume by contradiction that $u_\tau(\D) \cap \partial \mathcal{M} \neq \emptyset$ for some $\tau-1>0$ arbitrarily small. Since $\partial \mathcal{M}$ is the union of projections of $J$-holomorphic curves in $\R \times \partial \mathcal{M}$,  $\tilde u_\tau$ admits only isolated intersections with a curve $\tilde v$ projecting to $\S_j \subset \partial \M$ for some $j$. Since $u_\tau$ is transverse to the flow, we may assume that $\tilde v=(b,v):\C \to \R \times \mathcal{M}$ projects to one of the hemispheres of $\S_j \setminus P_{2,j}$.  Notice that $u_\tau(\partial \D)$ and $P_{2,j}$  are unlinked and $v(\C) \cap u_\tau(\partial \D) = \emptyset.$ Consider the shifted curves $\tilde v_a:=a+ \tilde v, a\in \R.$ Notice that  $u_\tau(\partial \D)$ does not intersect $v_a(\C) = v(\C)\subset \partial \M$. Since $a_\tau|_{\D} \leq 0$ and $b(z) \to +\infty$ as $z\to \infty$,  the points in $\C$ and $\D$ corresponding to intersections between $\tilde v_a,a\geq 0,$ and $\tilde u_\tau$ stay uniformly bounded in $\C$  and inside some disk $\D_r=\{z\in \C:|z|\leq r\} \subset \D,$  $0<r<1,$ respectively. However, for $a>0$ sufficiently large, such intersections cannot exist since $a_\tau|_{\D}\leq 0$ and $a+\min_{z\in \C} b(z) \to +\infty$ as $a\to +\infty$. From the stability and positivity of intersections between pseudo-holomorphic curves,  intersections between $\tilde u_\tau$ and $\tilde v_a$ cannot exist for any $a\geq 0,$ a contradiction.

 To prove that $u_\tau(\partial \D)$ lies in the exterior of $u_1(\partial \D)\subset \mathcal{D}\setminus \partial \mathcal{D}$, we need the following result from \cite{AH}  concerning uniqueness of pseudo-holomorphic discs.

\begin{prop}[Abbas-Hofer {\cite[Theorem 7.6.3]{AH}}]\label{prop_abbas_hofer}
Let $\tilde v,\tilde u_\tau:\D \to \R \times M,\tau\in(-1,1),$ satisfy \eqref{bishop_family}. Assume that $\tilde u_\tau(\D) \cap \tilde v(\D)=\emptyset$ for every $\tau \in (0,\epsilon)$,  and that $\tilde u_0(\D) \cap \tilde v(\D) \neq \emptyset.$ Then $\tilde v = \tilde u_0$, up to a bi-holomorphism of $\D$.
\end{prop}

We know that  $\tilde u_\tau(\D) \cap \tilde u_1(\D\setminus \{0\})=\emptyset$ for every $\tau<1$ sufficiently close to $1$. This follows from the fact that the sections of the normal bundle representing infinitesimal variations of the disks in the Bishop family never vanish. The same holds for the family $\tilde u_\tau$ for $\tau>1$ sufficiently close to $1$. It follows from this fact and  $\tilde u_\tau(\D) \subset \mathcal{M} \setminus \partial \mathcal{M}$ that if we fix $\tau_0<1$ arbitrarily close to $1$, then for every $\tau>1$ sufficiently close to $1$, we have $\tilde u_{\tau_0}(\D) \cap \tilde u_\tau(\D)=\emptyset.$

Let us assume by contradiction that $u_\tau(\partial \D),\tau>1,$ lies in the same component of $\mathcal{D} \setminus u_1(\partial \D)$  as $u_\tau(\partial \D), \tau<1$, i.e., in the disk-like region  bounded by $u_1(\partial \D)\subset \mathcal{D}$. We may fix $\tau_0<1$ sufficiently close to $1$ and choose $\tau_0^*>1$ sufficiently close to $1$ so that $\tilde u_{\tau_0}(\D) \cap \tilde u_{\tau_0^*}(\D) \neq \emptyset.$ 
We may assume that $\tilde u_{\tau_0}(\D) \cap (\tilde u_1(\D \setminus \{0\})\cup (\R \times \partial \mathcal{M})) = \emptyset$ and $\tilde u_{\tau_0^*}(\D) \cap (\tilde u_1(\D \setminus \{0\})\cup (\R \times \partial \mathcal{M})) = \emptyset$ as explained before. From the strict monotonicity of the circles $u_\tau(\partial  \D)\subset \mathcal{D}$, we know that there exists $\tau_1^*$ satisfying $1<\tau_1^*< \tau_0^*$ such that for  every $1<\tau < \tau_1^*$ we have
\begin{itemize}
    \item $\tilde u_\tau(\D) \cap \tilde u_{\tau_0}(\D) = \emptyset$. In particular, $u_\tau(\partial \D) \cap u_{\tau_0}(\partial \D) =\emptyset$.
    \item $u_{\tau}(\partial \D)$ lies in the annulus  in $\mathcal{D}$ bounded by $u_{\tau_0}(\partial \D)$ and $u_{1}(\partial \D)$.
\end{itemize}
Hence we find $\tau_2^*>1$ satisfying $1<\tau_1^* < \tau_2^* < \tau_0^*$ such that $\tilde u_{\tau_2^*}(\D) \cap \tilde u_{\tau_0}(\D) \neq \emptyset$ and  $\tilde u_\tau(\D) \cap \tilde u_{\tau_0}(\D) = \emptyset$ for every $1< \tau <\tau_2^*$.  By Proposition \ref{prop_abbas_hofer}, this implies that $\tilde u_{\tau_2^*}(\D) = \tilde u_{\tau_0}(\D)$. This is a contradiction since for $\tau_0<1<\tau_0^*$ sufficiently close to $1$, $u_{\tau_0}(\D)$ is arbitrarily close to  $u_1(\D \setminus \{0\})\cup P_{2,j} \cup v_1(\C)$ and $u_{\tau_2^*}(\D)$ is arbitrarily close to  $u_1(\D \setminus \{0\}) \cup (\S_j\setminus v_1(\C))$. We conclude that $u_\tau(\partial \D)$ lies in the exterior of $u_1(\partial \D)\subset \mathcal{D}.$
\end{proof}

We continue the proof of Proposition \ref{prop_no_boundary}. Consider the maximal family $\tilde u_\tau=(a_\tau,u_\tau),$ $\tau \in (1,2),$ satisfying \eqref{bishop_family}. Notice that
$\tilde u_\tau$ is embedded for every $\tau$, $\tilde u_\tau(\D) \cap \tilde u_{\tau'}(\D) = \emptyset, \forall \tau \neq \tau',$ and if $\tau < \tau'$
then $u_\tau(\partial \D)$ is contained in the interior of $u_{\tau'}(\partial \D) \subset \mathcal{D}.$ If $\inf_{\tau\in(1,2)} \text{dist}(u_\tau(\partial \D),\partial \mathcal{D}) =0$, then the proof of Proposition \ref{prop_no_boundary} is finished. Otherwise, take a sequence $\tilde w_n =(c_n,w_n):= \tilde u_{\tau_n}, \tau_n \to 2$, and apply Proposition \ref{prop_buildings_boundary}-(i) to obtain a new embedded half-cylinder $\tilde u_2=(a_2,u_2):\D \setminus \{0\} \to \R \times \mathcal{M}$ and a plane $\tilde v_2=(b_2,v_2):\C \to \R \times \mathcal{M}$ so that $u_2(\D) \setminus \{0\} \subset \mathcal{M} \setminus \partial \mathcal{M}$ and $v_2(\C) \subset \S_j\subset \partial \mathcal{M}$ for some $j$. Gluing $\tilde u_2$ with the plane projecting to the hemisphere $\S_j \setminus (P_{2,j} \cup v_2(\C))$ as before, we obtain a new family of embedded disks $\tilde u_\tau=(a_\tau,u_\tau),\tau \in (2,2+\epsilon),$ satisfying \eqref{bishop_family}. Using Proposition \ref{prop_abbas_hofer}, we conclude that $u_\tau(\partial \D)$ lies in the exterior of $u_2(\partial \D)\subset \mathcal{D}$ and the monotonicity of the boundary $u_\tau|_{\partial \D}$ still holds. Repeating this procedure either after finitely many steps we find a maximal family $\tilde u_\tau, \tau\in (k,k+1)$ satisfying \eqref{bishop_family} and so that $\inf_{\tau\in(k,k+1)} \text{dist}(u_\tau(\partial \D), \partial \mathcal{D}) =0$, or, after a gluing procedure as above, we obtain countably many families $\tilde u_\tau, \tau \in (k+1,k+2), k\in \N.$ In the first case, the proof of Proposition \ref{prop_no_boundary} is finished. In the second case, since every half-cylinders $\tilde w:\D \setminus \{0\} \to \R \times \M$ satisfying \eqref{bishop_family} and asymptotic to some $P_{2,i}$ is regular, $\tilde w$ is isolated in the space of such half-cylinders with the same asymptotic limit at the negative puncture. Hence, such half-cylinders cannot accumulate away from $\partial \mathcal{D}$ and thus Proposition \ref{prop_buildings_boundary} tells us that we necessarily have $\inf_{k\in \N} \text{dist}(u_k(\partial \D), \partial \mathcal{D}) =0$. Hence, the desired sequence of $J$-holomorphic disks is also obtained in this case. \end{proof}

By Proposition \ref{prop_no_boundary}, either we find a nicely embedded $J$-holomorphic plane $\tilde u=(a,u):\C \to \R \times \mathcal{M}$ asymptotic to $P_3^p$, or a nicely embedded $J$-holomorphic cylinder $\tilde v=(b,v):\C \setminus \{0\} \to \R \times \mathcal{M}$ with a positive puncture at $\infty$ asymptotic to $P_3^p$, whose leading eigenvalue has winding number $1$, and a negative puncture at $0$ asymptotic to some $P_{2,j} \subset \S_j$. We must have $u(\C) \subset \mathcal{M} \setminus (\partial \mathcal{M}\cup P_3)$ and $v(\C \setminus \{0\}) \subset \mathcal{M} \setminus (\partial \mathcal{M}\cup P_3)$ in each case, see Proposition \ref{prop_inside}. In the second case, we can glue $\tilde v$ with a $J$-holomorphic curve projecting to one of the hemispheres in $\S_j \setminus P_{2,j}$ to obtain a holomorphic plane $\tilde u=(a,u):\C \to \R \times \mathcal{M}$ asymptotic to $P_3^p$ such that $u(\C) \subset \mathcal{M}\setminus (\partial \mathcal{M} \cup P_3)$ as in the first case. Indeed, in the gluing procedure, we may consider only sections of the normal bundle that converge to $P_3^p$ with sufficiently large exponential decay, so that the glued curve $\tilde u$ approaches $P_3^p$ at $\infty$ with a leading eigenvalue that has winding number $1$ with respect to the global trivialization of the contact structure. In particular, intersections with $P_3^p$ are not created from infinity and $\wind_\pi(\tilde u) = \wind_\infty(\infty) -1 = 0$, implying that $u$ is transverse to the flow. The fact that $u$ is an embedding follows from standard intersection theory of pseudo-holomorphic curves, see \cite{HS2, Si2}. We have proved the following proposition.

\begin{prop}
    There exists an embedded $J$-holomorphic plane $\tilde u=(a,u):\C \to \R \times \mathcal{M}$ asymptotic to $P_3^p,$ so that its projection $u:\C \to \mathcal{M} \setminus (\partial \mathcal{M} \cup P_3^p)$ is embedded and transverse to the flow. The winding of the leading eigenvalue of $\tilde u$ at $\infty$ is $+1.$
\end{prop}

In the case $p=1$, these are the fast planes considered in \cite{hryn, hryn2}. The case $p>1$ was treated in \cite{HLS, HS2}. Let $\mathcal{M}_{P_3^p}$ denote the set of $J$-holomorphic planes $\tilde u=(a,u):\C \to \R \times \mathcal{M}$ asymptotic to $P_3^p$ and satisfying
$$
\left\{\begin{aligned}
& \tilde u \mbox{ and }  u \mbox{ are embeddings},\\
& u(\C) \subset \mathcal{M} \setminus (\partial \mathcal{M} \cup P_3^p),\\
& \wind_\pi(\tilde u) =0.
\end{aligned}
\right.
$$

The $\delta$-weighted index of any $\tilde u\in \mathcal{M}_{P_3^p}$ is $\text{ind}^\delta(\tilde u) = \mu^\delta (\tilde u) - (\chi(\C P^1) - 1) = 3-1=2.$ With this weight,  $\tilde u$ is automatically transverse, and $\mathcal{M}_{P_3^p}$ is a two-dimensional manifold containing all the $\R$-translations of $\tilde u$. One can cut out a transverse one-parameter family of $\tilde u_\tau=(a_\tau,u_\tau)\in \mathcal{M}_{P_3^p},\tau \in (-\epsilon,\epsilon)$, with $\tilde u_0 =\tilde u$, so that
\begin{equation}\label{cap_empty}
u_\tau(\C) \cap u_{\tau'}(\C) = \emptyset, \quad \forall \tau\neq \tau'.
\end{equation}

Consider the maximal family $\tilde u_\tau\in \mathcal{M}_{P_3^p},\tau\in(-1,1),$ satisfying \eqref{cap_empty}, where maximality means that $\cup_\tau u_\tau(\C)\subset \mathcal{M} \setminus \partial \mathcal{M}$  has the largest volume.

\begin{prop}\label{prop_buildings_no_boundary} Under the conditions above, the maximal family  $\tilde u_\tau, \tau \in (0,1),$ converges in the SFT sense  as $\tau \to 1$ (or as $\tau \to -1$) to a two-level building $\mathcal{B}=(\tilde u_1,\tilde v_1)$, so that  $\tilde u_1$ is a nicely embedded cylinder $\tilde u_1=(a_1,u_1):\C \setminus \{0\} \to \R \times \mathcal{M}$  with a positive puncture at $\infty$ asymptotic to $P_3^p=(x_3,pT_3)$, whose leading eigenvalue has winding number $+1$, and a negative puncture at $0$ asymptotic to some $P_{2,i} \subset \partial \M$. The lowest level $\tilde v_1$ is a plane asymptotic to $P_{2, i}$ and projecting to one of the hemispheres of $\S_i \setminus P_{2, i}$.
\end{prop}

\begin{proof} Take a sequence $\tilde u_n=(a_n,u_n):=\tilde u_{\tau_n},$  with $\tau_n \to 1$. Since the energy of $\tilde u_n$ is constant equal to $pT_3$ and $\alpha$ is nondegenerate up to action $\S(\mathcal{D},\alpha) \geq pT_3$, see condition H1, there exists a building $\mathcal{B}$ which is the SFT-limit of $\tilde u_n$, up to a subsequence. The first level of $\mathcal{B}$ is a curve $\tilde v_1:\C \setminus \Gamma \to \R \times \mathcal{M}$, with a positive puncture at $\infty$, where it is asymptotic to $P_3^p$. The asymptotic limits of $\tilde v_1$ at the negative punctures in $\Gamma$ are contractible. The $d\alpha$-area of $\tilde v_1$ is positive since smaller iterates $P_3^j$, $j<p$ are non-contractible and $\tilde v_1$ is not a trivial cylinder. Hence, the leading eigenvalue of $A_{P_3^p}$ has winding number $+1$. In particular, $\tilde v_1$ is somewhere injective. Every puncture  $z\in \Gamma$ is negative and the index of its asymptotic limit $P_z$ is $2$ and thus unlinked with $P_3$, see the proof of Proposition \ref{prop_buildings_boundary}.

We claim that $\#\Gamma = 1$ and $\tilde v_1$ is asymptotic to some $P_{2,i}\subset \partial \M$. Indeed, if  $\Gamma=\emptyset$, then $\tilde v_1\in \mathcal{M}_{P_3^p}$, contradicting the maximality of the family $\tilde u_\tau$. Hence $\#\Gamma \geq 1$. From the generic choice of $J$, see condition H3, we have $\#\Gamma =1$, and we can assume that $\Gamma = \{0\}$. Since the action of $P_2$ is $<pT_3$, the assumptions on $\alpha$ imply that $\tilde v_1$ is asymptotic to some $P_{2,i}\subset \partial \M$.  By Proposition \ref{prop_inside}, the curve $\tilde v_2$ below $\tilde v_1$ is an embedded $J$-holomorphic plane asymptotic to $P_{2,i}$ and projecting to a hemisphere of $\S_i\setminus P_{2,i}$. The embedding properties of $\tilde u_n$, see \eqref{cap_empty}, imply that $\mathcal{B}$ is the unique SFT-limit of the family $\tilde u_\tau$.
\end{proof}

We can assume that the family $\tilde u_\tau, \tau \in (0,1),$ is parametrized so that $u_\tau(\C)$ moves in the direction of the flow as $\tau$ increases. By Proposition \ref{prop_buildings_no_boundary}, the family $\tilde u_\tau$ breaks onto a nicely embedded cylinder $\tilde u_1$ asymptotic to $P_3^p$ at the positive puncture and asymptotic to $P_{2,j} \subset \S_j$ at the negative puncture, plus a plane $\tilde v_2$ that projects to a hemisphere of $\S_j\setminus P_{2,j}$. Let $\tilde v_1'$ be the plane asymptotic to $P_{2,j}$ projecting to the other hemisphere of $\S_j$. Using the appropriate weights at $P_3^p$, we can glue $\tilde u_1$ with $\tilde v_1'$, see \cite{HT1,HT2}, to obtain a family of planes $\tilde u_\tau\in \mathcal{M}_{P_3^p}, \tau \in (1,1+\epsilon),$ asymptotic to $P_3^p$ at $\infty$, whose leading eigenvalue has winding number $+1$.
Their projections $u_\tau$ to $\mathcal{M} \setminus \partial \mathcal{M}$ do not intersect each other and also do not intersect the corresponding maps of the family $\tilde u_\tau$, for $\tau \in (0,1).$ Considering the maximal family $\tilde u_\tau,\tau \in (1,2)$, we have two possibilities. If $l=1$, i.e. $\partial \mathcal{M}$ has only one boundary component, then for $\tau$ close to $2$, $\tilde u_\tau$ coincides with some $\tilde u_{\tau'}$ for some $\tau'<1$ close to $1$. Indeed, this follows from the uniqueness of cylinders $\tilde u_1$ from $P_3^p$ to $P_{2,1}$ and the uniqueness of planes obtained from gluing $\tilde u_1$ with the hemispheres of $\S_1$, see Theorem \ref{thm_uniqueness}. If $l>1$, then, again by the uniqueness of cylinders and planes, the family $\tilde u_\tau$ as $\tau \to 2$ cannot break onto $\tilde u_1$ and a plane projecting to $\S_j$. Indeed, otherwise, the projection of $u_\tau,$ $\tau<1$ sufficiently close to $1$ is homotopic to $u_{\tau'},$ $\tau'>1$ sufficiently close to $1$. This is not possible for topological reasons.  Hence the family $\tilde u_\tau, \tau\in(1,2)$ breaks into a new cylinder $\tilde u_2 \neq \tilde u_1$ connecting $P_3^p$ to $P_{2,i}\subset \S_i,$ for some $i\neq j$, and a plane $\tilde v_2$ asymptotic to $P_{2,i}$ projecting to $\S_i$. We can continue this procedure of gluing $\tilde u_2$ with the opposite plane $\tilde v_2'$ projecting to the other hemisphere of $\S_i\setminus P_{2, i}$, to extend the family of planes asymptotic to $P_3^p$. The number of breaks is necessarily equal to $l$, and we end up finding $l$ families of embedded planes $\{\tilde u_\tau\}_{\tau\in (i,i+1)}, $ $i=1,\ldots,l$, all asymptotic to $P_3^p$, their projections are also embedded and mutually disjoint. We also obtain $l$ embedded cylinders $\tilde v_i,i=1,\ldots,l,$ connecting $P_3^p$ at the positive puncture to a distinct $P_{2,j}\subset \S_j$ at the negative puncture. Their projections are embedded, and the union of such curves, with the ones projecting to $\partial \mathcal{M}$, determines a finite energy foliation of $\R \times \mathcal{M}$ whose projection to $\mathcal{M}$ is a weakly convex foliation. Notice that the families of planes asymptotic to $P_3^p$ and the cylinders connecting $P_3^p$ to each $P_{2,i}$ for an open and close subset of $\M \setminus \partial \M$ thus foliate the whole $\M \setminus \partial \M$.

\subsection{Passing to the degenerate case}

Let $\mathcal{M}, \alpha, \P$ and $J$ be as in Theorem \ref{main1}. In particular, the hemispheres of $\S_i\setminus P_{2,i}$ are the projections of two holomorphic planes $\tilde u_{i,1},\tilde u_{i,2}$, for every $i=1,\ldots, l$. After a $C^\infty$-small perturbation of $J$ supported in $\mathcal{M} \setminus \partial \mathcal{M},$ we may assume that
condition H3 is satisfied. We denote by $\J_{\text{reg}}(\alpha)$ the space of $J$'s satisfying condition H3.

Now let $\alpha_n$ be a sequence of nondegenerate contact forms on $(\mathcal M,\xi_0)$, coinciding with $\alpha$ on a small neighborhood of $\partial \mathcal{M}$ and converging in $C^\infty$ to $\alpha$ so that $P_3$ is a Reeb orbit of $\alpha_n$ for every $n$. Let $J\in \J_{\text{reg}}(\alpha)$ be such that the hemispheres in $\S_i\setminus P_{2,i}$ are the projections of $J$-holomorphic planes $\tilde u_{i,1},\tilde u_{i,2}$. Let $J_n\in \J(\alpha_n)$ be a sequence converging to $J$ in $C^\infty$ so that conditions H1-H3 are satisfied for $(\alpha_n, J_n)$, where in condition H2 the $p$-disk for $P_3$ is independent of $n$. We may assume that $J_n$ coincides with $J$ on a neighborhood of $\partial \mathcal{M}$. In particular, the hemispheres of $\S_i \setminus P_{2,i}$ are projections of $J_n$-holomorphic planes for every $n$ and $i$.

\begin{prop}\label{prop_alpha_n}Under the conditions above, the following holds: for every $n$ sufficiently large, $(\alpha_n, J_n)$ satisfies the condition $\mathcal{P}'=\emptyset$ in Theorem \ref{main1}, i.e., $\alpha_n$ does not admit a contractible periodic orbit $P'\subset \mathcal{M} \setminus (\partial \mathcal{M} \cup P_3)$ unlinked with $P_3$, with rotation number $1$ and action $\leq \S(\mathcal{D}, \alpha_n)$. In particular, $(\alpha_n, J_n)$ admits a finite energy foliation projecting to a weakly convex foliation whose binding orbits are $P_3\in \mathcal{M}\setminus\partial \mathcal{M}$ and $P_{2,1},\ldots, P_{2,l}\subset \partial \mathcal{M}.$
\end{prop}

\begin{proof}
    Fix a $p$-disk $\mathcal{D}$ for $P_3$ so that the characteristic foliation $(\xi_0 \cap T\mathcal{D})^\perp$ has a unique singularity, which is nicely elliptic. Since $\alpha_n \to \alpha$,  we have $\S(\mathcal{D},\alpha_n) \to \S(\mathcal{D}, \alpha)$ as $n\to \infty$. We claim that for every $n$ sufficiently large, $\alpha_n$ does not admit a contractible index-$2$ Reeb orbit $P_n \subset \mathcal{M} \setminus \partial \mathcal{M}$ with action $\leq \S(\alpha_n, \mathcal{D})$ and unlinked with $P_3$.  By contradiction, assume that such an orbit exists for $n$ arbitrarily large. Up to a subsequence $P_n \to Q$, where $Q\subset \mathcal{M}$ is a contractible periodic orbit of $\alpha$ with rotation number $1$ and action $\leq \S(\mathcal{D},\alpha)$. Since each $P_{2,i}\subset \partial \mathcal{M}$ is a hyperbolic orbit of $\alpha$, $Q$ cannot coincide with some $P_{2,i}$, and thus $Q\subset \mathcal{M} \setminus \partial \mathcal{M}$. Also, $Q$ is not a cover of $P_3$ since the rotation number of any contractible cover of $P_3$ is $>1$. Hence $Q\subset \mathcal{M} \setminus (\partial \mathcal{M} \cup P_3)$ has rotation number $1$ and is unlinked with $P_3$, contradicting the hypotheses on $\alpha$.
\end{proof}

Consider a sequence $\alpha_n \to \alpha$ and $J_n \to J$ as in Proposition \ref{prop_alpha_n}, and let $\F_n$ be the weakly convex foliation adapted to $(\alpha_n,J_n)$ given in Theorem \ref{main1} so that its binding orbits are $P_3,P_{2,1},\ldots,P_{2,l}$. Then the $J_n$-holomorphic planes projecting to $\partial \M$ do not depend on $n$ and hence we only need to consider the compactness properties of the family of planes asymptotic to $P_3^p$ and the cylinders connecting $P_3^p$ at the positive end to $P_{2, i}$ at the negative end. Before doing that, we need to show that $\alpha$ does not admit any periodic orbit which is unlinked with $P_3$, regardless of its index and knot-type.

\begin{prop}\label{prop_no_link}
    The contact form $\alpha$ admits no contractible periodic orbit $P'\subset \mathcal{M} \setminus (\partial \mathcal{M} \cup P_3)$ which is unlinked with $P_3$.
\end{prop}

\begin{proof} Suppose by contradiction the existence of a contractible periodic orbit $P'$ of $\alpha$ as in the statement. Then we can construct a sequence of nondegenerate contact forms $\alpha_n \to \alpha$ satisfying all conditions above with the additional property that $P'$ is a periodic orbit of $\alpha_n$ for every $n$. Choose $J_n \in \J_{\text{reg}}(\alpha_n)$ converging to $J$ as $n\to \infty$. Proposition \ref{prop_alpha_n} tells us that $(\alpha_n,J_n)$ admits a weakly convex foliation whose binding orbits are $P_3,P_{2,1},\ldots,P_{2,l}$ for $n$ sufficiently large. This implies, in particular, that every periodic orbit of $\alpha_n$ in $\mathcal{M}\setminus (\partial \mathcal{M} \cup P_3)$ is linked with $P_3$, a contradiction.
\end{proof}

Now, let us prove that rigid cylinders of $\F_n$ converge to a rigid cylinder for $(\alpha,J)$ with the same asymptotic limits. Denote by $\tilde v_n=(b_n,v_n):\R \times \R / \Z \to \R \times \mathcal M$ the nicely embedded $J_n$-holomorphic cylinder asymptotic to $P_3^p$ at its positive puncture $+\infty$ and to $P_{2,i}$ at its negative puncture $-\infty$. We know that $v_n(\R \times \R / \Z) \subset \mathcal{M} \setminus \partial \mathcal{M}$ for every $n$ and the leading eigenvalues at $+\infty$ and $-\infty$ have winding number $+1$. Fix a small compact tubular neighborhood $U \subset \mathcal M$ of $P_3$ satisfying the following conditions for $\alpha$
\begin{itemize}
    \item[(i)] $U$ has no periodic orbit which is contractible in $U$;
    \item[(ii)] There exists no periodic orbit $P'\subset U$ which is geometrically distinct from $P_3$, homotopic to $P_3^p$ in $U$ and unlinked with $P_3$.
    \end{itemize}
Parametrize $\tilde v_n$ so that
\begin{equation}\label{param_vn}
  \left\{   \begin{aligned}
     & v_n(\{0\} \times \R / \Z) \cap \partial U \neq \emptyset,\\
     & v_n(s,t) \in  U \setminus \partial U, \forall s>0,\\
     & b_n(1,0)=0.
     \end{aligned}\right.
 \end{equation}

 \begin{prop}\label{prop_cylinders_degenerate} Up to a subsequence, $\tilde v_n$ converges in $C^\infty_{\text{loc}}$ to a $J$-holomorphic cylinder $\tilde w_i=(b_i,w_i):\R \times \R / \Z \to \R \times \mathcal M$ asymptotic to $P_3^p$ at its positive puncture $+\infty$, whose leading eigenvalue has winding number $+1$, and to $P_{2,i}$ at its negative puncture $-\infty$.
 \end{prop}

 \begin{proof}  From the normalization \eqref{param_vn}, we can extract a subsequence of $\tilde v_n$ in $C^\infty_{\text{loc}}$ converging to a $J$-holomorphic punctured cylinder $\tilde w_i=(b_i,w_i):(\R \times \R / \Z)\setminus \Gamma\to \R \times \mathcal{M}$ which is necessarily asymptotic to $P_3^p$ at its positive puncture $+\infty$, see \cite[Lemma 3.12]{dPHKS}. The hypotheses on $U$ imply that no bubbling occurs on $\{s>0\}$. We must have $w_i((0,+\infty)\times \R / \Z) \subset U$ and thus $\Gamma \subset (-\infty, 0] \times \R / \Z$.  Also, $-\infty$ is necessarily a negative puncture of $\tilde w_i$. We shall prove that $\Gamma = \emptyset$ and $\tilde w_i$ is asymptotic to $P_{2,i}$ at $-\infty$.

 If the $d\alpha$-area of $\tilde w_i$ vanishes, then $\Gamma=\emptyset$ and $\tilde w_i$ is asymptotic to $P_3^p$ at $-\infty$. This follows from the fact that all asymptotic limits of $\tilde w_i$ are contractible and every iterate $P_3^j$, $1\leq j < p$, is non-contractible. However, in this case, we have $w_i(0,0) \in \partial U$, a contradiction. Hence, the $d\alpha$-area of $\tilde w_i$ is positive. 

 Let us prove that $\Gamma = \emptyset$ and $\tilde w_i$ is asymptotic to $P_{2,i}$ at $-\infty$. We know that $+\infty$ is a nondegenerate puncture. The leading eigenvalue of $A_{P_3^p}$ has winding number $+1$. Indeed, since $\mu(P_3^p) \geq 3$, the leading eigenvalue of $\tilde v_n$ at $+\infty$ has winding number $1$ and stays away from $0$ as $\alpha_n \to \alpha$. Hence, as proved in \cite{convex}, see also \cite[Chapter 8]{dPS1}, the limiting curve $\tilde w_i$ admits a negative leading eigenvalue whose eigenvector has winding number $+1$, which describes the asymptotic behavior of $\tilde w_i$ at $+\infty$ as in Theorem \ref{thm: half cylinder 1}.

 Any asymptotic limit at a negative puncture in $\Gamma\cup \{-\infty\}$ must be contractible in $\mathcal{M}$, its action is $<pT_3$, and it is unlinked with $P_3$. By Proposition \ref{prop_no_link}, the only orbits with these properties are the covers of orbits in $\partial \mathcal{M}$. Hence, every puncture of $\tilde w_i$ is nondegenerate. If $\# \Gamma \geq 1$, then the weighted index of $\tilde w_i$ is $\ind^\delta(\tilde w_i) \leq 3-2(\#\Gamma +1) -(2-\#\Gamma -2)=1-\#\Gamma\leq 0,$  contradicting the choice of $J\in \J_{\text{reg}}(\alpha)$, see condition H3. Hence $\Gamma = \emptyset$. To prove that $\tilde w_i$ is asymptotic to $P_{2,i}$ at $-\infty$, we invoke the SFT-compactness theorem. In fact, this is possible since any asymptotic limit of any $J$-holomorphic curve arising from rescaling $\tilde v_n$ near $-\infty$, which is not $P_3^p$, has action $<pT_3$, is contractible in $\mathcal{M}$, and is unlinked with $P_3$. Hence, it necessarily lies in $\partial \mathcal{M}$, and consists of a hyperbolic orbit. The asymptotic limit of $\tilde w_i$ at $-\infty$ necessarily has index $2$. Otherwise, the index of $\tilde w_i$ is $<0$, again a contradiction with the choice of $J$. If the asymptotic limit of $\tilde w_i$ is $P_{2,j},$  $j\neq i$, then there exists a nontrivial building $\mathcal{B}$ below $\tilde w_i$ formed by $J$-holomorphic curves asymptotic to covers of orbits in $\partial \mathcal{M}$, so that $\mathcal{B}$ has a positive puncture at $P_{2,j}, j\neq i,$ and a negative puncture at $P_{2, i}$. In particular, the first level of $\mathcal{B}$ consists of a $J$-holomorphic punctured sphere $\tilde w:\C \setminus \Gamma' \to \R \times \mathcal{M}$, which is not a trivial cylinder over $P_{2,j}$, and thus has positive $d\alpha$-area. It is asymptotic to $P_{2,j}$ at its unique positive puncture and to covers of orbits in $\partial \mathcal{M}$ at the negative punctures. However, by Proposition \ref{prop_inside}, such a curve $\tilde w$ does not exists.
 Hence the asymptotic limit of $\tilde w_i$ at $-\infty$ is $P_{2,i}$. \end{proof}

 Let us consider a sequence of $J_n$-holomorphic planes $\tilde u_n=(a_n,u_n):\C \to \R \times \mathcal{M}.$ As in the cylinder case, we may choose a small neighborhood $U\subset \mathcal{M}$ of $P_3$ as before, and parametrize $\tilde u_n$ so that
 \begin{equation}\label{param_un}
 \left\{\begin{aligned}
 & u_n(z_n^*), u_n(1) \in \partial U, \mbox{ where } \text{Re}(z_n^*) \leq 0.\\
 & u_n(\C \setminus \D) \subset U \setminus \partial U,\\
 & a_n(2) = 0.
 \end{aligned} \right.
 \end{equation}
This parametrization is obtained by considering the smallest disk in $\C$ that contains $u_n^{-1}(\mathcal{M} \setminus U)$ and suitably re-scaling it to $\D$.

\begin{prop}\label{prop_planes_degenerate}
Up to a subsequence, $\tilde u_n$ converges to one of the following buildings $\mathcal{B}$:
\begin{itemize}
\item[(i)] $\mathcal{B}$ has only one level, which is a $J$-holomorphic plane $\tilde u=(a,u):\C \to \R \times \mathcal{M}$ asymptotic to $P_3^p.$ Moreover, $\tilde u$ is an embedding, $u$ is an embedding transverse to the flow, $u(\C) \cap P_3 = \emptyset,$ and the leading eigenvalue of $\tilde u$ at $\infty$ has winding number $+1$.

\item[(ii)] $\mathcal{B}=(\tilde v,\tilde u)$ has two levels, the first level $\tilde v=(b,v):\R \times \R / \Z \to \R \times \mathcal{M}$ is an embedded $J$-holomorphic cylinder asymptotic to $P_3^p$ at its positive puncture and to some $P_{2,i}\subset \S_i$ at its negative puncture. Also, $v$ is an embedding transverse to the flow and $v(\R \times \R / \Z) \subset \mathcal{M} \setminus (\partial \mathcal{M} \cup P_3).$ The curve $\tilde v_2$ is a $J$-holomorphic plane projecting to a hemisphere of $\S_i \setminus P_{2,i}$.
\end{itemize}
 \end{prop}

\begin{proof}
From the normalization \eqref{param_un}, we know that up to a subsequence $\tilde u_n$ converges in $C^\infty_{\text{loc}}$ to a $J$-holomorphic curve $\tilde v=(b,v):\C \setminus \Gamma \to \R \times \mathcal{M},$ where $\Gamma\subset \D$ is a finite set of negative punctures, and $\tilde v$ is asymptotic to $P_3^p$ at $\infty$. Moreover, $+\infty$ is a nondegenerate puncture and the leading eigenvalue describing the approach of $\tilde v$ to $P_3^p$ has winding number $+1$ as in the proof of Proposition \ref{prop_cylinders_degenerate}. Also, the $d\alpha$-area of $\tilde v$ is positive. Indeed, if it vanishes, then $\Gamma\neq \emptyset$ and $\tilde v$ is a trivial cylinder over $P_3^p$. This contradicts $v(1) \in \partial U$. Hence, $\int_{\C \setminus \Gamma} v^* d\alpha >0$. Any asymptotic limit of a negative puncture in $\Gamma$ is contractible, has action $<pT_3$, and is unlinked with $P_3$. By Proposition \ref{prop_no_link}, such a Reeb orbit coincides with a cover of an orbit in $\partial \mathcal{M}$. As in the proof of Proposition \ref{prop_cylinders_degenerate}, if $\#\Gamma >1$, the weighted Fredholm index of $\tilde v$ is $\leq 0$, a contradiction with the choice of $J$, see condition H3. If $\Gamma=\emptyset,$ then $\tilde v$ is a plane as in (i). It is immediate that $\tilde v$ is an embedding, $v$ is an embedding transverse to the flow, and $v(\C) \subset \mathcal{M} \setminus (\partial \mathcal{M} \cup P_3)$ since the same properties hold for the curves $\tilde u_n$ in the sequence. If $\#\Gamma =1,$ then by the SFT-compactness theorem, $\tilde u_n$ converges up to a subsequence to a building $\mathcal{B}$  with at least two levels. The first level contains only $\tilde v=(b,v),$ as above. Also, $\tilde v$ is asymptotic to $P_3^p$ at $\infty$ and to some $P_{2,i}$ at its unique puncture in $\Gamma$.  The level below $\tilde v,$ consists of a $J$-holomorphic curve $\tilde v_2=(b_2, v_2):\C \setminus \Gamma' \to \R \times \mathcal{M}$, asymptotic to $P_{2,i}$ at $\infty$ and to other orbits in $\partial \mathcal{M}$ at its negative punctures in $\Gamma'$. Notice that $\tilde v_2$ is not a trivial cylinder and this fact implies that $\int_{\C \setminus \Gamma'} v_2^* d\alpha>0$.  If $\Gamma'\neq \emptyset,$ then $\tilde v_2$ does not coincide with a hemisphere of $\S_i \setminus P_{2,i}$, contradiction Proposition \ref{prop_inside}. Hence, $\Gamma'=\emptyset$ and the building $\mathcal{B}$ is as in (ii).
 \end{proof}

Propositions \ref{prop_cylinders_degenerate} and \ref{prop_planes_degenerate} imply that given $q\in \mathcal{M} \setminus (\partial \mathcal{M} \cup P_3)$, there exists either a $J$-holomorphic cylinder $\tilde v=(b,v)$ connecting $P_3^p$ to some $P_{2,i}\subset \partial \mathcal M$ whose projection to $\M$ passes through $q$, or a $J$-holomorphic plane $\tilde u=(a,u)$ asymptotic to $P_3^p$ with the same property. Both $\tilde v$ and $\tilde u$ admit a leading negative eigenvalue at $\infty$ whose winding number is $+1$. The weighted Fredholm theory developed for embedded curves implies that such a plane asymptotic to $P_3^p$ lies in a smooth family of planes whose projections to $\mathcal{M}$ are either identical or disjoint, see \cite{hryn, hryn2}. The cylinders $\tilde v$ are rigid, in the sense that they are isolated in the space of such cylinders. The orbit $P_3, P_{2,1}, \ldots, P_{2,l}$ together with the families of planes asymptotic to $P_3^p$, the rigid cylinders, and the curves projecting to $\partial \mathcal{M}$ cover the whole manifold $\mathcal{M}$ forming a weakly convex foliation with binding orbits $P_3, P_{2,1},\ldots, P_{2,l}$. In fact, by uniqueness of cylinders connecting $P_3^p$ to $P_{2,i}$, see Theorem \ref{thm_uniqueness}-(ii), there exist precisely $l$ rigid cylinders $\tilde v_i=(b_i,v_i):\R \times \R / \Z \to \R \times \mathcal{M}, i=1,\ldots,l,$ asymptotic to $P_3^p$ at the positive puncture $+\infty$ and to $P_{2,i}$ at the negative puncture $-\infty$. The family of planes $\tilde u_{i,\tau},i\in\{1,\ldots,l\},\tau \in (0,1)$ must break as $\tau \to 1^-$ onto a building consisting of a cylinder connecting $P_3^p$ to $P_{2,i}$ and $\tilde u_{i,1}$, which projects to $\S_i$. Gluing $\tilde v_i$ with $\tilde u_{i,2}$, we obtain a family of nicely embedded planes $\tilde u_{i+1,\tau},\tau \in (0,1),$ which coincide with the planes in $\tilde u_{i,\tau}$ if and only if $l=1$. If $l>1$, then this new family will break onto a building with a cylinder $\tilde v_{i+1}$ connecting $P_3^p$ to $P_{2,i+1}$ and a $J$-holomorphic plane $\tilde u_{i+1,1}$, after suitably labeling the rigid cylinders and rigid planes. We obtain a new family of planes by gluing $\tilde v_{i+1}$ with $\tilde u_{i+1,2}$. This stops after obtaining the $l$ families of planes, and the weakly convex foliation adapted to $(\alpha, J)$ is finally constructed. This completes the proof of Theorem \ref{main1}.

\subsection{Proof of Corollary \ref{cor_homoclinic}}

The proof of Corollary \ref{cor_homoclinic} follows from the ideas in \cite{fols}. Let $i\in \{1,\ldots,l\}$ be such that $P_{2,i}$ has the largest action $T_i>0$ among the actions of $P_{2,1}, \ldots, P_{2,l}$. Consider the local branch $W_i^u$ of the unstable manifold of $P_{2,i}$ that lies in the interior of $\M$. Consider the family of planes $D_\tau, \tau \in (0,1),$ in $\F$ so that as $\tau \to 0^+$, $D_\tau$ breaks into a hemisphere $U_i$ of $\S_i$ and a cylinder connecting $P_3^p$ to $P_{2,i}$. Assume that the Reeb vector field along $U_i$ points inward $\M$ and thus $\tau$ increases in the direction of the flow. In particular, $W^u_i$ intersects $D_\tau$ for $\tau>0$ sufficiently small, and this intersection is a smooth circle $C$ bounding a compact disk $B \subset D_{\tau}$, whose symplectic area is $T_i>0$. The forward flow pushes $B$ to a disk $B_1 \subset D_{\tau'}$, where $\tau'<1$ is arbitrarily close to $1$. Observe that for $\tau'<1$ sufficiently close to $1$, $D_{\tau'}$ intersects the local branch $W^s_{i+1}$ of the stable manifold of $P_{2,i+1}$ inside $\M$, and this intersection is an embedded circle $C'$ bounding a compact disk $B'\subset D_{\tau'}$. Since the symplectic area of $B'$ is $T_{i+1} \leq T_i,$  $B_1$ is not contained in the interior of $B'$. Hence, either $\partial B_1$ intersects $\partial B'$, or $B_1$ lies in $D_{\tau'} \setminus B'$, or $\partial B_1$ surrounds $B'$. In the first case, there exists a heteroclinic orbit from $P_{2,i}$ to $P_{2,i+1}$, which is a homoclinic orbit to $P_{2,i}$ if $P_{2,i}=P_{2,i+1}$. In the second and third cases, the forward flow of $C$ intersects the cylinder connecting $P_3^p$ and $P_{2,i+1}$ and thus also intersects the next family of disks $D_\tau, \tau\in (1,2)$, bounding a new disk $B_2\subset D_\tau$ with symplectic area $T_i$, for $\tau-1>0$ sufficiently small. We may repeat the argument with the intersection between the forward flow of $C$ and the local branch $W^s_{i+2}$ of the stable manifold of $P_{2,i+2}$. Either it gives a heteroclinic orbit from to $P_{2,i}$ to $P_{2,i+2}$ or the forward flow of $C$ intersects the rigid cylinder connecting $P_3^p$ to $P_{2,i+2}$ on an embedded circle and then we can push it to the next family $D_\tau, \tau \in (2,3).$ By contradiction, if the forward flow of $C$ does not intersect any local branch of the stable manifold of any $P_{2,j}$ then it will intersect infinitely many times some rigid cylinder $V\in \F$ connecting $P_3^p$ to some $P_{2,j}$. Denote by $C_n, n\in \N,$ the embedded circles in $V$ given by these intersections. We show that this leads to a contradiction since the symplectic area of $V$ is finite. In fact, if $C_n$ bounds a disk in $V$, then the symplectic area of such a disk is constant, equal to $T_i>0$. By the uniqueness of solutions, such disks must be disjoint, and thus only finitely many such disks are allowed since the symplectic area of $V$ is finite. If $C_n$ is a non-trivial circle in $V$, then $C_n$ and $P_{2,j}$ determine a half-cylinder in $V$ with constant symplectic area $T_i-T_j>0$. Hence, any two such embedded circles $C_n, C_m, m\neq n,$ must intersect. This is not allowed by the uniqueness of solutions. We conclude that there exists at most a finite number of intersections between the forward flow of $W^u_i$ with $V$. This implies that the forward flow of $C$ eventually intersects the local branch of the stable manifold of some $P_{2,j}$, producing a heteroclinic orbit if $j\neq i$ or a homoclinic orbit if $j=i$. This finishes the proof of Corollary \ref{cor_homoclinic}.

\section{The circular planar restricted three-body problem}
The planar circular restricted three-body problem is the study of the dynamics of the following time-dependent Hamiltonian
\begin{equation}\label{equ: Hamiltonian}
E_\mu(q,p,t)=\frac{1}{2}|p|^2-\frac{\mu}{|q-m(t)|}-\frac{1-\mu}{|q-e(t)|},
\end{equation}
where $q=q_1+iq_2\in \C \equiv \R^2$ is the position of the massless satellite, $p=p_1+ip_2$ is its momentum, $m(t)$ is the position of the moon with mass $0<\mu<1$, and $e(t)$ is the position of the earth with mass $1-\mu$. Hamilton's equations of $E_\mu$ are equivalent to
$$
\begin{aligned}
\ddot q = -\mu \frac{q-m(t)}{|q-m(t)|^3} - (1-\mu)\frac{q-e(t)}{|q-e(t)|^3}.
\end{aligned}
$$
 The moon and the earth move as in the Kepler problem. We  fix their center of mass  at $0$,  that is  $(1-\mu)e(t)+\mu m(t)=0, \forall t,$ and assume that $e(t)$ and $m(t)$ move along ellipses with eccentricity $c\in[0,1]$. We restrict to the circular case  $c=0$, so that
\begin{equation}\label{equ_circular}
e(t)=-\mu e^{-it} \quad \mbox{ and } \quad m(t)=(1-\mu)e^{-it}.
\end{equation}

In the rotating system $\hat q =  q e^{it}$, the earth and the moon are fixed at $-\mu$ and $1-\mu$ respectively. We obtain
\begin{equation}\label{equ_second_order}
\ddot{ \hat q} =  -\mu \frac{\hat q-(1-\mu)}{|\hat q-(1-\mu)|^3} - (1-\mu)\frac{\hat q+\mu}{|\hat q+\mu|^3} +2\dot {\hat q} i +\hat q.
\end{equation}
Denoting $\hat q$ again by $q$, we see that the solutions of the autonomous Hamiltonian
\begin{equation}\label{equ: hamiltonian autonomous}
\begin{aligned}
H_\mu(p,q)&:=\frac{1}{2}|p|^2-\frac{\mu}{|q-(1-\mu)|}-\frac{1-\mu}{|q+\mu|}+q_1p_2-q_2p_1,
\end{aligned}
\end{equation}
recover the trajectories of \eqref{equ_second_order}. Here, $1-\mu \in \C$ and $-\mu\in \C$ are the fixed positions of the moon and the earth in the rotation system, respectively. This system is referred to as the \textbf{circular planar restricted three-body problem}.

We briefly abbreviate the effective potential of $H=H_\mu$ by
$$
U(q):=U_0(q)-\frac{1}{2}|q|^2=-\frac{\mu}{|q-(1-\mu)|}-\frac{1-\mu}{|q+\mu|}-\frac{1}{2}|q|^2,
$$ where the dependence on $\mu$ in the notation is omitted.  Hamilton's equations of $H$ become
\begin{equation}\label{equ: motion of H}
\left\{\begin{aligned}
\dot q_1=p_1-q_2,\quad \quad \dot p_1=-\partial_{q_1}U_0-p_2=-\partial_{q_1}U-p_2-q_1,\\
\dot q_2=p_2+q_1, \quad \quad \dot p_2=-\partial_{q_2}U_0+p_1=-\partial_{q_2}U+p_1-q_2,
\end{aligned}\right.
\end{equation}
or, equivalently,
\begin{equation}\label{equ: motion of H2}
\left\{\begin{aligned}
\ddot q_1&=-\partial_{q_1}U_0-2\dot q_2+q_1=-2\dot q_2-\partial_{q_1} U,\\
\ddot q_2&=-\partial_{q_2}U_0+2\dot q_1+q_2=2\dot q_1-\partial_{q_2}U.
\end{aligned}\right.
\end{equation}
Notice that $F(q,\dot q):=|\dot q|^2/2+U(q)$ is a conserved quantity of \eqref{equ: motion of H2}, and equals $H(q,p)$ under the correspondence $(q,\dot q=p+iq)\leftrightarrow(q,p)$. For any given $h\in \R$, the set $\Omega_h:=\{U(q)\leq h\} \subset \C \setminus \{m,e\},$ is called the Hill region at energy $h$. The boundary $\partial \Omega_h$ is called the zero velocity curve since $\dot q=0$ if $F(q,\dot q)=h$ and $q\in \partial \Omega_h$.

For any given $0<\mu<1$, the Hamiltonian $H$ has five critical points $l_1,\ldots,l_5$, corresponding to equilibrium points of
\eqref{equ: motion of H}. The respective projections $\hat l_1,\hat l_2,\hat l_3$ of $l_1$, $l_2$ and $l_3$ to the $q$-plane are co-linear with $e$ and $m$, and $\hat l_1$ lies in between the earth and the moon.  If $0<\mu<1/2,$ then $\hat l_2$ lies to the right of the moon and $\hat l_3$ to the left of the earth. The points $l_1,l_2,l_3$ are saddle-center equilibrium points of $H$, which means that the linearized vector field at those points admits a pair of real eigenvalues and a pair of purely imaginary eigenvalues. Writing
$\hat l_1 = 1-\mu - r_1(\mu)$, we have
\begin{equation}\label{equ: r1, r2 and r3}
\begin{aligned}
&\frac{1-\mu}{(1-r_1)^2}-\frac{\mu}{r_1^2}=1-\mu-r_1.
\end{aligned}
\end{equation}
We can explicitly find the inverse function
\begin{equation}\label{equ: mu of r1}
\mu(r_1)=\frac{r_1^3(3-3r_1+r_1^2)}{1-2r_1+r_1^2+2r_1^3-r_1^4},
\end{equation}
and $r_1\in (0,1)$ can be regarded as a parameter replacing $\mu$.
The projections $\hat l_4,\hat l_5$ of $l_4,l_5$ to the $q$-plane are symmetric with respect to the $q_1$-axis, and each one of them forms an equilateral triangle with the primaries.

\begin{figure}[hbpt]
    \centering
    \includegraphics[width=0.4\linewidth]{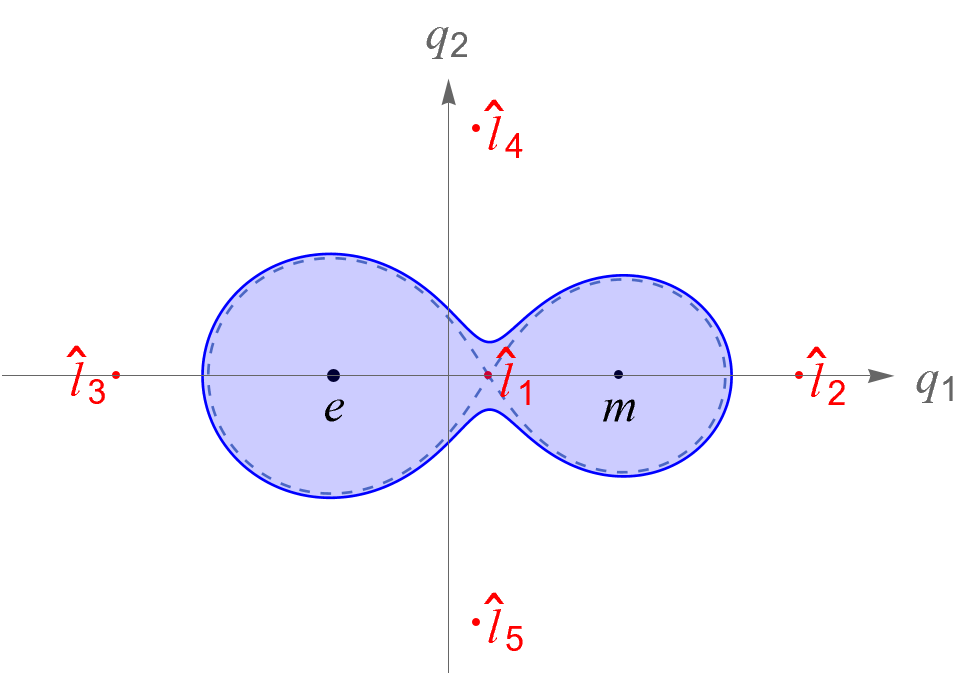}
    \caption{The projection of $\mathcal{M}_{\mu,E}^{e\# m}$ to the $q$-plane for $E$ slightly above $L_1(\mu)$.}
    \label{fig:Hill regions}
\end{figure}

Denote the critical values of $H$ by $L_i(\mu):=H(l_i),\ i=1,\ldots,5$. They satisfy $L_1 < L_2 \leq L_3 <L_4=L_5$ and are called Lagrange values. The second inequality is an equality if and only if $\mu=1/2$.  We study the dynamics on $H^{-1}(E)$, where $E$ is up to slightly above the first Lagrange value $L_1(\mu)$. We are interested in the dynamics on the component of $H^{-1}(E)$ that projects to the the disk-like region with two punctures at the primaries,  and has a small neck near $\hat l_1$, see Figure~\ref{fig:Hill regions}.


\section{Proof of Theorem \ref{thm_retrograde}}

For every mass ratio and energy below the first Lagrange value, Birkhoff used the shooting method to prove the existence of a retrograde in the bounded components of the energy surface around each primary.

\begin{thm}[Birkhoff \cite{Birkhoff1915}] \label{thm:Birkhoff1915}
Fix $0<\mu<1$. Then for every energy  $E< L_1(\mu)$, $\mathcal{M}_{\mu,E}^e$ admits  a $q_2$-symmetric retrograde orbit. A similar statement holds for $\mathcal{M}^m_{\mu,E}.$
\end{thm}

Theorem \ref{thm:Birkhoff1915} also holds for energies below the critical value $-9/2$ in the Hill's lunar problem, and the retrograde orbit is both symmetric in $q_1$ and $q_2$, see Section 8.3 in \cite{FvK}. 
We briefly explain the proof of Theorem \ref{thm:Birkhoff1915} and later show how it can be adapted to energies slightly above the first Lagrange value. We start replacing coordinates $(p,q)$ with $(p+i\mu,q-\mu)$ so that the primaries stay at $0, 1\in \C$. The Hamiltonian becomes
$$
H(p,q)=\frac{1}{2}((p_1-q_2)^2+(p_2+q_1)^2)+U(q),
$$
where
$$
U(q)=-\frac{\mu}{|q-1|}-\frac{1-\mu}{|q|}-\frac{1}{2}|q-\mu|^2.
$$

Let us consider two families of trajectories with initial conditions in $\{(0,p_1,q_1,0), p_1,q_1<0\}$ and $ \{(0,p_1,q_1,0), p_1,q_1>0\}$ at energy $E$. Using Levi-Civita coordinates to regularize collisions with the primary at $0$, one can show that for every $-q_1(0)>0$ sufficiently small, there exists $t_0>0$, depending on $q_1(0)$, such that $q_1(t_0)=0$,
$q_2(t)<0$ and $\dot q_1(t)>0$ for every $t\in(0,t_0]$. Moreover, $\mathrm{arg}(\dot q(t_0))\to -\pi/4$ as $q_1(0)\rightarrow 0^-$. If $q_1(0)>0$ is sufficiently small, then similar properties hold for the backward flow $t\in [-t_0,0)$, where $t_0>0$ depends on $q_1(0)$. In particular, $q_1(-t_0)=0$ and $\mathrm{arg}(\dot q(-t_0))\rightarrow \pi/4$ as $q_1(0)\to 0^+$. Let $\bar q_2>0$ be the maximal value of $q_2$ satisfying $E=U(0,\bar q_2)$. The families of solutions above induce real-analytic curves in the interior of the rectangle $Q:=[-\pi/2,\pi/2]\times [-\bar q_2,0]$, given by
$$
\Gamma_1:=\{(\mathrm{arg}(\dot q(t_0)),q_2(t_0)), q_{1}(0)\in(-\epsilon_0,0)\},\quad
\Gamma_2:=\{(\mathrm{arg}(\dot q(-t_0)),q_2(-t_0)),q_{1}(0)\in(0,\epsilon_0)\},
$$
where $\epsilon_0>0$ is sufficiently small.

For every energy $E$ up to slightly above $L_1(\mu)$, there exists the least value $\bar q_1<0$ of $q_1$ in the bounded component of the Hill region around $(0,0)$ such that $E=U(\bar q_1,0)$. Moreover, using Levi-Civita coordinates again, one shows that for every $q_{1}(0)-\bar q_1>0$ sufficiently small, the trajectory $q(t)$  returns to the $q_1$-axis without touching the $q_2$-axis. By continuity, there exists a minimal $\bar q_{1}^* \in (\bar q_1,0)$ such that every trajectory $q(t)$ with energy $E$ and initial conditions $q(0)\in (\bar q_{1}^*,0)\times \{0\}$, $\dot q_2(0) <0, \dot q_1(0)=0,$ satisfies
$$
q(t)|_{t\in(0,t_0)}\in\{q_1,q_2< 0\},\quad \dot q_1(t)|_{t\in(0,t_0]}> 0,\quad q_1(t_0)=0.
$$
If $q_1(0)=\bar q_1^*$, then at least one of the following conditions holds
\begin{itemize}
\item[(a)] $q(t_0)=(0,0)$.
\item[(b)]  $\exists\ t_1\in (0,t_0)$ so that $q_1(t_1)<0$ and $q_2(t_1)=\dot q_2(t_1)=0$.
\item[(c)]  $\exists\ t_1\in (0,t_0)$ so that $q_1(t_1),q_2(t_1)<0$, $\dot q_1(t_1)=0$.
\item[(d)] $\dot q_1(t_0)=0$,
\end{itemize}
where $t_0$ depends on $q_1(0)$. Cases (b) and (c) can be ruled out by a monotonicity argument proved by Birkhoff in \cite{Birkhoff1915}. In case (d), we necessarily have $\dot q_2(t_0)>0$.
Hence, we can extend $\Gamma_1$ to a real-analytic curve defined on the interval $(\bar q_1^*,0)$ denoted
$$
\Gamma_1=\{(\mathrm{arg}(\dot q(t_0)),q_2(t_0)), q_{1}(0)\in(\bar q_1^*,0)\}\subset Q\setminus\partial Q.
$$
If $q_{1}(0)\to 0^+$, $\Gamma_1$ converges to $(-\pi/4,0)\in \partial Q$. If $q_{1}(0)\to \bar q_1^*$, then either $\Gamma_1$ converges to $(\pi/4+e,0)$ for some $0<e\leq \pi/4$ (case (a)), or to $(\pi/2,q_2(t_0))$ for some $-\bar q_2 < q_2(t_0) < 0$ (case (d)). This is all proved in \cite{Birkhoff1915}.

A similar argument holds for the family $\Gamma_2$. More precisely, there exists $\hat q_1^*\in(0,l_1(\mu)+\mu)$ such that the backward trajectory $q(t)$ with energy $E$ up to $L_1(\mu)$ and initial conditions $q(0)\in (0,\hat q_1^*)$, $\dot q_1(0)=0, \dot q_1(0)>0,$ satisfies
$$
q(t)|_{t\in(-t_0,0)}\in\{q_2< 0< q_1\},\quad \dot q_1|_{t\in[-t_0,0)}>0,\quad q_1(-t_0)=0,
$$
where $t_0$ depends on $q_1(0).$
Moreover, for $q_1(0)=\hat q_1^*$ at least one of the following cases happen:
\begin{itemize}
\item[(e)] $q(-t_0)=(0,0)$.

\item[(f)] $\exists \ t_1\in (-t_0,0)$ so that $q_1(t_1)>0$ and $q_2(t_1)=\dot q_2(t_1)=0$.
\end{itemize}
As in (b), case (f) can be ruled out. Hence, we extend $\Gamma_2$ to a real-analytic curve defined on the interval $(0, \hat q_1^*)$ denoted
$$
\Gamma_2=\{(\mathrm{arg}(\dot q(-t_0)),q_2(-t_0)),q_{1}(0)\in(0,\hat q_1^*)\}\subset Q\setminus \partial Q.$$
If $q_1(0)\to 0^+$, then $\Gamma_2$ converges to $(\pi/4,0)\in \partial Q$. If $q_1(0)\to \hat q_1^*$, then $\Gamma_2$ converges to $(-d-\pi/4,0)$ for some $0<d\leq\pi/4$ (case (e)).

By uniqueness of solutions, both $\Gamma_1$ and $\Gamma_2$ do not admit self-intersections and are properly embedded in $\dot Q:= Q \setminus \partial Q\equiv \R^2$. Hence, they separate $\dot Q$ into twp disjoint open components. Moreover, the starting and ending points of $\Gamma_1$ lie in different components of $\dot Q \setminus \Gamma_2$ and thus $\Gamma_1$ must intersect $\Gamma_2$. Each intersection point corresponds to a retrograde orbit.  Since both curves are real-analytic, the intersection points are isolated.
A crossing point between these two curves is a point where one of the curves, say $\Gamma_1$, changes the component of $\dot Q \setminus \Gamma_2$. Such an intersection point always exists and is stable in the sense that if $\mu$ and $E$ are perturbed, there exists a nearby intersection point for the new parameters. We conclude that at least one crossing point exists between $\Gamma_1$ and $\Gamma_2$.

\begin{figure}[hbpt]
    \centering
    \includegraphics[width=0.4\linewidth]{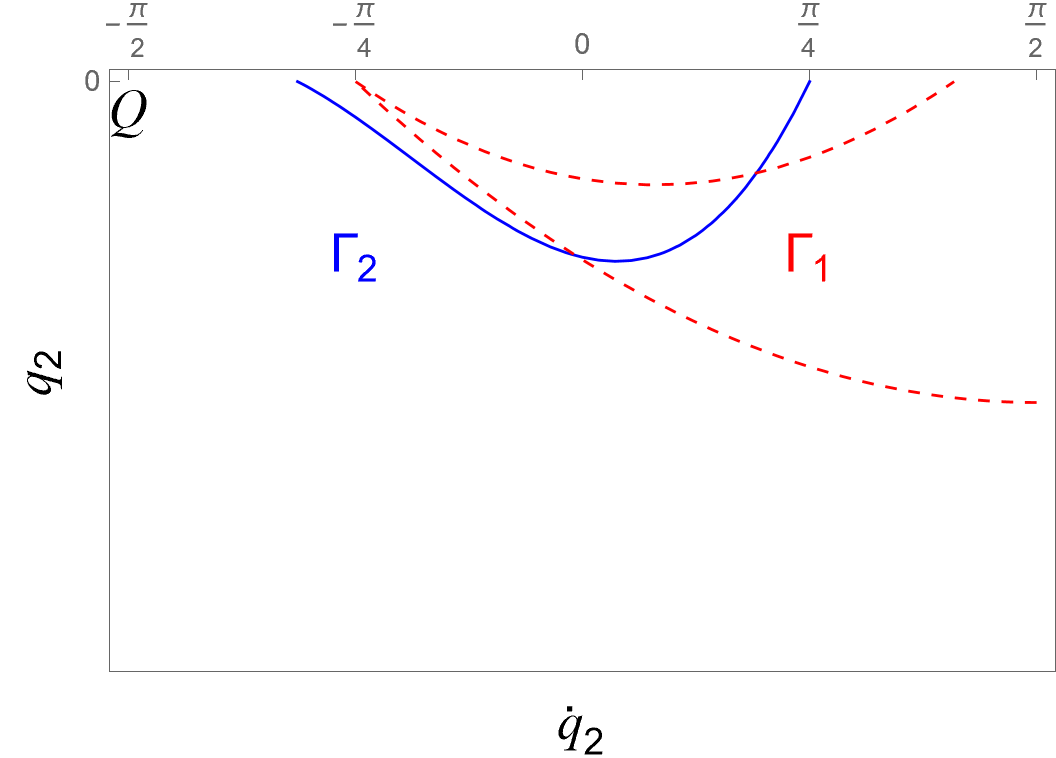}
    \caption{The curves $\Gamma_1$ and $\Gamma_2$ in $Q$}
    \label{fig:enter-label}
\end{figure}

In the following, we shall prove that Birkhoff's shooting method can be used to find retrograde orbits for energies slightly above the first Lagrange value. We start with the following lemma, which states that the endpoint $\hat q_1^*$ of the interval $(0,\hat q_1^*)$ defined above stays away from $l_1(\mu)$ as the energy increases to slightly above $L_1(\mu)$. In particular, the curves $\Gamma_1$ and $\Gamma_2$ are well-defined and share the same properties as in the case of lower energies.

\begin{lem}\label{lem: hat q2<1-r1-delta}
Fix $0<\mu_0<1$. There exists $\delta_0>0$ small so that for every $(\mu,E)$ sufficiently close to $(\mu_0, L_1(\mu_0))$ the solution in $H^{-1}(E)$ satisfying the initial conditions $q_1(0) = \hat l_1(\mu_0)+\mu_0 - \delta_0,$ $q_2(0)=0$, $\dot q_1(0)=0$ and $\dot q_2(0)>0$ satisfies the following conditions backward in time: there exists $t_0>0$ so that $q_2(-t_0)=0,$ $q_2(t) < 0 < \dot q_1(t)$ for every $t\in (-t_0,0)$, $0< q_1(-t_0)<q_1(0)$ and $\dot q_2(-t_0)<0$.
\end{lem}
\begin{proof}
It is enough to prove the lemma for $(\mu, E) = (\mu_0, L_1(\mu_0))$ and the statement follows from the transversality properties of the solution at the times where $q_2=0$. We look at the solutions of the linearized dynamics near $l_1(\mu_0)$ with initial conditions as in the statement. If $\delta_0$ is taken sufficiently small, we approximate the dynamics in $H^{-1}(L_1(\mu_0))$ by the linearized dynamics at $l_1(\mu_0)$, which is described by the quadratic Hamiltonian $H_2=H_2(\bar x)$ given in \eqref{H2}, where $\bar x:=\left((p+i\mu_0,q-\mu_0)- l_1(\mu_0)\right)(V^T)^{-1}$ and $V$ is the matrix given by \eqref{V1234b} so that $l_1(\mu)$ corresponds to$\bar x=0$. Hence,
$$
q(t)-(\mu_0 + \hat l_1(\mu_0),0)=\left(\frac{2\sqrt{\lambda_2}}{C_2C_0}\bar x_2+\frac{2\sqrt{\lambda_1}}{C_1C_0}\bar x_3,\frac{C_1}{\sqrt{\lambda_1}C_0}\bar x_1+\frac{C_2}{\sqrt{\lambda_2}C_0}\bar x_4\right)(t).
$$
The dynamics of $H$ is given by
$$
\begin{aligned}
\dot q(t)&=p(t)+iq(t)\\
&=\left(\frac{C_2^2+2a-2}{\sqrt{\lambda_1}C_1C_0}\bar x_1+\frac{C_1^2+2a-2}{\sqrt{\lambda_2}C_2C_0}\bar x_4,\frac{\sqrt{\lambda_2}(C_2^2-2)}{C_2C_0}\bar x_2+\frac{\sqrt{\lambda_1}(C_1^2-2)}{C_1C_0}\bar x_3\right)(t)\\
&+\left(-\frac{C_1}{\sqrt{\lambda_1}C_0}\bar x_1-\frac{C_2}{\sqrt{\lambda_2}C_0}\bar x_4,\frac{2\sqrt{\lambda_2}}{C_2C_0}\bar x_2+\frac{2\sqrt{\lambda_1}}{C_1C_0}\bar x_3\right)(t).
\end{aligned}$$

Since $q_1(0)=1-r_1(\mu_0)-\delta_0$, $\dot q_2(0)>0= q_2(0)=\dot q_1(0)$, we conclude that $\bar x_1(0)=\bar x_4(0)=0$ and $-\delta_0=\frac{2\sqrt{\lambda_2}}{C_2C_0}\bar x_2(0)+\frac{2\sqrt{\lambda_1}}{C_1C_0}\bar x_3(0)$. Recall that $C_2>C_1>0$ for every $\mu\in [0,1]$. Since $H(0,\bar x_2,\bar x_3,0)-L_1(\mu_0)=\frac{\lambda_2}{2}\bar x_2^2-\frac{\lambda_1}{2}\bar x_3^2+O(|\bar x_2|^3+|\bar x_3|^3)=0$, we see that
$$
\bar x_2(0)=\frac{ C_0C_1C_2}{2\sqrt{\lambda_2}(C_2-C_1)}\delta_0+O(\delta_0^2)>0,\ \
\bar x_3(0)=-\frac{C_0C_1C_2}{2\sqrt{\lambda_1}(C_2-C_1)}\delta_0+O(\delta_0^2)<0,$$
and
$$
\dot q_2(0)= \frac{\sqrt{\lambda_2}C_2\bar x_2(0) +\sqrt{\lambda_1}C_1\bar x_3(0)}{C_0}=\frac{C_1C_2}{2}\delta_0+O(\delta_0^2)>0,$$
for every $\delta_0>0$ sufficiently small.

Consider the re-scaling $\bar x=\epsilon^{1/2}x$ variable, we see that $H_\epsilon(x):=\epsilon^{-1}(H(\epsilon^{1/2} x)-L_1(\mu_0))$ converges in $C^\infty_{\text{loc}}$ to $H_2(x)$. We thus consider a solution $\tilde  x(t)$ in $H_2^{-1}(0)$ with initial conditions $\tilde x(0)=\frac{\delta_0 C_0C_1C_2}{2(C_2-C_1)}(0,\frac{1}{\sqrt{\lambda_2}},-\frac{1}{\sqrt{\lambda_1}},0)$ approximating the solution $x(t)$ in $H^{-1}_\epsilon(0)$ with initial condition $\bar x(0)$ mentioned above, where $H_2(\tilde x(t))=\frac{\lambda_1}{2}(\tilde x_1^2-\tilde x_3^2)+\frac{\lambda_2}{2}(\tilde x_2^2+\tilde x_4^2)(t)=0$. This is given by
$$
\begin{aligned}
\tilde x(t)=\frac{\delta_0 C_0C_1C_2}{2(C_2-C_1)}\bigg(-\frac{\sinh(\lambda_1t)}{\sqrt{\lambda_1}},\frac{\cos(\lambda_2t)}{\sqrt{\lambda_2}},
-\frac{\cosh(\lambda_1t)}{\sqrt{\lambda_1}},\frac{\sin(\lambda_2t)}{\sqrt{\lambda_2}}\bigg).
\end{aligned}
$$
In $q$-coordinates up to a re-scaling and a shifting, we write $\tilde x(t)$ as $\tilde q(t)$, i.e.,  $V\tilde x^T=(\tilde p,\tilde q)^T$. We thus obtain
$$
\tilde q(t)=\frac{\delta_0 C_1C_2}{2(C_2-C_1)}\left(\frac{2\cos(\lambda_2t)}{C_2}-\frac{2\cosh(\lambda_1t)}{C_1},-\frac{C_1\sinh(\lambda_1t)}{\lambda_1}+\frac{C_2\sin(\lambda_2t)}{\lambda_2}\right),
$$
with $\tilde q(0)=(-\delta_0,0)$ and $\dot{\tilde q}(0)=(0,\frac{C_1C_2}{2}\delta_0)$. Except for $t=0$, $\tilde q_2$ admits another zero at $-t_0\in(-\pi/\lambda_2,0)$. 
Therefore, we conclude that for every $\delta_0>0$ sufficiently small, the orbit $q(t)$ admits an intersection with $q_2$-axis at some time $-t_0\in(-\pi/\lambda_2,0)$ with $q_1(t_0)<-\delta_0$. To see that $\dot q_1|_{t\in(-t_0,0)}>0$, we need the following lemma from \cite{Birkhoff1915}.

\begin{lem}[Birkhoff {\cite[Section 17]{Birkhoff1915}}]\label{lem: partial_q1 U positive}
For every $0<\mu<1$, the inequality $\partial_{q_1}U>0$ holds on $B_\mu\cap \{0 \leq q_1 <l_1(\mu)+\mu\}$, where $B_\mu\subset \mathbb C\setminus\{-\mu\}$ is the projection of $\mathcal M^e_{\mu,L_1(\mu)}$ to the $q$-plane.
\end{lem}

From this lemma, we know that $\partial_{q_1}U(q)>0$ for every $q\in B_{\mu_0}\cap\{0\leq q_1< \hat l_1(\mu_0)+\mu_0\}$. Then $\frac{d}{dt}(\dot q_1+2q_2)=-\partial_{q_1}U<0$ whenever $0\leq q_1(t)<\hat l_1(\mu_0)+\mu_0$, for a solution $(p(t),q(t))\in \mathcal{M}_{\mu_0, L_1(\mu_0)}^e$. This means that $(\dot q_1+2q_2)(t)$ is a decreasing function on $t \in [-t_0,0)$. Since $\dot q_{1}(0)+2q_{2}(0)=0$, we conclude that $(\dot q_1+2q_2)(t)>0$ or equivalently $\dot q_1(t)>-2q_2(t)>0$ for every $t\in[-t_0,0)$.
Hence, this lemma holds.
\end{proof}

We are ready to complete the proof of Theorem \ref{thm_retrograde}. From the arguments above and Lemma \ref{lem: hat q2<1-r1-delta}, we know that both $\Gamma_1$ and $\Gamma_2$ are well-defined real-analytic curves for every $(\mu,E)$ sufficiently close to $(\mu_0, L_1(\mu_0))$. Indeed, in the definition of $\Gamma_2$, we need the existence of $\hat q_1^*$ satisfying the properties above. Lemma \ref{lem: hat q2<1-r1-delta} implies that such $\hat q_1^*$ exists for every $(\mu, E)$ sufficiently close to $(\mu_0, L_1(\mu_0))$, including energies $E$ greater than $L_1(\mu)$. The curves $\Gamma_1$ and $\Gamma_2$ do not self-intersect and analytically depend on $(\mu, E)$. As mentioned before, they admit at least one crossing point, which is isolated. Such a crossing point varies continuously with $(\mu,E)$. For each crossing point, there exists a $q_2$-symmetric retrograde orbit in $H^{-1}(E)$. If $(\mu,E) \to (\mu_0, L_1(\mu_0))$, the retrograde orbit for $(\mu,E)$ converges in $C^\infty$ to the one of $(\mu_0, L_1(\mu_0))$ since their initial conditions are arbitrarily close to each other. The proof of Theorem \ref{thm_retrograde} is now complete.

\section{Proof of Theorem \ref{thm_alphaJ}}
In this section, we prove Theorem \ref{thm_alphaJ}. We restate it for convenience.

\begin{thm}\label{thm_alphaJ_appendix}
Let $0<\mu_0<1$. Then for every $(\mu,E)$ sufficiently close to $(\mu_0, L_1(\mu_0))$, with $E> L_1(\mu)$, the following statements hold:
\begin{itemize}
\item[(i)] There exists a contact form $\alpha=\alpha_{\mu,E}=i_{Y}\omega_0$ on $\mathcal{M}^{e\# m}_{\mu,E}\equiv \R P^3 \# \R P^3$ whose Reeb flow is equivalent to the regularized Hamiltonian flow. Here, $Y=Y_{\mu, E}$ is a Liouville vector field, defined on a neighborhood of $\mathcal{M}_{\mu, E}^{e\#m}$ in $\R^4$ and transverse to $\mathcal{M}^{e\# m}_{\mu, E}$. Also, $\omega_0$ is the canonical symplectic form $\sum_i dp_i \wedge dq_i$.

\item[(ii)] There exists a compatible almost complex structure $J=J_{\mu,E}$ on $\R \times \mathcal{M}^{e\#m}_{\mu,E}$  adapted to $\alpha$ that admits a pair of $J$-holomorphic planes asymptotic to $P_{2,E}$ through opposite directions. The closure of their projections to $\mathcal{M}^{e\#m}_{\mu,E}$ form a regular two-sphere $\S=\S_{\mu,E}$ containing $P_{2,E}$. Furthermore, $\text{dist}(\S,l_1(\mu)) \to 0$  as $E \to L_1(\mu)^+$ uniformly in $\mu$.

\item[(iii)] There exists a contact form $\alpha = \alpha_{\mu,L_1(\mu)}=i_{Y_{\mu,L_1(\mu)}}\omega_0$ on the sphere-like singular subset $\dot \M^e:=\M_{\mu,L_1(\mu)}^e\setminus \{l_1(\mu)\}$ so that the contact forms $\alpha_{\mu,E}$ in (i) converge in $C^\infty_{\text{loc}}(\dot \M^e)$ to $\alpha_{\mu,L_1(\mu)}$ as $E \to L_1(\mu)^+$ uniformly in $\mu$. The same conclusion holds for $\dot{\mathcal{M}}^m:=\mathcal{M}_{\mu,L_1(\mu)}^m\setminus \{l_1(\mu)\}$.
\end{itemize}
\end{thm}

The proof has two main steps. In the first step, we study the limiting linear system near $l_1(\mu_0)$ for every $(\mu,E)$, $E>L_1(\mu),$ sufficiently close to $(\mu_0, L_1(\mu_0))$. Under a suitable re-scaling of variables, the limiting system becomes a standard linear system with a saddle-center equilibrium at the origin. It thus admits a natural Liouville vector field $Y_2$. For a suitable choice of almost complex structure $J$, a pair of $J$-holomorphic planes can be explicitly found, which are asymptotic to the Lyapunoff orbits through opposite directions and form a spherical shield. Finally, the automatic transversality of such $J$-holomorphic planes, see Theorem \ref{prop_automatic_transversality}, guarantees the existence of similar $J$-holomorphic planes for every $(\mu,E)$, $E>L_1(\mu)$, sufficiently close to $(\mu_0, L_1(\mu_0))$. In particular, this planes project arbitrarily close to $l_1(\mu_0)$. The second step is to construct a suitable interpolation between $Y_2$ and the natural Liouville vector fields $Y_e,Y_m$ centered at the earth and the moon, see in \cite{AFKP2012}.

\subsection{Spherical shields near the neck}\label{subsec: holo two sphere near l1}

In this section, we construct a pair of $J$-holomorphic planes $\tilde u_{1,E}=(a_{1,E},u_{1,E}),\tilde u_{2,E}=(a_{2,E},u_{2,E}):\C \to \R \times H^{-1}(E)$
asymptotic to the Lyapunoff orbit $P_{2,E}$ near $l_1(\mu)$, for $(\mu,E)$ sufficiently close to $(\mu_0, L_1(\mu_0))$, with $E>L_1(\mu)$.  Recall that
$H(p,q)=\frac{1}{2}((p_1-q_2)^2+(p_2+q_1)^2)+U(q),$ where
$$ U(q)=-\frac{\mu}{|q-(1-\mu)|}-\frac{1-\mu}{|q+\mu|}-\frac{1}{2}|q|^2.$$
The energy of the first Lagrange point is denoted
$$L_1(\mu)=U(\hat l_1(\mu))=-\frac{\mu}{r_1}-\frac{1-\mu}{1-r_1}-\frac{1}{2}(1-\mu-r_1)^2$$
where
$\hat l_1(\mu)=(1-\mu-r_1,0)$ denotes the projection of $l_1(\mu)$ to the $q$-plane and $r_1\in(0,1)$ solves \eqref{equ: r1, r2 and r3}. Given $\delta,\epsilon_0,c_0>0$, we define $\hat q_1^{\pm}:=1-\mu-r_1\pm \epsilon^{1/2}\delta$ and
\begin{equation}\label{equ: hat L}
\hat{\mathcal{L}}_\epsilon:=\left\{(p,q)\in \mathcal{M}_{\mu,L_1(\mu) + c_0 \epsilon}^{e\#m},\ q_1\in[\hat q_1^-,\hat q_1^+]\right\}.
\end{equation}

Let $(\bar p,\bar q)=(p,q)-l_1(\mu)=(p+i\hat l_1(\mu),q-\hat l_1(\mu))$ be shifted coordinates. The Hamiltonian $H$ near $l_1(\mu)$ can be written as
$$H(\bar p,\bar q)=L_1(\mu)+H_2(\bar p,\bar q)+R(\bar p,\bar q),$$
where
$H_2(\bar p,\bar q)=\frac{1}{2}\big((\bar p_1-\bar q_2)^2+(\bar p_2+\bar q_1)^2+(\nabla^2U(\hat l_1)\bar q,\bar q)\big),$ $R$ vanishes up to order $2$ at $l_1(\mu)$,
$$
\nabla^2U(\hat l_1)=\begin{pmatrix}-1-4a & 0 \\ 0 & -1+2a
\end{pmatrix},\quad a=\frac{1-\mu}{2(1-r_1)^3}+\frac{\mu}{2r_1^3}.
$$
Using \eqref{equ: r1, r2 and r3}, we compute
\begin{equation}\label{mua}
\begin{aligned}
a(r_1)=\frac{2-r_1+r_1^2}{1-2r_1+r_1^2+2r_1^3-r_1^4}.
\end{aligned}
\end{equation}
Moreover, $\max\{a(r_1), r_1\in[0,1]\}=a(1/2)=4$ and $\min\{a(r_1), r_1\in[0,1]\}=a(0)=a(1)=2$ as one readily checks.
We re-scale $(\bar p, \bar q)$ and obtain a new Hamiltonian in coordinates $(\hat p, \hat q)$
$$
H_\epsilon(\hat p,\hat q):=\epsilon^{-1}(H(\epsilon^{1/2}\hat p-i(1-\mu-r_1),\epsilon^{1/2}\hat q+(1-\mu-r_1))-L_1(\mu)).
$$
Notice that $H_\epsilon(\hat p, \hat q)$  converges in $C^\infty_{\text{loc}}$ to $H_2(\hat p,\hat q)$ as $\epsilon \to 0^+$. In coordinates $(\hat p,\hat q)$, the energy surface $H^{-1}(L_1(\mu)+c_0\epsilon)$ and the re-scaled region $\hat{\mathcal{L}_\epsilon}$ smoothly converge to $H_2^{-1}(c_0)$ and $H_2^{-1}(c_0)\cap\{\hat q_1\in[-\delta,\delta]\}$, respectively.

The dynamics of $H_2$ is determined by the linear system
$$
\left(\begin{array}{c} \dot {\hat p} \\ \dot {\hat q} \end{array} \right)=J_4\nabla^2H(l_1(\mu))\left(\begin{array}{c} \hat p \\  \hat q \end{array} \right)
=\left(\begin{array}{cc} J_2 & -\nabla^2U(\hat l_1(\mu))-I_2\\ I_2 & J_2\end{array}\right)\left(\begin{array}{c} \hat p \\ \hat q \end{array} \right),
$$
where $\lambda_1,\lambda_2>0$ satisfy
$$
\nabla^2H(l_1(\mu))=\left(\begin{array}{cc}
I_2 & J_2 \\ J_2^T & \mathrm{diag}(-4a,2a)
\end{array}\right).
$$
Here, $J_{2n}= \left(\begin{array}{cc}0_n & -I_n\\ I_n & 0_n\end{array}\right)$ is the standard complex matrix, and $0_n$ and $I_n$ are the $n\times n$ zero and identity matrices, respectively. The eigenvalues of $J_4\nabla^2H(l_1(\mu))$ are
$\pm \lambda_1$ and $\pm i\lambda_2$,
where
\begin{equation}\label{lambda12}
\begin{aligned}
\lambda_1^2=a-1+\sqrt{a(9a-4)}\geq 1+2\sqrt{7}\quad \mbox{ and } \quad
\lambda_2^2=1-a+\sqrt{a(9a-4)}\geq-1+2\sqrt{7}.
\end{aligned}
\end{equation}
Notice that $\lambda_1\geq \lambda_2$ since $a\geq 2$. The following vectors form a symplectic basis of $\R^4$
\begin{equation}\label{V1234}
\begin{aligned}
V_1 & := k_1\left(\frac{5a+\sqrt{a(9a-4)}}{2(1+4a)\sqrt{a(9a-4)}},0,0,\frac{2+3a-\sqrt{a(9a-4)}}{2(1+4a)\sqrt{a(9a-4)}}\right),\\
V_2 & :=k_2\left(0,\frac{3a}{2\sqrt{a(9a-4)}}+\frac{1}{2},\frac{1}{\sqrt{a(9a-4)}},0\right),\\
V_3 & :=k_1^{-1}\left(0,\frac{(1+4a)(3a-\sqrt{a(9a-4)})}{2+3a-\sqrt{a(9a-4)}},\frac{2(1+4a)}{2+3a-\sqrt{a(9a-4)}},0\right),\\
V_4 & :=k_2^{-1}\left(\frac{5a-\sqrt{a(9a-4)}}{2+3a+\sqrt{a(9a-4)}},0,0,1\right),
\end{aligned}
\end{equation}
where
$$\begin{aligned}
k_1& =\frac{\sqrt{2}(a(9a-4))^{1/4}(1+4a)}{(2+3a-\sqrt{a(9a-4)})^{1/2}(-1+a+\sqrt{a(9a-4)})^{1/4}},\\
k_2& =\frac{\sqrt{2}(a(9a-4))^{1/4}(1-a+\sqrt{a(9a-4)})^{1/4}}{(2+3a+\sqrt{a(9a-4)})^{1/2}}.
\end{aligned}$$
Moreover, $\R V_1\oplus \R V_3$ contains the eigenspaces of $\pm \lambda_1$  and $\R V_2\oplus \R V_4$ contains the generalized eigenspace of $\pm i \lambda_2$. We write $V:=\big(V_1^T,V_2^T,V_3^T,V_4^T\big)\in\mathrm{Sp}(4)$. Let
\begin{equation}\label{C123}
\begin{aligned}
C_0 & :=\sqrt{2}(a(9a-4))^{1/4}, \quad
C_1:=(2+3a-(a(9a-4))^{1/2})^{1/2}, \\ C_2 & :=(2+3a+(a(9a-4))^{1/2})^{1/2}.
\end{aligned}
\end{equation}
 Then $V$ gets simplified as
\begin{equation} \label{V1234b}
V=\begin{pmatrix}
\frac{C_2^2+(2a-2)}{\sqrt{\lambda_1}C_1C_0} & 0 & 0 & \frac{C_1^2+(2a-2)}{\sqrt{\lambda_2}C_2C_0}\\
0 & \frac{\sqrt{\lambda_2}(C_2^2-2)}{C_2C_0} & \frac{\sqrt{\lambda_1}(C_1^2-2)}{C_1C_0} & 0\\
0 & \frac{2\sqrt{\lambda_2}}{C_2C_0} & \frac{2\sqrt{\lambda_1}}{C_1C_0} & 0\\
\frac{C_1}{\sqrt{\lambda_1}C_0}&0&0&\frac{C_2}{\sqrt{\lambda_2}C_0}
\end{pmatrix}.
\end{equation}

Under the linear symplectic change of coordinates
$\Phi(x) := V x^T = (\hat p,\hat q)^T$, $H_2$ becomes
\begin{equation}\label{H2}
H_2=\frac{\lambda_1}{2}(x_1^2-x_3^2)+\frac{\lambda_2}{2}(x_2^2+x_4^2),
\end{equation}
and thus Hamilton's equations decouple
\begin{equation}\label{equ: linearized equation}
\left\{\begin{aligned}
\dot x_1&=\lambda_1x_3,\\
\dot x_3&=\lambda_1x_1,
\end{aligned}\right.\quad\text{and}\quad
\left\{\begin{aligned}
\dot x_2&=-\lambda_2x_4,\\
\dot x_4&=\lambda_2x_2.
\end{aligned}\right.
\end{equation}
The energy surface $H_2^{-1}(c_0)$, $c_0>0$, contains a unique Lyapunov orbit
\begin{equation}\label{equ: lyapunov orbit}
P_{0,L}=\left(0,\sqrt{2c_0/\lambda_2}\cos(\lambda_2 t),0,\sqrt{2c_0/\lambda_2}\sin(\lambda_2t)\right), \quad t\in \R / \frac{2\pi}{\lambda_2} \Z,
\end{equation}
which is hyperbolic and has index $2$.

Let $\Lambda:=\mathrm{diag}(1-b,1/2,b,1/2)$, where $0<b<1$. We consider the Liouville vector field
\begin{equation}\label{Y2}
Y_2:=\Lambda\cdot x\partial_x=(1-b)x_1\partial_{x_1}+\frac{1}{2}x_2\partial_{x_2}+bx_3\partial_{x_3}+\frac{1}{2}x_4 \partial_{x_4}.
\end{equation}
We compute on $H_2^{-1}(c_0)$
$$
\begin{aligned}
dH_2 \cdot Y_2 &=\lambda_1(1-b)x_1^2-\lambda_1bx_3^2+\frac{\lambda_2}{2}(x_2^2+x_4^2).
\end{aligned}
$$
Hence, $Y_2$ is transverse to $H_2^{-1}(c_0), c_0>0,$ if $|x_3|$ is sufficiently small.  Notice that $x_1=x_3=0$ implies $x_2^2+x_4^2>0$.
Under this condition, $Y_2$ induces a contact form $\alpha_2$ on $H_2^{-1}(c_0)$ near $x_3=0$, given by the restriction of
$$
\alpha_2:=\iota_{Y_2} \omega = -bx_3dx_1-\frac{1}{2}x_4dx_2+(1-b)x_1dx_3+\frac{1}{2}x_2dx_4.
$$
The contact structure $\xi = \ker \alpha_2$ is
spanned by $\xi_1:=X_1-A_1Y_2$ and $\xi_2:=X_2-A_2Y_2,$
where
$$
\begin{aligned}
A_1 & =\frac{dH_2\cdot X_1}{dH_2 \cdot Y_2},\quad A_2=\frac{dH_2 \cdot X_2}{dH_2 \cdot Y_2},\\
X_1 & =\frac{1}{2}x_4\partial_{x_1}-bx_3\partial_{x_2}+\frac{1}{2}x_2\partial_{x_3}-(1-b)x_1\partial_{x_4},\\
X_2 & =-\frac{1}{2}x_2\partial_{x_1}+(1-b)x_1\partial_{x_2}+\frac{1}{2}x_4\partial_{x_3}-bx_3\partial_{x_4}.
\end{aligned}
$$
The Reeb vector field of $\alpha_2$ is $R_2=(dH_2(Y_2))^{-1}X_{H_2}$. We choose the almost complex structure $J$ on $\mathbb{R}\times H_2^{-1}(c_0)$ such that $J \cdot \xi_1 = \xi_2$ and $J \cdot \partial_a = R_2$.

Next we construct a $J$-holomorphic cylinder
$\tilde u=(a,u):\mathbb{R}\times \R / \Z\to \mathbb{R}\times H_2^{-1}(c_0),$ asymptotic to $P_{0,L}$ at  $+\infty$, and with a removable singularity at $-\infty$. We have
$\pi \partial_t u = J \cdot \pi \partial_s u,$
$\partial_s a=\alpha_2(\partial_t u),$ and $ \partial_ta=-\alpha_2(\partial_su),$
where $\pi:TH_2^{-1}(c_0) \to \xi$ is the projection along $R_2$.
In coordinates $a\in \R$ and $(x_1,x_2,x_3,x_4)\in \R^4$, we assume that
$$
a(s,t)=a(s) \quad \mbox{ and } \quad u(s,t)=(x_1(s),r(s)\cos(2\pi t),0,r(s)\sin(2\pi t)),$$ where $x_1(s)^2=(2c_0-\lambda_2r(s)^2)/\lambda_1>0$ and $r(s)>0$. An extensive computation gives
$$
r'(s)=\frac{-2\pi r(\lambda_2r^2-2c_0)(\lambda_1r^2+4(b-1)^2(2c_0-\lambda_2r^2))}{(2c_0+(1-2b)(2c_0-\lambda_2 r^2))^2}=:g(r).
$$

Since $0<b<1$ and $0<r<r_0:=\sqrt{2c_0/\lambda_2}$, we have $2c_0+(1-2b)(2c_0-\lambda_2r^2)>0$, and thus $g=g(r)$ is well-defined on the interval $\left[-r_0, r_0\right]$. It vanishes at $r=0,\pm r_0$, is positive on $\left(0,r_0\right)$ and negative on $\left(-r_0,0\right)$. Assume that $r(0)\in \left(0,r_0\right)$ and $a(0)=0$. Then $r(s)\to r_0$ exponentially fast as $s\to +\infty$.
We also have
$a'(s)=\alpha_2(\partial_t u)=\pi r(s)^2\to \pi r_0^2=2\pi c_0/\lambda_2$ as  $s \to +\infty.$ Hence $s=+\infty$ is a positive end asymptotic to $P_{0,L}$. Notice that $x_1(s) \to 0$ as $s\to +\infty$. The end at $s=-\infty$ is removable since $r(s),a'(s) \to 0$ and $x_1(s) \to \sqrt{2c_0/\lambda_1}$ as $s\to -\infty$. In particular, $u(s,t)\to \big(\sqrt{2c_0/\lambda_1},0,0,0\big)$ as $s\to -\infty$. Hence, we obtain a $J$-holomorphic plane $\tilde u_{1}:\C \to \R \times H^{-1}_2(c_0)$ asymptotic to $P_{0,L}$. Its energy is $E(\tilde u_{1})=\pi r_0^2=2\pi c_0/\lambda_2$ and coincides with the action of $P_{0,L}$. Similarly, if $r(0)\in (-r_0,0),$ then we obtain another $J$-holomorphic plane $\tilde u_{2}:\C \to \R \times H^{-1}_2(c_0)$ asymptotic to $P_{0,L}$. Both $J$-holomorphic planes and their projections to $H^{-1}_2(c_0)\cap \{x_3=0\}$ are embedded and do not intersect each other. They approach $P_{0, L}$ through opposite directions. By uniqueness of such $J$-holomorphic planes, these are the only $J$-holomorphic planes asymptotic to $P_{0, L}$ up to reparametrization and $\R$-translation, see Theorem~\ref{thm_uniqueness}.
The discussion above implies the following proposition.
\begin{prop}\label{prop: holo S2 for Y}
Let $0<b<1$ and $c_0>0$. Then the following assertions hold on $H_2^{-1}(c_0)$:
\begin{itemize}
    \item[(i)] There exists a neighborhood $U_2$ of $H_2^{-1}(c_0)\cap \{x_3=0\}$ in $H_2^{-1}(c_0)$, so that $Y_2$ is transverse to $U_2$. Let $\xi:=\ker \alpha_2$, where $\alpha_2 := \iota_{Y_2} \omega|_{U_2}$ is the contact form on $U_2$ induced by $Y_2$, and let $R_2$ be the Reeb vector field of $\alpha_2$.
    \item[(ii)] There exists a compatible $\R$-invariant almost complex structure $J$ on $\R \times H_2^{-1}(c_0)$ satisfying $J\cdot \xi = \xi$ and $J \cdot \partial_a = R_2$, admitting a pair of embedded holomorphic planes $\tilde u_1=(a_1,u_1),\tilde u_2=(a_2,u_2):\mathbb{C}\to \mathbb{R}\times H_2^{-1}(c_0)$, asymptotic to $P_{0,L}$ through opposite directions, and satisfying $u_1(\C) \cup u_2(\C) = H_2^{-1}(c_0)\cap \{x_3=0\}$.
\end{itemize}
\end{prop}

We fix $0<\mu_0<1$, $c_0=1$ and $0<b<1$. The re-scaled energy surface $H_\epsilon^{-1}(1), \epsilon>0,$ near $l_1(\mu)$, converges in $C^\infty_{\text{loc}}$ to $H_2^{-1}(1)$ as $(\mu,\epsilon) \to (\mu_0, 0^+)$.  Since $P_{0,L}$ is hyperbolic, $H_\epsilon^{-1}(1)$ admits an index-$2$ hyperbolic orbit $P_{2,\epsilon}$ converging in $C^\infty$ to $P_{0,L}$ as $(\mu,\epsilon) \to (\mu_0,0^+)$. We may assume that there exists a small compact tubular neighborhood $\mathcal{U}\subset \R^4$ of $H_2^{-1}(1) \cap \{-\delta \leq x_3\leq \delta\}$, for some $\delta>0$ small, so that $Y_2$ is transverse to $H^{-1}_\epsilon(1) \cap \mathcal{U}$ for every
$(\mu,\epsilon), \epsilon>0,$ sufficiently close to $(\mu_0, 0)$. Denote by $\alpha_\epsilon=i_{Y_2}\omega_0|_{H^{-1}_\epsilon(1) \cap \U}$ the contact form on the neck-region $H^{-1}_\epsilon(1)\cap \mathcal{U}$ induced by $Y_2$. Up to a diffeomorphism $C^\infty$-close to the identity, we can see $\alpha_\epsilon$ as a contact form on $H_2^{-1}(1) \cap \mathcal{U}$  so that $\alpha_\epsilon \to \alpha_2$ in $C^\infty_{\text{loc}}$ as $(\mu,\epsilon) \to (\mu_0, 0^+)$. We may choose any almost complex structure $J_\epsilon=J_{\mu,\epsilon}$ on $\R \times (H_2^{-1}(1) \cap \mathcal{U})$ adapted to $\alpha_\epsilon$ so that $J_\epsilon \to J$ in $C^\infty_{\text{loc}}$ as $(\mu,\epsilon) \to (\mu_0, 0^+)$.
Since the curves $\tilde u_1,\tilde u_2$ are Fredholm regular, see Proposition \ref{prop_automatic_transversality}-(iii), we find for each $(\mu,\epsilon), \epsilon>0$ sufficiently close to $(\mu_0, 0)$, a pair of $J_\epsilon$-holomorphic planes $\tilde u_{1,\epsilon}=(a_{1,\epsilon}, u_{1,\epsilon})$, $\tilde u_{2,\epsilon}=(a_{2,\epsilon},u_{2,\epsilon}):\C \to \R \times H_\epsilon^{-1}(1)$ asymptotic to $P_{2,\epsilon}$.
Recall that $H_\epsilon=\epsilon^{-1}(H\circ \phi-L_1(\mu))$, where $\phi:(\hat p,\hat q)\mapsto (p,q)=\epsilon^{1/2}(\hat p,\hat q)+ l_1(\mu)$ induces a diffeomorphism between suitable subsets of $H_\epsilon^{-1}(1)$ and $H^{-1}(L_1(\mu)+\epsilon)$. We obtain $\hat J_\epsilon$-holomorphic curves $\tilde U_{i,\epsilon}=(A_{i,\epsilon},U_{i,\epsilon}):\C \to \R \times H^{-1}(L_1(\mu)+\epsilon)$, $i=1,2,$ given by
$$
\tilde U_{i,\epsilon}(s,t):=(\epsilon a_{i,\epsilon}(s,t), \phi(u_{i,\epsilon}(s,t)))=(\epsilon a_{i,\epsilon}(s,t),\epsilon^{1/2}u_{i,\epsilon}(s,t)+ l_1(\mu)), \quad  \forall (s,t)\in \mathbb{R}\times \R / \Z.
$$
Here, the almost complex structure $\hat J_\epsilon$ on $\R \times H^{-1}(L_1(\mu)+\epsilon)$ is the one induced by the contact form $\hat \alpha_\epsilon:= \epsilon(\phi^{-1})^*\alpha_\epsilon,$ and the push-forward $\phi_*J_\epsilon|_{\xi=\ker \alpha_\epsilon}.$ Indeed, we check that
$$
\begin{aligned}
\partial_s A_{i,\epsilon}& = \epsilon \partial_s a_{i,\epsilon} = \epsilon \alpha_\epsilon(\partial_t u_{i,\epsilon})  = \epsilon \alpha_\epsilon(D\phi^{-1}(U_{i,\epsilon}) \cdot \partial_t U_{i,\epsilon})\\ &  =
\epsilon (\phi^{-1})^*\alpha_\epsilon (\partial_t U_{i,\epsilon})  = \hat \alpha_\epsilon(\partial_t U_{i,\epsilon}),\\
\pi(\partial_s U_{i,\epsilon})&=D\phi(u_{i,\epsilon})\pi(\partial_s u_{i,\epsilon})=D\phi(u_{i,\epsilon})J_\epsilon\cdot \pi(-\partial_t u_{i,\epsilon})\\
&=(D\phi J_\epsilon D\phi^{-1})|_{U_{i,\epsilon}}\cdot \pi(-\partial_t U_{i,\epsilon})=(\phi_*J_\epsilon)|_{U_{i,\epsilon}}\cdot \pi(-\partial_t U_{i,\epsilon}).
\end{aligned}
$$
Denote $\omega = d\hat p \wedge d\hat q$ the standard symplectic form. Since $\phi_* Y_2=Y_2$ and $\phi^*\omega = \epsilon \omega$, we have $\hat \alpha_\epsilon = \iota_{Y_2} \omega|_{H^{-1}(L_1(\mu) + \epsilon)}$. As observed before, $\hat \alpha_\epsilon$ restricts to a contact form on $H^{-1}(L_1(\mu)+\epsilon)$ near $x_3=0$.
We have proved the following proposition.

\begin{prop}\label{prop: holo S2 for Y on H}
Fix $0<\mu_0<1$. Then for every $(\mu,\epsilon), \epsilon >0$, sufficiently close to $(\mu_0, 0)$, the following assertions hold:
\begin{itemize}
    \item[(i)] There exists a contact form $\hat \alpha_\epsilon=\hat \alpha_{\mu,\epsilon}$ on a small neighborhood $U_\epsilon$ of $\{x_3=0\}$ in $H^{-1}(L_1(\mu)+\epsilon)$ so that its Reeb vector field is parallel to the Hamiltonian vector field of $H$.

    \item[(ii)] There exists an almost complex structure $\hat J_\epsilon=\hat J_{\mu,\epsilon}$ on $\R \times H^{-1}(L_1(\mu) + \epsilon)$ near $x_3=0$, adapted to $\hat \alpha_\epsilon$, and  a pair of $\hat J_\epsilon$-holomorphic planes $\tilde U_{i,\epsilon}=(A_{i,\epsilon},U_{i,\epsilon}):\C \to \R \times H^{-1}(L_1(\mu)+\epsilon),i=1,2,$ asymptotic to the Lyapunoff orbit $P_{2,\epsilon} \subset H^{-1}(L_1(\mu)+\epsilon)$ through opposite directions.

    \item[(iii)] Given any neighborhood $\mathcal{U} \subset \R^4$ of $l_1(\mu_0)$, we have $U_{1,\epsilon}(\C) \cup U_{2,\epsilon}(\C) \subset \mathcal{U}$ for every $(\mu,\epsilon), \epsilon>0$ sufficiently close to $(\mu_0, 0)$.
    \end{itemize}
\end{prop}

\subsection{Interpolation of Liouville vector fields} \label{subset: interpolation of Liouville vector field}
Fix $0<\mu_0<1$. For every $(\mu,E)$ sufficiently close to $(\mu_0,L_1(\mu_0))$, we denote by $\mathcal{D}_\mu\subset \C$ the closed disk of radius $0<\hat l_1(\mu) + \mu<1$  centered at the earth $q=-\mu$. Notice that $\hat l_1(\mu)\in \partial \mathcal{D}_\mu =C_\mu:=\{ (q_1+\mu)^2+q_2^2=(\hat l_1(\mu)+\mu)^2\}$.  Recall that $B_\mu\subset \C\setminus\{-\mu\}$ is the projection of $\mathcal{M}^e_{\mu,L_1(\mu)}$ to the $q$-plane. Consider the Liouville vector field
$Y_e:=(q_1+\mu)\partial_{q_1}+q_2\partial_{q_2}$ in coordinates $(p,q)$ centered at the earth. We also consider the Liouville vector field $Y_m:=(q_1 -(1-\mu))\partial_{q_1} + q_2 \partial_{q_2}$ centered at the moon. The following proposition is essentially proved in \cite{AFKP2012} with a slightly different notation. 

\begin{prop}[Albers-Frauenfelder-van Koert-Paternain \cite{AFKP2012}]\label{prop_transversality of $Y_e$}
Consider polar coordinates $q=(-\mu+\rho\cos\theta,\rho\sin\theta), \rho>0, \theta \in \R / 2\pi \Z,$ centered at the earth $q=-\mu$. Then
\begin{itemize}
    \item[(i)] $B_\mu\setminus \{\hat l_1(\mu)\}\subset \mathcal{D}_\mu\setminus \partial \mathcal{D}_\mu$ and $\partial_\rho U>0$ on $\mathcal{D}_\mu\setminus \{-\mu,\hat l_1(\mu)\}$.

    \item[(ii)] For every $0<\rho<1$, the even function $\theta \mapsto U(-\mu+\rho\cos\theta,\rho\sin\theta)$ attains its  minimum at $\theta=0$ ($\theta =\pi$ is a local minimum), and its maximum at $\pm \hat \theta$, where $\hat \theta \in (0, \pi)$ satisfies $2\cos \hat\theta=\rho$, i.e., $\hat q:=(-\mu+\rho \cos \hat \theta, \rho \sin \hat \theta)$ is such that $|\hat q-(1-\mu)|=1$. Moreover, $\theta \mapsto U(-\mu+\rho\cos\theta,\rho\sin\theta)$ is strictly increasing on $[0,\hat \theta]$ and strictly decreasing on $[\hat \theta,\pi]$.

    \item[(iii)] For every $E<L_1(\mu)$, $Y_e$ is positively transverse to the unregularized component $\mathcal{M}^e_{\mu,E}\subset H^{-1}(E)$, i.e. $dH \cdot Y_e>0$ on $\mathcal{M}^e_{\mu,E}$.

    \item[(iv)] Let $E\geq L_1(\mu)$ and $\theta\in\R/2\pi\Z$. If there exists $\rho_\theta\in(0,\hat l_1(\mu)+\mu)$ such that $U(-\mu+\rho_\theta\cos\theta,\rho_\theta\sin\theta)=E$, then $dH \cdot Y_e>0$ on the unregularized component $\mathcal{M}_{\mu,E}^e\cap \{(p,q),q=(-\mu+\rho\cos\theta,\rho\sin \theta), 0<\rho \leq \rho_\theta\}$.
\end{itemize}
\end{prop}
\begin{rem}An analogous statement holds on the moon side via the symmetry $H_\mu(p,q) = H_{1-\mu}(-p,-q)$. \end{rem}

\begin{proof}
Item (i) follows from Corollary 5.3 and Lemma 5.4 in \cite{AFKP2012}, (ii) follows from Lemma 5.2 in \cite{AFKP2012}, (iii) follows from Proposition 5.1 in \cite{AFKP2012}, and (iv) follows from the proof of Proposition 5.1 in \cite{AFKP2012}.
\end{proof}

Let $\gamma(\theta):= (-\mu + (\hat l_1(\mu)+\mu)\cos \theta, (\hat l_1(\mu) + \mu)\sin \theta)), \theta \in \R / 2\pi \Z$, be a parametrization of $C_\mu$. Then $\gamma(0)=\hat l_1(\mu)$.
  By Proposition \ref{prop_transversality of $Y_e$}, the function $\theta \mapsto U(\gamma(\theta))$ attends  its minimum $L_1(\mu)$ precisely at $\theta=0$. Hence, for every $\epsilon:=E-L_1(\mu)>0$ sufficiently small, there exists $\theta_*=\theta_*(\epsilon)>0$ arbitrarily small so that
$$
\begin{aligned}
& U(\gamma(\pm\theta_*))  =L_1(\mu)+\epsilon,\\
& U(\gamma(\theta))  >L_1(\mu)+\epsilon,\quad \forall \theta_* < \theta < 2\pi-\theta_*,\\
& U(\gamma(\theta))  <L_1(\mu)+\epsilon,\quad \forall -\theta_* < \theta < \theta_*.
\end{aligned}
$$

\begin{prop}\label{prop_transversality2}Let $\mathcal{D}_\epsilon\subset \mathcal{D}_\mu$ be the open set $\mathcal{D}_\mu\setminus (\{-\mu\}\cup \partial \mathcal{D}_\mu)$ with  the annular sector $(\rho,\theta)\in [\mu + \hat l_1(\mu) - \epsilon^{1/2},  \mu + \hat l_1(\mu)] \times [-\theta_*,\theta_*]$ removed.
For every $\epsilon>0$ sufficiently small, $dH \cdot Y_e>0$ on $H^{-1}(L_1(\mu)+\epsilon)\cap \{(p,q),q\in \mathcal{D}_\epsilon\}.$
\end{prop}

\begin{proof}
Let $\epsilon>0$ be small. Since $U(\gamma(\theta))>L_1(\mu)+\epsilon,$ for every $\theta_* < \theta < 2\pi -\theta_*$,  Proposition \ref{prop_transversality of $Y_e$}-(i) and (iv) implies that $dH \cdot Y_e> 0$ on $H^{-1}(L_1(\mu)+\epsilon)\cap \{(p,q),q=(-\mu+\rho\cos\theta,\rho\sin \theta), 0<\rho <\hat l_1(\mu)+\mu,\theta_*< \theta< 2\pi -\theta_*\}.$ Hence, it remains to check that $dH \cdot Y_e>0$ on $H^{-1}(L_1(\mu)+\epsilon) \cap \{0<\rho< \mu + \hat l_1(\mu) -\epsilon^{1/2}, -\theta_*<\theta< \theta_*\}.$

Assume that $\theta\in[-\theta_*,\theta_*]$, and denote $p=\dot q-iq=(r_{\dot q}\cos\theta_{\dot q}+\rho\sin\theta,r_{\dot q}\sin\theta_{\dot q}+\mu-\rho\cos\theta)$. Then
$$
\begin{aligned}
dH\cdot Y_e & = (p_2+q_1+\partial_{q_1}U)(q_1+\mu)+(q_2-p_1+\partial_{q_2} U)q_2\\
 & = \rho(\partial_\rho U-r_{\dot q}\sin(\theta-\theta_{\dot q})) \geq \rho(\partial_\rho U - r_{\dot q})\\
& = \frac{1-\mu}{\rho}+\frac{\mu (\rho - \cos\theta)\rho}{(1+\rho^2-2\rho\cos\theta)^{3/2}} - \rho^2 + \mu \rho \cos\theta - \rho r_{\dot q},
\end{aligned}
$$
where $r_{\dot q}^2=2(L_1(\mu)+\epsilon-U)$.  For simplicity, we denote $s=\cos\theta$. 
Let
$$
\begin{aligned}
L(\rho,s)
& := \frac{1-\mu}{\rho}+\frac{\mu (\rho - s)\rho}{(1+\rho^2-2\rho s)^{3/2}}-\rho^2 + \mu \rho s - \rho \sqrt{2(\epsilon+L_1(\mu)- U)}.
\end{aligned}
$$
We compute
$$
\frac{1}{\mu \rho}\partial_{s}L=\frac{3 \rho (\rho - s)}{(1 + \rho^2 - 2 \rho s)^{5/2}} + 1 - \frac{1}{(1 + \rho^2 - 2 \rho s)^{3/2}}-\frac{\rho(1 + \rho^2 - 2 \rho s)^{-3/2}-\rho}{\sqrt{2(\epsilon+L_1(\mu)- U)}}.
$$

Recall that $\theta \in [-\theta_*,\theta_*]$. For $\epsilon>0$ sufficiently small, $\theta_*>0$  is arbitrarily close to $0$ and thus $s=\cos \theta$ is arbitrarily close to $1$. This implies that $1+\rho^2 - 2\rho s = 1-\rho(2s-\rho) \in (0,1)$ for every $\rho \in (0,1)$. Moreover, if $\epsilon>0$ is sufficiently small, we have $\rho < s=\cos \theta$ for every $\rho < 1-r_1=\mu + \hat l_1(\mu)<1$. We conclude that if $\epsilon>0$ is sufficiently small and $\theta\in [-\theta_*,\theta_*]$, then $\partial_sL<0$. This implies that $L$ is strictly increasing in $\theta\in[0,\theta_*]$, if $\epsilon>0$ is sufficiently small and $0<\rho < 1-r_1$. By symmetry, it suffices to check that $L>0$ when $s=1$ (or $\theta =0$) and $\rho \in (0, 1-r_1 - \epsilon^{1/2})$.

We fix $s=1$. Notice that $L/\rho = \partial_\rho U - \sqrt{2(L_1(\mu)+\epsilon - U)}$.
We know that $\partial_{\rho}U>0$ if $\rho\in(0,1-r_1)$. Hence $L>0$ if and only if
$F_1(\rho):=(\partial_\rho U)^2-2(L_1(\mu)+\epsilon-U)>0$.
After a straightforward computation using the expression \eqref{mua} for $\mu=\mu(r_1)$, we obtain
$$
F_1'(\rho)=-\frac{4(1-r_1-\rho)F_2(\rho)F_3(\rho)}{(1-\rho)^5\rho^5((1-r_1)^2 +r_1^3(2-r_1))^2},
$$
where
$$
\begin{aligned}
F_2(\rho)&=(1-\rho)^2((1-r_1)(1-r_1^3)+\rho(1-r_1^3)+\rho^2(1+r_1+r_1^2))\\
&+\rho^3(1-\rho+(2-\rho)r_1)r_1(3-3r_1+r_1^2),\\
F_3(\rho)&=(1-r_1)^2((1-3\rho+3\rho^2)(1-r_1^3)-\rho^3)+\rho^3r_1^3(4-5r_1+2r_1^2).
\end{aligned}
$$
It is immediate that $F_2(\rho)>0$ for every $0<\rho < 1-r_1$. We also have $F_3(\rho)>0$ for every $\rho\in(0,1-r_1)$, since  $F_4(\rho):=(1-3\rho+3\rho^2)(1-r_1^3)-\rho^3$ satisfies $F_4'(\rho)=-3 ((\rho-1+r_1^3)^2+r_1^3(1-r_1^3))<0$ and $F_4(1-r_1)=3(1-r_1)r_1^4>0$. Therefore,  $F_1$ is strictly decreasing on the interval $(0,1-r_1)$, $F_1(0^+)=+\infty$ and $F_1(1-r_1)=-2\epsilon<0$.

We still need to check that $F_1(1-r_1-\epsilon^{1/2})>0$. Expanding it in $\epsilon^{1/2}$ near $\epsilon=0$, we obtain
$$
F_1(1-r_1-\epsilon^{1/2})=\frac{2f(r_1)\epsilon}{((1-r_1)^2+r_1^3(2-r_1))^2}+O(\epsilon^{3/2}),
$$
where
$f(r_1)=5(7-8r_1+6r_1^2)+r_1(2+5r_1^3+3r_1^6)+r_1^2(14+6r_1^2+4r_1^3+r_1^5)(1-r_1)>0$ for $r_1\in(0,1)$. Hence $F_1(1-r_1-\epsilon^{1/2})>0$ for every $\epsilon>0$ sufficiently small. We conclude that if $\epsilon>0$ is sufficiently small, then $dH \cdot Y_e >0$ on $H^{-1}(L_1(\mu)+\epsilon)\cap \{(q,p),q\in \mathcal{D}_\epsilon\}$, and the proof is finished.
\end{proof}

Take $(\mu,\epsilon)$, $\epsilon>0$, close to $(\mu_0,0^+)$, and let $Y_2$ be the Liouville vector field defined in \eqref{Y2}, where  $b\leq 1/2$. We aim at interpolating $Y_2$ and $Y_e$ to obtain a Liouville vector field $Y_\epsilon$ which is positively transverse to $H^{-1}(L_1(\mu)+\epsilon)$ for every $\epsilon>0$ sufficiently small.

Consider $x$ coordinates as before, $Vx^T = (\hat p,\hat q)^T,$ where the symplectic matrix $V$ is given in \eqref{V1234b} and $(\hat p,\hat q)$ denote the re-scaled coordinates with origin at $l_1(\mu)$. The canonical symplectic form $\omega_0$ is given by $dx_1 \wedge dx_3 + dx_2 \wedge dx_4$. Let $c_0>0$. Then $dH_2 \cdot Y_2 = c_0 + \lambda_1(1/2-b)(x_1^2+x_3^2)>0$ on $H_2^{-1}(c_0)$, which implies that $Y_2$ is everywhere transverse to $H_2^{-1}(c_0)$.

In coordinates $x$, $Y_e$ becomes
$$
\begin{aligned}
Y_e & =(\bar q_1+(1-r_1))\partial_{\bar q_1}+\bar q_2\partial_{\bar q_2} = \hat q\partial_{\hat q} + \epsilon^{-1/2}(1-r_1) \partial_{\hat q_1}\\
&= (xQ_0 +\epsilon^{-1/2}(0,d_5,d_6,0))\partial_x^T,
\end{aligned}
$$
where
$$
Q_0  :=\left(\begin{array}{cccc}
1-d_1 & 0 & 0 & d_3 \\
0 & 1-d_2 & d_3 & 0 \\
0 & d_4 & d_1 & 0\\
d_4 & 0 & 0 & d_2
\end{array}\right),
$$
and
$$
\begin{aligned}
d_1 & :=1+\frac{C_1^2-2}{C_0^2}>0, \quad \quad  \quad  \quad \quad \quad \quad  d_2   :=\frac{C_2^2-2}{C_0^2}>0, \\
d_3 & :=\frac{(C_2^2-2)C_1\sqrt{\lambda_2}}{C_0^2C_2\sqrt{\lambda_1}}>0,
\quad \quad \quad \quad \quad   d_4  :=\frac{(2-C_1^2)C_2\sqrt{\lambda_1}}{C_0^2C_1\sqrt{\lambda_2}}<0, \\
d_5 & :=-\frac{(2(a-1)+C_1^2)(1-r_1)}{C_0C_2\sqrt{\lambda_2}}<0,
\quad  d_6 :=\frac{(2(a-1)+C_2^2)(1-r_1)}{C_0C_1\sqrt{\lambda_1}}>0.
\end{aligned}
$$
Here, $C_1,C_2,C_3$ are as in \eqref{C123}, and $a=a(r_1)$, see \eqref{mua}.

Since both $Y_e$ and $Y_2$ are Liouville vector fields, there exists a function $G$ defined near $x=0$ such that $\iota_{Y_e-Y_2}\omega=dG.$
In particular, $Y_2-Y_e=X_G$ and
$$
\nabla G=(Y_2-Y_e)J_4=(x(\Lambda-Q_0)-\epsilon^{-1/2}(0,d_5,d_6,0))J_4\partial_x^T,
$$
where $\Lambda:=\mathrm{diag}(1-b,1/2,b,1/2)$. Hence, we may choose
$
G(x)=\frac{1}{2}xQ_Gx^T+\epsilon^{-1/2}(-d_6x_1+d_5x_4),
$
where
\begin{equation}\label{QG}
Q_G=\begin{pmatrix}
0 & -d_3 & b-d_1 & 0\\ -d_3 & 0 & 0 & \frac{1}{2}-d_2\\
b-d_1 & 0 & 0 & d_4\\ 0 & \frac{1}{2}-d_2 & d_4 & 0
\end{pmatrix}.
\end{equation}

Let $\epsilon>0$ be small. By Proposition \ref{prop_transversality2}, we know that $dH\cdot Y_e>0$ for any $\epsilon^{1/2}x=(\bar p,\bar q)(V^T)^{-1}=((p,q)-l_1(\mu)) \cdot (V^T)^{-1}$, where $(p,q)\in H^{-1}(L_1(\mu)+\epsilon)$ with $|q+\mu|<\hat l_1(\mu)+\mu-\epsilon^{1/2}$. We want to interpolate  $Y_e$ and $Y_2$ in order to obtain a Liouville vector field $Y_\epsilon$ which is positively transverse to  $H^{-1}(L_1(\mu)+\epsilon)$.
We shall construct this interpolation in coordinates $x=(x_1,x_2,x_3,x_4)$.

Denote $\mathcal M^{e\# m}_{\mu,E}$ shortly as $\mathcal M^{e\# m}_{\epsilon}$, where $\epsilon=E-L_1(\mu)>0$ is small. For every $\epsilon>0$, denote by $\mathcal{N}_\epsilon$ the neck region consisting of points $(p,q)\in \mathcal{M}^{e\# m}_{\epsilon}$ satisfying $|q+\mu| \geq \mu + \hat l_1(\mu) - \epsilon^{1/2}$ and $|q-1+\mu| \geq 1-\mu - \hat l_1(\mu) - \epsilon^{1/2}$.

\begin{prop}\label{prop_neck_Liouville} There exist $0<\delta<N$ so that for every $\epsilon>0$ sufficiently small, the following statements hold:
\begin{itemize}
    \item[(i)] $\mathcal{N}_\epsilon \subset \{(p,q):|(p,q)-l_1(\mu)|<10\epsilon^{1/2}\}$. Moreover, if $(p,q)\in \mathcal{N}_\epsilon$, then  $|x_3|<N$.

    \item[(ii)]  There exists a Liouville vector field $Y_\epsilon$  on a neighborhood of $\mathcal{M}^{e\# m}_\epsilon$  which is transverse to $\mathcal{M}^{e\# m}_\epsilon$, coincides with $Y_e$ on $-2N < x_3 < -N $, with $Y_m$ on $N < x_3 < 2N$ and with $Y_2$ on $-\delta  < x_3 < \delta $.
\end{itemize}
\end{prop}

\begin{proof}Recall that $l_1(\mu) = (-i\hat l_1(\mu), \hat l_1(\mu))\in \C^2$. Near $l_1(\mu)$, consider the re-scaled coordinates $(p,q)=\epsilon^{1/2}(\hat p,\hat q) +  l_1(\mu)$. The re-scaled Hamiltonian
$
H_\epsilon:= \epsilon^{-1} (H(\epsilon^{1/2}(\hat p,\hat q) + l_1(\mu))-L_1(\mu))
$
admits the potential $U_\epsilon := \epsilon^{-1}(U(\epsilon^{1/2}\hat q + \hat l_1(\mu))-L_1(\mu))$. As $\epsilon \to 0^+$, $U_\epsilon$ converges  uniformly in $C^\infty$ to $\frac{1}{2}\nabla^2 U(\hat l_1(\mu))(\hat q,\hat q)=\frac{1}{2}((2a-1)\hat q_2^2-(4a+1)\hat q_1^2)$ on $|\hat q|<C$ for any fixed $C>0$. Let $(p,q)\in \mathcal N_\epsilon\subset \mathcal{M}^{e\# m}_\epsilon$. Then $U_\epsilon(\hat p,\hat q)\leq 1$, $(\epsilon^{1/2}\hat q_1+\hat l_1(\mu)+\mu)^2+(\epsilon^{1/2}\hat q_2)^2\geq (\mu+\hat l_1(\mu)-\epsilon^{1/2})^2$ and $(\epsilon^{1/2}\hat q_1+\hat l_1(\mu)+\mu-1)^2+(\epsilon^{1/2}\hat q_2)^2\geq (1-\mu-\hat l_1(\mu)-\epsilon^{1/2})^2$. Assuming that $|\hat q|<C$ for a large constant $C$, and $(p,q)\in \mathcal{N}_\epsilon$, we use these estimates to obtain
$$
(2a-1)\hat q_2^2-(4a+1)\hat q_1^2+O(\epsilon^{1/2})\leq 2
$$
and
$$
-1-\frac{\hat q_1^2+\hat q_2^2-1}{2(\mu+\hat l_1(\mu))}\epsilon^{1/2}\leq \hat q_1\leq 1+\frac{\hat q_1^2+\hat q_2^2-1}{2(1-\mu-\hat l_1(\mu))}\epsilon^{1/2},
$$
where $2<a=a(\mu)\leq 4$ is given in \eqref{mua}. Combining the inequalities above, we obtain
$$
\begin{aligned}
&\quad (2a-1)(\hat q_1^2+\hat q_2^2)\leq 2+6a\hat q_1^2+O(\epsilon^{1/2})\\
&\leq 2+6a\max\left\{\left(1+\frac{\epsilon^{1/2}(\hat q_1^2+\hat q_2^2-1)}{2(\mu+\hat l_1(\mu))}\right)^2,\left(1+\frac{\epsilon^{1/2}(\hat q_1^2+\hat q_2^2-1)}{2(1-\mu-\hat l_1(\mu))}\right)^2\right\}+O(\epsilon^{1/2}).
\end{aligned}
$$
Taking $\epsilon\to 0$, we see that $\hat q_1^2+\hat q_2^2\leq (2+6a)/(2a-1)<5$. Therefore, for any $\epsilon>0$ sufficiently small, $|q-\hat l_1(\mu)|<\sqrt{5}\epsilon^{1/2}$ for every $(p,q)$ in the neck region and we take $C= \sqrt{5}$. Since $\mathcal N_\epsilon$ contains only one connected component, we conclude that $\mathcal{N}_\epsilon \subset \{(p,q),|q-\hat l_1(\mu)|<\sqrt{5}\epsilon^{1/2}\}$ for every $\epsilon>0$ sufficiently small.

Since $\dot q=p+iq$, we estimate
$$
\begin{aligned}
|p+i\hat l_1(\mu)|^2&=|\dot q- i(q-\hat l_1(\mu))|^2\leq 2(L_1(\mu)+\epsilon-U(q)))+|q-\hat l_1(\mu)|^2\\
& \leq 2\epsilon-(2a-1)q_2^2+(4a+1)(q_1-\hat l_1(\mu))^2+O(\epsilon^{3/2})+5\epsilon\\
& \leq (20a+12)\epsilon< 92\epsilon.
\end{aligned}
$$

The estimates above imply that if $(p,q)\in\mathcal{N}_\epsilon $, then $|(p,q)-l_1(\mu)|<10\epsilon^{1/2}$. Since $\epsilon^{1/2} x^T=V^{-1}((p,q)-l_1(\mu))^T$ is an affine map and $V$ only depends on $\mu$, we may take $N:=10 \alpha_1$, where $\alpha_1$ is the largest absolute value of an eigenvalue of $V^{-1}$. Therefore, if $(p,q)\in \mathcal{N}_\epsilon$, then $|x_3|< N$. This proves (i). In the proof of (ii) below, we may take $N$ even larger if necessary.

Now we construct the interpolation between $Y_e, Y_2$, and $Y_m$. We assume that $b=1/2$, that is $Y_2$ is a radial vector field both in coordinates $(\hat p,\hat q)$ and $x$. Our interpolation is constructed in coordinates $x$ and supported in $|x_3|<N$. Choose $N>0$ as above. It follows from (i) and Proposition \ref{prop_transversality2} that for every $\epsilon>0$ sufficiently small, $Y_e$ is transverse to $\mathcal{M}^{e\# m}_{\epsilon} \setminus \mathcal{N}_\epsilon$. Moreover, recall that $Y_2$ is globally transverse to $H_2^{-1}(1)$ where $H_2 = \frac{\lambda_1}{2}(x_1^2-x_3^2) + \frac{\lambda_2}{2}(x_2^2+x_4^2)$. Also, $Y_2$ is invariant under the re-scaling of $x$. Hence, for every $\epsilon>0$ sufficiently small, $Y_2$ is transverse to $\mathcal{M}^{e\# m}_{\epsilon}\cap \{|x|< N\}$.

Let $\beta:\R\to [0,1]$ be a smooth even function,  non-increasing on $(-\infty,0],$ $\beta(x_3)=1$  near $x_3=-N$ and $\beta(x_3)=0$  near $x_3=0$.  Later, we will impose additional conditions on $\beta$. Recall that $G(x)=\frac{1}{2}xQ_Gx^T+\epsilon^{-1/2}(-d_6x_1+d_5x_4)$, where $Q_G$ is given in \eqref{QG}. Let
$$
Y_\epsilon(x):=Y_2(x)-X_{\beta(x_3)G(x)}.
$$
Since $H(p,q)=L_1(\mu)+\epsilon H_2(x)+\epsilon^{3/2}R(x)$ near $l_1(\mu)$, where $R(x)$ contains higher order terms, we have
\begin{equation}\label{equ: transversality of Yepsilon}
\begin{aligned}
dH \cdot Y_\epsilon&=dH \cdot Y_2+d(\beta G)\cdot X_{H}=dH \cdot (Y_2-\beta X_G)+Gd\beta \cdot X_{H}\\
&=\beta dH \cdot Y_e +(1-\beta)dH \cdot Y_2+Gd\beta \cdot X_{H}.
\end{aligned}
\end{equation}
In the last identity we used that $Y_2 - X_G = Y_e$. We compute the main terms above
\begin{equation}\label{equ: computation of dH_2.}
\begin{aligned}
 dH \cdot Y_e
&=\epsilon\big(\lambda_1(1-d_1)x_1^2+\lambda_2(1-d_2)x_2^2-\lambda_1d_1x_3^2+\lambda_2d_2x_4^2\\
&\quad+(\lambda_1d_4+\lambda_2d_3)x_1x_4+(\lambda_2d_4-\lambda_1d_3)x_2x_3\big)+\epsilon^{3/2}dR\cdot Y_e\\
&\quad +\epsilon^{1/2}(\lambda_2d_5x_2-\lambda_1d_6x_3),\\
dH \cdot Y_2 &=\epsilon dH_2\cdot Y_2 + \epsilon^{3/2}dR \cdot Y_2=\epsilon + O(\epsilon^{3/2}) + \epsilon^{3/2}dR \cdot Y_2,\\
Gd\beta \cdot X_{H} & =\beta'(x_3)\left(xQ_Gx^T/2+\epsilon^{-1/2}(-d_6x_1+d_5x_4)\right) \left(\epsilon\lambda_1x_1 + \epsilon^{3/2} dx_3 \cdot X_R\right),
\end{aligned}
\end{equation}

Denote by $\mathcal{L}_\epsilon\subset \mathcal{N}_\epsilon$ the neck-region  of $\mathcal{M}^{e \# m}_\epsilon$ contained in $-N\leq x_3 \leq N$. Then there exists $R_0>0$ so that $|R(x)|< R_0$ for every $x\in \mathcal{L}_\epsilon$ and every $\epsilon>0$ small. Let $x\in \mathcal{L}_\epsilon\cap \{x_3=-c\},$ where $c\in [0,N]$. Then
\begin{equation}\label{equ: estimate in (0,N)}
\lambda_1 x_1^2+\lambda_2(x_2^2+x_4^2)=(2+\lambda_1 c^2)-2\epsilon^{1/2}R(x)<(3+\lambda_1 c^2),
\end{equation}
for every $\epsilon>0$ sufficiently small.
Moreover, recall that \eqref{lambda12} and \eqref{C123} give
\begin{gather*}
\lambda_1=(a-1+\sqrt{a(9a-4)})^{1/2}\geq \lambda_2=(1-a+\sqrt{a(9a-4)})^{1/2},\\
C_1=(2+3a-\sqrt{a(9a-4)})^{1/2}< C_2=(2+3a+\sqrt{a(9a-4)})^{1/2},\\
C_0^2=2(a(9a-4))^{1/2}.
\end{gather*}
Since $C_2^4\lambda_2^2-C_1^4\lambda_1^2=8-24a+44a^2+72a^3>0$,  we have
$$
0<-d_5=\frac{(2(a-1)+C_1^2)(1-r_1)}{C_0C_2\sqrt{\lambda_2}}\leq d_6=\frac{(2(a-1)+C_2^2)(1-r_1)}{C_0C_1\sqrt{\lambda_1}}.
$$
Therefore, we estimate the following term of $dH \cdot Y_e$
$$
\begin{aligned}
&\qquad \lambda_2 d_5x_2-\lambda_1d_6 x_3\\
&=(1-r_1)\left(\frac{\sqrt{\lambda_1}(2(a-1)+C_2^2)|x_3|}{C_0C_1}-\frac{\sqrt{\lambda_2}(2(a-1)+C_1^2)x_2}{C_0C_2}\right)\\
&\geq \frac{(1-r_1)}{C_0C_1C_2}\left(c\sqrt{\lambda_1 }C_2(5a+\sqrt{a(9a-4)})-\sqrt{3+\lambda_1c^2 }C_1(5a-\sqrt{a(9a-4)})\right)\\
&\geq \frac{(1-r_1)}{C_0C_1C_2}\left(c\sqrt{\lambda_1} (C_1+C_2)\frac{C_0^2}{2}+5a(c\sqrt{\lambda_1} (C_2-C_1)-\sqrt{3}C_1)\right).
\end{aligned}
$$

Let $\hat c:=\frac{(3/\lambda_1)^{1/2}C_1}{C_2-C_1}$. Notice that $\hat c$ depends only on $\mu$ and if $c=\hat c$, then $c\sqrt{\lambda_1}(C_2-C_1) - \sqrt{3}C_1$ in the last term above vanishes. Choose $N>\hat c$ so that \eqref{equ: estimate in (0,N)} still holds for every $\epsilon>0$ sufficiently small. Therefore, for any $c\in[\hat c,N]$, we have
$$
\lambda_2 d_5x_2-\lambda_1d_6 x_3\geq \frac{(1-r_1)c\sqrt{\lambda_1}(C_1+C_2)C_0}{2C_1C_2}=:-\hat v_1x_3> 0,$$
where $\hat v_1:=\frac{(1-r_1)\sqrt{\lambda_1}(C_1+C_2)C_0}{2C_1C_2}>0$ only depends on $\mu$. Moreover, it follows from \eqref{equ: computation of dH_2.} that
$dH \cdot Y_e-\epsilon^{1/2}(\lambda_2 d_5x_2-\lambda_1d_6x_3)=O(\epsilon)$ for every $\epsilon>0$ small.
We conclude that
$$
\beta dH \cdot Y_e \geq -\epsilon^{1/2}\beta\hat v_1x_3>0 \quad \text{on}\quad \mathcal{L}_\epsilon \cap \{-N \leq x_3 \leq -\hat c\}.
$$
Our interpolation will be supported in this subset, possibly after taking $N$ even larger.

Since $dH_2 \cdot Y_2= H_2 = 1 + O(\epsilon^{1/2})$ on $\mathcal L_\epsilon$, we obtain from \eqref{equ: computation of dH_2.} the following estimate
$$
\begin{aligned}
(1-\beta)dH \cdot Y_2& =(1-\beta)(\epsilon+ O(\epsilon^{3/2})+\epsilon^{3/2}dR\cdot Y_2)\\
& \geq(1-\beta)(\epsilon+O(\epsilon^{3/2}))\\ & \geq (1-\beta)\epsilon/2\end{aligned}
$$
on $\mathcal{L}_\epsilon$, for every $\epsilon>0$ sufficiently small.

Using \eqref{equ: computation of dH_2.} and \eqref{equ: estimate in (0,N)}, we estimate the dominating term of $Gd\beta \cdot X_{H}$
$$
\begin{aligned}
 |\beta'(x_3)\lambda_1(-d_6x_1^2+d_5x_1x_4)| & \leq \lambda_1|\beta'(x_3)|\cdot \left|(d_6+|d_5|)x_1^2+|d_5|x_4^2\right|\\
& \leq |\beta'(x_3)|\cdot (d_6+|d_5|(1+\lambda_1/\lambda_2))(3+\lambda_1 c^2)\\
& =|\beta'(x_3)| \hat v_2(3+\lambda_1x_3^2),
\end{aligned}
$$
where $\hat v_2:=d_6+|d_5|(1+\lambda_1/\lambda_2)>0$ only depends on $\mu$. We conclude from \eqref{equ: computation of dH_2.} that
$$
\begin{aligned}
Gd\beta\cdot X_H&\geq -\epsilon^{1/2}|\beta'(x_3)|(|\lambda_1(d_6x_1^2-d_5x_1x_4)|+O(\epsilon^{1/2}))\\
&\geq -\epsilon^{1/2}|\beta'(x_3)| \hat v_2(4+\lambda_1x_3^2),
\end{aligned}
$$
on $\mathcal{L}_\epsilon$ for every $\epsilon>0$ sufficiently small.

Notice that $\hat v_1,\hat v_2>0$ only depend on $\mu$. Combining the estimates above, we see from \eqref{equ: transversality of Yepsilon} that the inequality $dH \cdot Y_\epsilon>0$ holds true if
\begin{equation}\label{equ: condition on beta}
|\beta'(x_3)|<-\frac{\hat v_1x_3\beta(x_3)}{\hat v_2(4+\lambda_1x_3^2)}+\frac{(1-\beta(x_3))\epsilon^{1/2}}{2\hat v_2(4+\lambda_1x_3^2)}, \quad \forall x_3 \in [-N, -\hat c].
\end{equation}

Hence, in order to construct $\beta$ satisfying \eqref{equ: condition on beta}, we first choose
$$
\beta(x_3):=\left(\hat v_2(4+\lambda_1 x_3^2)\right)^{\frac{\hat v_1}{2\hat v_2\lambda_1}}-\left(\hat v_2(4+\lambda_1\hat c^2)\right)^{\frac{\hat v_1}{2\hat v_2\lambda_1}},\quad \forall x_3 \in [-\check c,-\hat c],
$$
where $\check c>0$ is the unique point satisfying $\beta(-\check c)=1$. Indeed, it follows from this definition that $\beta$ satisfies the following properties
$$
\beta'(x_3)=\frac{\hat v_1x_3\beta(x_3)}{\hat v_2(4+\lambda_1x_3^2)}< 0,\quad \forall x_3\in [-\check{c}, \hat c], \quad \beta(-\check c)=1 \quad \mbox{ and } \quad \beta(-\hat c)=0.
$$
In particular, \eqref{equ: condition on beta} holds.
Notice that $\check c$ is explicitly given by
$$
\check c^2=\frac{1}{\hat v_2\lambda_1}\left(\left(\hat c_2(4+\hat c^2\lambda_1)\right)^{\frac{\hat v_1}{2\hat c_2\lambda_1}}+1\right)^{\frac{2\hat v_2\lambda_1}{\hat v_1}}-\frac{4}{\lambda_1},
$$
and does not depend on $\epsilon$.
We take $N$ even larger so that  $N>\check c$ and \eqref{equ: estimate in (0,N)} still holds for every $\epsilon>0$ sufficiently small. Since $\beta$ decreases from $1$ to $0$ on $[-\check c,-\hat c]$, we can change $\beta$ on small neighborhoods of $x_3 = -\check c$ and $x_3=-\hat c$ to obtain a new smooth even function $\beta:\R \to [0,1]$, still denoted $\beta$, so that $\beta$ has the desired properties, i.e., it is a non-increasing smooth function on $(-\infty,0]$, coincides with $1$ on a neighborhood of $x_3=-N$ and with $0$ on a neighborhood $x_3 = - \hat c$. It is important to notice that the change of $\beta$ near the points $x_3 = -\check c$ and $x_3=-\hat c$ can be done in such a way that $|\beta'|$ of the new function is less or equal than $|\beta'|$ of the initial function. Hence, condition \eqref{equ: condition on beta} still holds, and the interpolated Liouville vector field $Y_\epsilon$ as constructed above is transverse to $\mathcal{M}^{e \# m}_\epsilon$ for every $\epsilon>0$ sufficiently small.  Finally, we can take $\delta:= \hat c.$
\end{proof}

Proposition \ref{prop_neck_Liouville} implies the following proposition.

\begin{prop}\label{prop: interpolation of alpha}
Fix $0<\mu_0<1$. For every $(\mu,\epsilon)$, $\epsilon>0,$ sufficiently close to $(\mu_0, 0)$, the following assertions hold:
\begin{itemize}

\item[(i)] There exists a Liouville vector field $Y_{\epsilon}$ defined on a neighborhood of $\mathcal{M}^{e\# m}_{\mu, L_1(\mu)+\epsilon}$, which is everywhere transverse to $\mathcal{M}^{e\# m}_{\mu, L_1(\mu)+\epsilon}$ and induces a contact form $\hat \alpha_\epsilon$ so that its Reeb flow is a re-paramerization of the Hamiltonian flow.

\item[(ii)] There exist two neighborhoods $U_\epsilon\subset \mathcal{U}_\epsilon$ of $l_1(\mu_0)$, such that $Y_\epsilon$ coincides with $Y_e=(q_1+\mu)\partial_{q_1}+q_2\partial_{q_2}$ and $Y_m=(q_1-1+\mu)\partial_{q_1}+q_2\partial_{q_2}$ in $\mathcal{M}^e_{\mu,L_1(\mu)+\epsilon}\setminus \mathcal{U}_\epsilon$ and $\mathcal{M}^m_{\mu,L_1(\mu)+\epsilon}\setminus \mathcal{U}_\epsilon$, respectively, and $Y_\epsilon$ coincides with $Y_2=\frac{1}{2}(x_1\partial_{x_1}+x_2\partial_{x_2}+x_3\partial_{x_3}+x_4\partial_{x_4})$ in $\mathcal{M}_{\mu,L_1(\mu)+\epsilon}\cap U_\epsilon$.

\item[(iii)] Given any neighborhood $\mathcal{U}\subset \R^4$ of $l_1(\mu)$, $\alpha_\epsilon$ is uniformly converging to $\alpha_0$ on the subset $\mathcal{M}_{\mu,L_1(\mu)+\epsilon}^{e \# m}\setminus \mathcal{U}$ as $\epsilon\to 0$, where $\alpha_0$ is induced by $Y_e$ and $Y_m$.
\end{itemize}
\end{prop}

Finally, Theorem \ref{thm_alphaJ} directly follows from Propositions \ref{prop: interpolation of alpha} and \ref{prop: holo S2 for Y on H}.

\section{Proof of Theorem \ref{thm_muRSl1}}

In this section, we aim to prove Theorem \ref{thm_muRSl1} by showing that periodic orbits on the regularized energy surface $\hat H_{\mu,E}^{-1}(0)$ passing sufficiently close to $S_\pm(\mu_0)$ have high index. Similar index estimates were found in \cite{ dPS1, dPHKS, GRS} using a different approach. Due to the antipodal symmetry, it is enough to consider $S_+(\mu_0)$. Before proving Theorem \ref{thm_muRSl1}, we need some relevant index estimates.

It will be convenient to consider the Robbin-Salamon index of Lagrangian and symplectic paths. Let $\mathrm{Lag}(\mathbb{R}^{2n},\omega)$ be the collection of all the Lagrangian subspaces of $(\mathbb{R}^{2n},\omega)$. Fix $V,W\in \mathrm{Lag}(\mathbb{R}^{2n},\omega)$ satisfying $\R^{2n} = V \oplus W$. Let $\Lambda:[a,b] \to \mathrm{Lag}(\R^{2n}, \omega)$ be a piecewise $C^1$ path of Lagrangian subspaces. Denote $\Lambda_t:=\Lambda(t), \forall t\in [a,b].$ We call $t_0\in [a,b]$ a crossing of $\Lambda$ if $\Lambda_{t_0}\cap V\neq \{0\}$ and $\Lambda_{t_0} \pitchfork W$. For each crossing $t_0$, we define its crossing form as
$$
\Gamma(\Lambda,V,t_0)(v):=\frac{d}{dt}\omega(v,v+w_t)\Big|_{t=t_0},\quad \quad  \forall v\in \Lambda_{t_0}\cap V,.
$$
where $w_t\in W$ is uniquely determined by the condition $v+w_t\in \Lambda_t$ for $t$ close to $t_0$. Notice that $\Gamma$ does not depend on the choice of $W$.
We call $t_0$ a regular crossing if $\Gamma(\Lambda,V,t_0)$ is non-degenerate as a quadratic form. In that case, we denote by $\text{Sign}\Gamma(\Lambda,V,t_0)$ the signature of the crossing form at $t_0$, i.e., the difference between the number of positive and negative eigenvalues. Generically, $\Lambda$ admits only finitely many regular crossings, and the Robbin-Salamon index of $\Lambda$ is defined as the half-integer
\begin{equation}\label{equ: Robbin-Salamon index L}
\mu(\Lambda,V):=\frac{1}{2}\mathrm{Sign}\Gamma(\Lambda,V,a)+\sum_{t_0\in(a,b)  \text{ is crossing}}\mathrm{Sign}\Gamma(\Lambda,V,t_0)+\frac{1}{2}\mathrm{Sign}\Gamma(\Lambda,V,b),
\end{equation}
where the first and last terms on the right-hand side only appear if the points $a$ and $b$ are crossings, respectively. Every Lagrangian path is homotopic with fixed ends to a regular path.

Since the diagonal $\Delta=\{(v,v)|v\in\mathbb{R}^{2n}\}$ and, more generally, the graph $\mathrm{Gr}(M)=
\{(v,Mv)|v\in\mathbb{R}^{2n}\}$, $M\in \mathrm{Sp}(2n),$ are Lagrangians of $\mathrm{Lag}(\mathbb{R}^{2n}\oplus\mathbb{R}^{2n},-\omega_0\oplus\omega_0)$, we define the Robbin-Salamon index of a piecewise $C^1$ path $\psi:[a,b] \to \mathrm{Sp}(\R^{2n})$ of symplectic matrices by
$$
\mu_{\text{RS}}(\psi):=\mu(\mathrm{Gr}(\psi),\Delta)\in \frac{1}{2}\Z.
$$
Notice that $\psi$ does not necessarily start from the identity. In particular, $\det(I - \psi(t_0))=0$ for each crossing $t_0\in [a,b]$. The path $\psi$ determines a path of symmetric matrices $S(t):= -J\dot \psi(t)\psi^{-1}(t),t\in [a,b].$ Conversely, given a path of symmetric matrices $S(t),t\in [a,b],$ and $\psi_a\in \text{Sp}(2n),$ one can integrate $\dot \psi = - J S(t) \psi$ to obtain a path $\psi(t) \in \text{Sp}(2n), t\in [a,b]$ with $\psi(a) = \psi_a$. In that case,
$$
\Gamma(t_0):=\Gamma(\mathrm{Gr}(\psi), \Delta, t_0)= S(t_0)|_{\ker(I - \psi(t_0))}
$$
is nondegenerate for every regular crossing $t_0\in[a,b]$. If all crossings are regular, the Robbin-Salamon index of the path $\psi$ is given by
\begin{equation}\label{equ: Robbin-Salamon index}
\mu_{\text{RS}}(\psi) = \frac{1}{2}\text{Sign}(\Gamma(a)) + \sum_{t_0\in (a,b) \text{ is crossing}}\text{Sign}(\Gamma(t_0)) + \frac{1}{2}\text{Sign}(\Gamma(b)).
\end{equation}

Among the axiomatic properties of the Robbin-Salamon index, we know that $\mu_{\text{RS}}$ is invariant under homotopies with fixed end points, is additive under catenation $\mu_{\text{RS}}(\psi)=\mu_{\text{RS}}(\psi|_{[a,c]}) + \mu_{\text{RS}}(\psi|_{[c,b]}),\forall c\in (a,b)$, and satisfies the product property $\mu_{\text{RS}}(\psi\oplus \psi') =\mu_{\text{RS}}(\psi)+ \mu_{\text{RS}}(\psi')$. Moreover, if $\psi(a) = I$ and $\psi(b)$ does not have $1$ as an eigenvalue, then $\mu_{\text{RS}}(\psi)$ coincides with the Conley-Zehnder index $\mu_{\text{CZ}}(\psi)$ of a nondegenerate path of symplectic matrices. For a general path $\psi$, with $\psi(a)=I$, the difference between $\mu_{\text{RS}}(\psi)$ and $\mu_{\text{CZ}}(\psi)$ is bounded by a uniform constant depending only on $n$.

For $2n=2$, let
\begin{equation}\label{ll1}
\psi_1(t) := \left(\begin{array}{rr} \cosh \lambda_1 t & \sinh \lambda_1 t \\  \sinh \lambda_1 t &  \cosh \lambda_1 t \end{array} \right)\psi_1(0), \quad t\in [a,b],
\end{equation}
where $\lambda_1>0$ and $\psi_1(0)=(a_{ij})_{2\times 2}\in \text{Sp}(2)$. Then $\psi_1$ has at most two crossings. This follows from the fact that $\text{tr}(\psi_1(t))=c_1e^{\lambda_1t}+c_2e^{-\lambda_1t}=2$ has at most two solutions $t_1,t_2\in \R$, where $c_1:=\frac{1}{2}(a_{11}+a_{22}+a_{12}+a_{21})$ and $c_2:=\frac{1}{2}(a_{11}+a_{22}-a_{12}-a_{21})$. The existence of two crossings $t_1\leq t_2$, with multiplicities counted, only occurs if both coefficients $c_1$ and $c_2$ are positive. Otherwise, there exists at most $1$ crossing, say $t_1$, which must be simple, i.e. $\dim \ker (I-\psi_1(t_1)) =1$. Assume that $c_1,c_2>0$. Notice that $S_1 :=-J\dot \psi_1 \psi_1^{-1}= \text{diag}(\lambda_1,-\lambda_1)$. Then $t_1=t_2$ is equivalent to $c_1c_2=1$ and $t_1\neq t_2$ is equivalent $c_1c_2<1$. Assume that $t_1=t_2$. Since $a_{11}a_{22}-a_{12}a_{21}=1$, we see that $c_1c_2=1$ is equivalent to $|a_{11}-a_{22}|=|a_{12}-a_{21}|=:v$. If $v=0$, then  $a_{11}=a_{22}>0,a_{12}=a_{21}$. In this case, $\psi_1(t_1)=I$ and the crossing form is $S_1$, which has signature $0$. If $v>0$, then $\psi_1(t_1)$ has eigenvalue $1$ with $\ker(\psi_1(t_1)-I)=\R(1,1)$ or $\R(1,-1)$. Then the crossing form $S_1|_{\ker(\psi_1(t_1)-I)}=0$. If $t_1\neq t_2$, we can reduce this case to a small neighborhood of $\{c_1c_2=1\}$ by continuity of $c_1c_2$ and the homotopy invariance of $\mu_{\text{RS}}$. Then, as a small perturbation of the previous case, the signature of the crossing forms at $t_1$ and $t_2$ is $0$. Hence, in all cases, we have
\begin{equation}\label{muRS1}
|\mu_{\text{RS}}(\psi_1)|\leq 1.
\end{equation}

If
\begin{equation}\label{ll2}
\psi_2(t) = \left(\begin{array}{rr} \cos \lambda_2 t & -\sin \lambda_2 t \\ \sin \lambda_2 t &  \cos \lambda_2 t \end{array} \right)\psi_2(0), \quad t\in [a,b],
\end{equation}
for some $\lambda_2>0$, then there exist precisely two crossing points $t_1,t_2\in [a,a+2\pi/\lambda_2)$ for every $a\in \R$,  where multiplicities are counted, i.e., if $t_1\neq t_2$ then $\dim \ker (I - \psi_2(t_i))=1,i=1,2$, and if $t_1=t_2$, then $\dim \ker (I-\ker \psi_2(t_i))=2$. Since $S_2 :=-J\dot \psi_2\psi_2^{-1}= \text{diag}(\lambda_2,\lambda_2)>0,$ we have $\sum_{t}\mathrm{Sign}(\Gamma(t))= 2$, where the sum is over the crossing points $t\in [a,a+2\pi/\lambda_2)$. Hence
\begin{equation}\label{muRS2}
\mu_{\text{RS}}(\psi_2) \geq 2\lfloor \lambda_2(b-a)/2\pi  \rfloor , \quad \forall a<b.
\end{equation}

Now, for convenience, we denote $\hat H=\hat H(y,x)$ the regularized Hamiltonian omitting the subscripts $\mu,E$. Consider a smooth path $\psi(t)\in \text{Sp}(2n),t\in [a,b]$, satisfying $\dot \psi = J B(t) \psi,$ where $B(t)= \nabla^2\hat H(\gamma(t))$ is the Hessian of $\hat H$ along a solution $\gamma\subset \hat H^{-1}(0)$ of $\dot \gamma(t) = J \nabla \hat H(\gamma(t))$.
Since $S_+(\mu)$ is a saddle-center singularity, see section \ref{subsec: holo two sphere near l1}, there exist symplectic coordinates $\tilde x\in \R^4$  near $S_+(\mu)$ so that the Hamiltonian has the form $\hat H = \hat H_2 + \hat R,$ where $\hat H_2 = \frac{\hat \lambda_1}{2}(\hat x_1^2-\hat x_3^2) + \frac{\hat \lambda_2}{2}(\hat x_2^2+\hat x_4^2)$, and $\hat R$ vanishes up to second order. If $\gamma(t)\equiv S_+(\mu), \forall t\in \R,$ then $B_{S_+(\mu)}(t) \equiv \text{diag}(\hat \lambda_1,\hat \lambda_2,-\hat \lambda_1,\hat \lambda_2)$. In particular, the linearized flow over $S_+(\mu)$ decouples $\psi_{S_+(\mu)}=\psi_1 \oplus \psi_2$ on the  $(\hat x_1,\hat x_3)$ and $(\hat x_2,\hat x_4)$ directions,  as in \eqref{ll1} and \eqref{ll2}, respectively. From \eqref{muRS1}, \eqref{muRS2}, and the product property of the Robbin-Salamon index, we obtain
\begin{equation}\label{muRS}
\begin{aligned}
\mu_{\text{RS}}(\psi_{S_+(\mu)}|_{[a,b]})=\mu_{\text{RS}}(\psi_1|_{[a,b]})+\mu_{\text{RS}}(\psi_2|_{[a,b]})
 \geq 2\lfloor \lambda_2(b-a)/2\pi\rfloor -1, \quad \forall a<b.
\end{aligned}
\end{equation}

We observe that the restriction of $\nabla^2 \hat H$ to $\mathbb R\{\partial_{p_1},\partial_{p_2}\}$ is always positive-definite. Using the above Robbin-Salamon index, a uniform estimate of the Conley-Zehnder index can be obtained.
\begin{prop}\label{prop_lower_bound_index}
Let $x(t)=\varphi_t(x(0))$ be a Hamiltonian trajectory of $\hat H=\hat H(y,x)$. Let $\psi_t:=D\varphi_t(x(t))\psi_0$, $\forall t\in[0,T]$ be the linearized path of symplectic matrices starting from $\psi_0\in \mathrm{Sp}(4).$ Then $\mu_{\mathrm{RS}}(\psi_t|_{[0,T]})\geq -9$. 
\end{prop}

\begin{proof}
Recall from \eqref{equ: Robbin-Salamon index} that $\mu_{\text{RS}}(\psi_t|_{[0,T]})$ is the sum of $\mathrm{Sign}(\Gamma(t))$ at the crossings $t\in[0,T]$. Let $t_0$ be a crossing. We have $\ker(I-\psi_{t_0})\cong \mathrm{Gr}(\psi_{t_0})\cap \Delta$, where $\mathrm{Gr}(\psi_t):=\{(v,\psi_t(v))| v\in\mathbb{R}^{4}\}$ and $\Delta:=\{(v,v)|v\in\mathbb{R}^{4}\}$. Let $\mathcal{L}(4)$ denote the collection of Lagrangian subspaces of $(\mathbb{R}^{4}\oplus\mathbb R^{4},-\omega_0\oplus \omega_0)$. Then $\mathrm{Gr}(\psi_t),\Delta \in \mathcal{L}(4)$. By definition, we have $\mu_{\text{RS}}(\psi_t|_{[0,T]})=\mu(\mathrm{Gr}(\psi_t|_{[0,T]}),\Delta)$.

Let $\mathrm{L}(2)$ be the collection of  Lagrangian subspaces in $(\mathbb{R}^{4},\omega_0)$. Consider the Lagrangian subspace $V=\mathbb{R}\{\partial_{y_1},\partial_{y_2}\}\in \mathrm{L}(2)$. By flowing $V$, we obtain a Lagrangian path $\psi_t\psi_0^{-1}V|_{[0,T]}$ in $\mathrm{L}(2)$. Since $\dot \psi_t=J\nabla^2\hat H(x(t))\psi_t$ for every $t$ and $\nabla^2\hat H|_V=I$, we compute for every crossing $t_0\in[0,T]$
$$
\begin{aligned}
\Gamma(\psi_t\psi_0^{-1}V|_{[0,T]},V,t_0)(v)&=\frac{d}{dt}\omega_0(v,\psi_{t}\psi_{t_0}^{-1}v)\big|_{t=t_0}=(J_4v,\dot \psi_{t_0}\psi_{t_0}^{-1}v)\\
&=(\nabla^2\hat H(x(t_0))v,v)=\|v\|^2>0,
\end{aligned}
$$
for every nonzero $v\in (\psi_{t_0}\psi_0^{-1}V)\cap V$. Since $t=0$ is a boundary crossing with multiplicity $2$, we conclude that
$$
\mu(\psi_t\psi_0^{-1}V|_{[0,T]},V)=\mu(\Lambda_t|_{[0,T]},V\oplus V)\geq 1,\quad \Lambda_t:=\mathrm{Gr}(\psi_t\psi_0^{-1})\in \mathcal L(4).
$$
Let $s(V\oplus V,\Delta; \Lambda_0,\Lambda_T):=\mu(\Lambda_t|_{[0,T]},\Delta)-\mu(\Lambda_t|_{[0,T]},V\oplus V)$ be the Hormander index, where $\Lambda_0=\Delta$ and $\Lambda_T=\mathrm{Gr}(\psi_T\psi_0^{-1})$, see Theorem 3.5 in \cite{RS1}. We thus obtain
$$\begin{aligned}
\mu_{\text{RS}}(\psi_t\psi_0^{-1}|_{[0,T]})&=
\mu(\Lambda_t|_{[0,T]},\Delta)=\mu(\Lambda_t|_{[0,T]},V\oplus V)+s(V\oplus V,\Delta; \Lambda_0,\Lambda_T)\\
&\geq 1 + s(V\oplus V,\Delta; \Delta,\Lambda_T)\geq -5.
\end{aligned}$$
The last inequality uses Theorem 3.5-(3) in \cite{RS1}, since $\{V\oplus V,\Delta,\Lambda_T\}$ can be written as the graph of symmetric metrics at the same time, under a suitable change of coordinates.

Choose a path $\beta_1$ from $\psi_0$ to $I_4$ and let $\beta_2:=\psi_T\psi_0^{-1}\beta_1$. From Appendix \ref{sec: positive paths}, one can choose $\beta_1$ so that $-J\dot \beta_1 \beta_1^{-1}$ is non-positive-definite and satisfies
$$
-4\leq \mu_{\text{RS}}(\beta_1) \leq 0.
$$
In particular, $-J\dot \beta_2 \beta_2^{-1}$ is also non-positive-definite, then we have $\mu_{\text{RS}}(\beta_2) \leq 0$. Notice that $\beta_2$ is a path from $\psi_T$ to $\psi_T\psi_0^{-1}$, and then
$\psi_t\psi_0^{-1}$ is a path from $I$ to $\psi_T\psi_0^{-1}$. Notice that the path $\psi_t$ is homotopic to the catenation $\bar \beta_2 * \psi_t\psi_0^{-1} * \beta_1$, starting from $\beta_1$, where the homotopy is given by  $s\mapsto \psi_t\psi_0^{-1}\beta_1(s)$. We conclude from the homotopy invariance and the catenation property that
$
\mu_{\text{RS}}(\psi_t|_{[0,T]})=\mu_{\text{RS}}(\beta_1)+\mu_{\text{RS}}(\psi_t\psi_0^{-1}|_{[0,T]})-\mu_{\text{RS}}(\beta_2)\geq -9.$
\end{proof}

Now we are ready to prove Theorem \ref{thm_muRSl1}. Let $P'=(\gamma,T)\subset \hat H^{-1}(0)$ be a $T$-periodic orbit which is not a cover of the Lyapunov orbit near $S_\pm(\mu)$, where $|\mu-\mu_0|$ and $E - L_1(\mu_0)$ are small. Due to the antipodal symmetry of $\hat H$, we only need to consider $S_+(\mu_0)$. Let $\mathcal{U}_1\subset \R^2 \times (\R \times \R / 2\pi \Z)$ be a small open neighborhood of $S_+(\mu_0)$ so that if $\gamma$ intersects $\mathcal{U}_1$, then for $(\mu,E)$ sufficiently close to $(\mu_0, L_1(\mu_0))$, $\gamma$ lies in $\widetilde{\mathcal{M}}_{\mu,E}:=\widetilde{\mathcal{M}}_{\mu,E}^e \cup \widetilde{\mathcal{M}}_{\mu,E}^m$. Since $\gamma$ is not a cover of the Lyapunov orbit near $S_+(\mu)$, we can take $\U_1$ sufficiently small so that $\gamma$ necessarily intersects   $\widetilde{\mathcal{M}}_{\mu,E} \setminus \U_1$. Denote by $\tilde \alpha_{\mu,E}$ the lift of the contact form to $\widetilde{\mathcal{M}}_{\mu,E}$, and recall that $\mathcal{A}(\gamma) = \int_\gamma \tilde \alpha_{\mu,E}$. Here, the contact form is also well-defined if $E \leq L_1(\mu)$. Let $\psi(t)\in \text{Sp}(4)$ be the path of symplectic matrices obtained by integrating the linearized flow along $\gamma$ so that  $\psi(0)=I$.

Since $S_+(\mu)$ is an equilibrium point that varies smoothly with $\mu$, we can take $\mathcal{U}_1$ even smaller, if necessary, and find an open neighborhood $\mathcal{U}_2\subset \mathcal{U}_1$ of $S_+(\mu_0)$ so that for $(\mu,E)$ sufficiently close to $(\mu_0, L_1(\mu_0))$, the following holds: if $\gamma$ intersects $\mathcal{U}_2$, then for an arbitrarily large time interval $[a,b]$, $\gamma$ lies inside $\mathcal{U}_1$ and thus its period is $T> b-a\gg 0$. Taking $\mathcal{U}_1, \U_2$ even small and $(\mu,E)$ even closer to $(\mu_0, L_1(\mu_0))$, if necessary, we can assume furthermore that $B(t):=\nabla^2H(\gamma(t))$ along $\gamma|_{[a,b]}$ is arbitrarily close to the constant path of symmetric matrices $B_{S_+(\mu)}\equiv \nabla^2H(S_+(\mu))$.
Moreover, it follows from \eqref{muRS} that given $N \in \N$, there exists a neighborhood $\V_{N} \subset \U_2$ of $S_+(\mu)$ so that if $\gamma([a,b])\subset \mathcal{U}_1$ and $\gamma([a,b]) \cap \mathcal{V}_{N} \neq \emptyset$, then $\mu_{\text{RS}}(\psi|_{[a,b]})> N$.  By Proposition \ref{prop_lower_bound_index}, the contribution to the Robbin-Salamon index of $\gamma$ along the complement of $[a,b]$ in the domain of $\gamma$ is $\geq -9$. This follows from the catenation axiom. Then $\mu_{\text{RS}}(\gamma)=\mu_{\text{RS}}(\gamma|_{[a,b]})+\mu_{\text{RS}}(\gamma|_{[b,T+a]})>N-9$. Since $|\mu_{\text{CZ}}(\gamma) - \mu_{\text{RS}}(\gamma)|$ is uniformly bounded by $2$, Theorem \ref{thm_muRSl1} follows by taking the symmetric neighborhood $\U_N := \V_{N+11}$ of $S_\pm$.

\section{Proof of Theorem \ref{thm_main3}}

Let $\alpha=\alpha_{\mu,E},J=J_{\mu,E}$ and $\mathcal S=\partial \mathcal M^e_{\mu,E}=\partial \mathcal M^m_{\mu,E}$ be as in Theorem \ref{thm_alphaJ} for every $(\mu,E)$ sufficiently close to $(\mu_0,L_1(\mu_0))$, with $E>L_1(\mu)$. Let $P_3^e=P_{3,\mu,E}^e\subset \mathcal{M}_{\mu,E}^e$ be the continuous family of retrograde orbits and $\mathcal{D}= \mathcal{D}_{\mu,E}$ be the $2$-disks for $P_3^e$ as in Theorem \ref{thm_retrograde}, for $(\mu,E)$ sufficiently close to $(\mu_0, L_1(\mu_0))$. Assume that $\mathcal{M}_{\mu_0, L_1(\mu_0)}$ satisfies the conditions given in Theorem \ref{thm_main3}, i.e., the double cover of $P_{3,\mu_0, L_1(\mu_0)}^e$ has index $\geq 3$ and $\P'_0 = \emptyset$.

First, we show that for $(\mu,E)$ sufficiently close to $(\mu_0, L_1(\mu_0))$, with $E>L_1(\mu),$ the regularized component $\mathcal{M}_{\mu,E}^{e\# m}$ does not contain a contractable periodic orbit $P' \neq P_2$ that is unlinked with $P_3^e$, has rotation number $1$ and action $\leq \S(\mathcal{D}, \alpha)$. Fix $N=3$ and let $\mathcal{U}_3\subset \R^2 \times (\R \times \R / 2\pi \Z)$ be the small neighborhood of the singularities $S_\pm(\mu_0)$ corresponding to $l_1(\mu_0)$, as given in Theorem \ref{thm_muRSl1}, so that for $(\mu,E)$ sufficiently close to $(\mu_0,L_1(\mu_0))$ every trajectory in $\mathcal{M}_{\mu,E}^e$ which intersects $\U_3$ has index $>3$. By contradiction, assume there exists $(\mu_i,E_i)\rightarrow (\mu_0,L_1(\mu_0))$, $E_i > L_1(\mu_i)$, and a sequence of contractable periodic orbits $P'_i\subset \mathcal M^e_{\mu_i,E_i} \setminus \U_3$ that are unlinked with $P_3^e$, have rotation number $1$ and action $\leq \S(\mathcal{D}, \alpha)$. Passing to a subsequence, $P'_i$ uniformly converge to a contractible periodic orbit $P'_0\subset \mathcal{M}_{\mu_0,L_1(\mu_0)}^e\setminus \{l_1(\mu)\}$ with rotation number $1$ and action $\leq \S(\mathcal{D}_{\mu_0,L_1(\mu_0)}, \alpha_{\mu_0, L_1(\mu_0)})$. This orbit is geometrically distinct from $P_3^e$ since the index of $(P_3^e)^2$ is $\geq 3$, or equivalently, its rotation number is $>1$. Hence, $P'_0$ is unlinked with $P_3^e$. This contradicts the hypothesis $\P'_0 = \emptyset$ and proves (i).

To prove (ii), we recall that the $2$-disks $\mathcal{D}\subset \mathcal{M}_{\mu,E}^e \setminus \S$ for $P_{3}^e$ have $|d\alpha|$-area uniformly bounded by $\S(\mathcal{D}_{\mu_0, L_1(\mu_0)},\alpha_{\mu_0, L_1(\mu_0)})+1$ for every $(\mu,E)$ sufficiently close to $(\mu_0, L_1(\mu_0))$. From (i), we can assume that $\mathcal{M}_{\mu,E}^e \setminus \S$ has no periodic orbit with rotation number $1$ and action $\leq \S(\mathcal{D},\alpha)$ for $(\mu,E)$ sufficiently close to $(\mu_0, L_1(\mu_0))$. A direct application of Theorem \ref{main1} gives the desired $2-3$ foliation of $\mathcal{M}_{\mu,E}^e$ for $(\mu,E)$ sufficiently close to $(\mu_0, L_1(\mu_0))$, with $E> L_1(\mu)$, which is adapted to the contact form $\alpha=\alpha_{\mu,E}$ and $J'=J'_{\mu,E}$, where $J'$ coincides with $J$ on a neighborhood of $\S$ and is arbitrarily $C^\infty$-close to $J$, see Theorem \ref{thm_alphaJ}. This finishes the proof of (ii).

Now we prove (iii). Corollary \ref{cor_homoclinic} gives at least one homoclinic orbit $\gamma\subset \mathcal{M}_{\mu,E}^e$ to $P_2$ for every $(\mu,E)$ sufficiently close to $(\mu_0,L_1(\mu_0)),$ with $E>L_1(\mu)$. Let us show that if the branch $W^u$ of the unstable manifold of $P_2$ in $\mathcal{M}_{\mu, E}^e$ does not coincide with the branch $W^s$ of the stable manifold of $P_2$ in $\mathcal{M}_{\mu, E}^e$, then there exist infinitely many transverse homoclinic orbits to $P_2$, which implies infinitely many periodic orbits and positive topological entropy. The proof of this fact goes back to Conley \cite{conley} and was further exploited in \cite{dPS2, dPHKS}. Hence, we only briefly address the argument of those references.

Let $(\theta, r) \in \R / 2\pi \Z \times (-\epsilon, \epsilon), \epsilon>0$ small, be suitable real-analytic coordinates on small annuli $A^s, A^u \subset \mathcal{M}_{\mu,E}^e$ transverse to the local branches $W^s_{\text{loc}}, W^u_{\text{loc}}$ of the stable/unstable manifolds $W^s, W^u$ of $P_2$, respectively at $r=0$. The points in $A^s$ whose forward trajectory immediately transit to $\mathcal{M}_{\mu,E}^m$ correspond to $r<0$, and there exists a well-defined local map from $A^s\cap \{r>0\}$ to $A^u\cap \{r>0\}$ given by
$$
l(\theta,r) = (\theta + \Delta(r), r), \quad \Delta(r) = g(r) - h(r) \ln r,  \quad r>0,
$$
where $g(r)$ and $h(r)$ are real-analytic functions defined near $r=0$ and $h(0)>0,$ see \cite[section 4]{dPHKS}. This map describes the infinite twist from $A^s$ to $A^u$ as $r\to 0^+$. We may assume that the homoclinic orbit $\gamma$ locally intersect $A^s$ and $A^u$ precisely at $(\theta, r) = (0,0)$. The flow determines a real-analytic map $G(\theta,r),$ defined near $(\theta,r)=(0,0)$ and satisfying $G(0)=0$, corresponding to the first hit map from $A^u$ to $A^s$ along $\gamma$. The image under $G$ of a small arc in $r=0$ centered $(0,0)$ is a real-analytic curve $\hat \eta_u$ intersecting $r=0$ only at $(0,0)$. This is so since $W^s$ does not coincide with $W^u$ and the flow is real-analytic. Also, we may assume that $\hat \eta_u$ contains an open sub-arc $\eta_u$ contained in $r>0$ and with endpoint in $(0,0)$, and there exists a small arc $\eta_s$ in $r=0$ with an endpoint in $(0,0)$ which does not lie in the image of $G|_{\{r<0\}}$, see scenarios (b) and (c) in \cite[section 4]{dPHKS}, or the special homoclinic orbits in \cite[section 6]{dPS2}.
Now the intersections between $l(\eta_u)$ and $G^{-1}(\eta_s)$ correspond to new homoclinic orbits to $P_2$. In fact, since $G$ is a diffeomorphism onto the image
 and $l(\eta_u)$ twists infinitely many times around the circle $r=0$, there exist infinitely many new homoclinic orbits. From the properties of $l$, it is then a crucial remark from Conley \cite{conley} that such intersections are transverse for $r>0$ sufficiently small. This implies infinitely many transverse homoclinic orbits to $P_2$, also implying infinitely many periodic orbits and positive topological entropy.

 If the branches $W^s, W^u \subset \mathcal{M}_{\mu,E}^e$ coincide, then their intersections with a plane $\mathcal{D}\in \mathcal{F}$ asymptotic to $(P_3^e)^2$ are formed by finitely many embedded circles $C_i, i=1,\ldots, m,$ which bound disjoint closed disks $B_i \subset \mathcal{D}$. Moreover, the $2-3$ foliation in $\mathcal{M}_{\mu,E}^e$ determines a first return area-preserving map $h$ defined in $\mathcal{D} \setminus \cup_i B_i$. The map $h^m$ behaves like the local map $l$ near each $C_i$ in suitable coordinates. Hence, $h^m$ twists infinitely many times around $C_i$. Frank's generalization of the Poincar¨¦-Birkhoff Theorem, see \cite{franks1, franks_erratum}, implies the existence of infinitely many periodic points of $h^m$, which correspond to infinitely many periodic orbits in $\mathcal{M}_{\mu,E}^e$, see more details in \cite[section 7]{dPS2}.  This finishes the proof of (iii).

 The same argument holds for $\mathcal{M}_{\mu, E}^m$, and if both regularized subset $\mathcal{M}_{\mu_0, L_1(\mu_0)}^e, \mathcal{M}_{\mu_0, L_1(\mu_0)}^m$ satisfy the conditions  given in Theorem \ref{thm_main3}, then the union of both $2-3$ foliations for $\mathcal{M}_{\mu, E}^e$ and $ \mathcal{M}_{\mu, E}^m$ gives a $3-2-3$ foliation for $\mathcal{M}_{\mu,E}^{e \# m}$ for every $(\mu,E)$ sufficiently close to $(\mu_0, L_1(\mu_0)),$ with $E> L_1(\mu)$. The proof of Theorem \ref{thm_main3} is complete.

\section{Proof of Theorem \ref{thm_3bp2}}

In this section, we consider the circular planar restricted three-body problem with mass ratio $\mu=1/2$. We prove that the regularized subsets $\mathcal M^e_{\mu,E}$ and $\mathcal M^m_{\mu,E}$ are strictly convex for every $E\leq L_1(\mu)=-2$.

To prove the strict convexity of those subsets of the energy surface, we start with a generalization of the results in \cite{Sa1} by presenting conditions for the Hessian of a non-mechanical Hamiltonian to be definite on the tangent space of the energy surface, see Theorem \ref{thm: convexity formula}.  This condition is equivalent to the positivity of the sectional curvatures at the regular points. We assume that the magnetic field and the potential are decoupled, which is the case of $\mu=1/2$. The criterion formula is further simplified, see Corollary \ref{cor: a special case}. We then use a certain monotonicity argument to convert the problem of lower energies into that of checking the positivity of a smooth function on the regularized Hill region at the critical level, see Lemma \ref{lem: I0>0}. The positivity of the function is then checked at the critical level and a simple local-to-global argument shows that positive curvatures at the regular points imply global convexity.

\subsection{The generalized criterion}\label{subsec: criterion formula}
Assume that the Hamiltonian $H$ has the following form
\begin{equation}\label{equ: magnetic hamiltonian}
H(y,x)=\frac{(y_1+f_1(x))^2}{2}+\frac{(y_2+f_2(x))^2}{2}+V(x),\quad \forall y,x\in \mathbb{R}^2,
\end{equation}
where $F=(f_1,f_2)$ is a smooth magnetic field and $V$ is a smooth potential function. Let $y_F:=y+F$ and $J=\begin{pmatrix} 0&-1\\1&0\end{pmatrix}$. Then
$$
\begin{aligned}
\nabla H&=(y_1+f_1,y_2+f_2,(y+F)\cdot \partial_{x_1}F+\partial_{x_1}V,(y+F)\cdot \partial_{x_2}F+\partial_{x_2}V)\\
&=(y_F,y_F\nabla F+\nabla V),
\end{aligned}
$$
and
$$
\nabla^2H=\begin{pmatrix} I_2 & \nabla F \\
\nabla F ^{T}& M\end{pmatrix},
$$
where $\nabla F=(\partial_{x_1}F^T,\partial_{x_2}F^T)$ and
$$\begin{aligned}
M&:=\frac{1}{2}\nabla^2(f_1^2+f_2^2)+y_1\nabla^2f_1+y_2\nabla^2f_2+\nabla^2V\\
&=\nabla f_1\otimes \nabla f_1+\nabla f_2\otimes \nabla f_2+(y_1+f_1)\nabla^2f_1+(y_2+f_2)\nabla^2f_2+\nabla^2 V.
\end{aligned}$$

Fix the energy surface $\mathcal{M}:=H^{-1}(h)$ and assume that $\mathcal{M}$ has at most a finite number of singularities. Denote by $S$ the singular set of $\mathcal{M}$ and notice that $\mathcal{M} \setminus S$ is connected. We assume that the projection of $\mathcal{M}$ to the $x$-plane is a topological disk $\overline{\mathcal{H}} \subset \R^2$ whose interior $\mathcal{H}$ is formed by points satisfying $y_F\neq0$ and the points on the boundary $\partial \mathcal{H}$ satisfy $y_F=0$. For a regular point $(y,x)\in \mathcal{M}$, with $y_F\neq0\Rightarrow V(x) < h$, one can choose a global tangent frame
$$
\begin{aligned}
X_1&=(-\partial_{y_2}H,\partial_{y_1}H,0,0)=(y_FJ^T,0),\\
X_2&=(-\partial_{x_1}H,-\partial_{x_2}H,\partial_{y_1}H,\partial_{y_2}H)=(-y_F\nabla F-\nabla V,y_F),\\
X_3&=(\partial_{x_2}H,-\partial_{x_1}H,\partial_{y_2}H,-\partial_{y_1}H)=((y_F\nabla F+\nabla V)J,y_FJ).
\end{aligned}
$$
We see that $X_1(y,x) \neq 0$ is tangent to the fiber direction whenever $y_F\neq 0$. If $(y,x)\in \mathcal M$ is a regular point of $H$, satisfying $y_F=0$, then $x\in \partial \mathcal H$, and we can choose the global tangent frame
$$X_{1b}=(1,0,0,0),\quad X_{2b}=(0,1,0,0),\quad X_{3b}=(0,0,V_{2},-V_{1}).$$

Let $W:=\nabla^2H|_{T\mathcal M}$ be the tangent Hessian of $H$ restricted to $T\mathcal M$. In the frame $\{X_1,X_2,X_3\}$, $W$ is represented by the matrix  $(X_i\nabla^2 H X_j^T)_{3\times 3}$ along $\mathcal{M}\cap \{y_F\neq 0\}$. In the frame $\{X_{1b}, X_{2b}, X_{3b}\}$, $W$ is represented by the matrix $(X_{ib}\nabla^2HX_{jb}^T)_{3\times 3}$ along $\mathcal{M}\cap \{y_F=0\}$.

\begin{thm}\label{thm: convexity formula}
The tangent Hessian $W= \nabla^2H|_{T\mathcal M}$  is positive-definite if and only if $$\det(U_W(\theta,x_1,x_2))>0,$$
where $U_W$ is the following $2\times 2$-matrix
\begin{equation}\label{equ: U}
U_W(\theta,x_1,x_2):=r^2(\cos\theta\nabla^2f_1+\sin \theta\nabla^2 f_2)+r\nabla^2V+r^{-1}\nabla V\otimes \nabla V.
\end{equation}
 Here, $r=|y_F|=\sqrt{2(h-V)}$ and $\theta$ is the argument of $y_F$, if $y_F\neq 0$. Notice that $\det U_W=V_{22}V_1^2+V_{11}V_2^2-2V_{12}V_1V_2$ is well-defined for $y_F=0$.  Moreover, if $F$ and $V$ satisfy the following symmetries
\begin{equation}\label{equ: symmetries}
\begin{aligned}
&F(x_1,x_2)=NF(x_1,-x_2)=-NF(-x_1,x_2),\\
&V(x_1,x_2)=V(x_1,-x_2)=V(-x_1,x_2),
\end{aligned}
\end{equation}
with $N=\mathrm{diag}(-1,1)$, then
$$
U_W(\theta,x_1,x_2)=NU_W(\pi-\theta,x_1,-x_2)N
=NU_W(-\theta,-x_1,x_2)N.
$$
\end{thm}
\begin{proof} Assume first that $y_F\neq0$. We identify $W$ with the matrix $(X_i\nabla^2 H X_j^T)_{3\times 3}=:(W_{ij})_{3\times 3}$. A direct computation gives
$$\begin{aligned}
W_{11}&=|y_F|^2,\\
W_{12}&=y_FJ(y_F\nabla F + \nabla V)^T - y_FJ\nabla F y_F^T,\\
W_{13}&=-y_F(y_F\nabla F + \nabla V)^T + y_FJ\nabla F Jy_F^T,\\
W_{22}&=|y_F\nabla F + \nabla V|^2-2(y_F\nabla F + \nabla V)\nabla F y_F^T +
  y_F M y_F^T,\\
W_{23}&=(y_F\nabla F + \nabla V)(J\nabla F +\nabla F J)y_F^T - y_F M Jy_F^T,\\
W_{33}&=|y_F\nabla F + \nabla V|^2 -
  2(y_F\nabla F + \nabla V)J\nabla F Jy_F^T - y_F J M Jy_F^T.
\end{aligned}$$
Then the determinant of $W$ becomes
$$
\det W=|y_F|^4(C_1+|y_F|^2C_2)=:|y_F|^4C_0,
$$
where $y_F=\sqrt{2(h-V)}(\cos \theta,\sin\theta)$. Denoting $(s,t):=(\cos \theta,\sin\theta)$, we obtain
$$
\begin{aligned}
C_1&=(\nabla VJ)M(\nabla VJ)^T-(\nabla f_1( \nabla V J)^T)^2-(\nabla f_2 (\nabla V J)^T)^2\\
&=\sqrt{2(h-V)}\big(s\cdot (\nabla VJ)\nabla^2f_1(\nabla VJ)^T+t\cdot(\nabla VJ)\nabla^2f_2(\nabla VJ)^T\big)\\
&\quad\ +(\nabla VJ)(\nabla f_1\otimes \nabla f_1+\nabla f_2\otimes \nabla f_2+\nabla^2 V)(\nabla VJ)^T\\
&\quad\ -(\nabla f_1( \nabla V J)^T)^2-(\nabla f_2 (\nabla V J)^T)^2,\\
C_2&=\det M-(\nabla f_1J)M(\nabla f_1J)^T-(\nabla f_2J)M(\nabla f_2J)^T+(\nabla f_1 (\nabla f_2J)^T)^2.
\end{aligned}
$$
Since for any smooth function $g:\R^2 \to  \R$, we have
\begin{equation}\label{equ: formula g}
(\nabla gJ)(\nabla f_1\otimes \nabla f_1+\nabla f_2\otimes \nabla f_2)(\nabla gJ)^T=(\nabla f_1( \nabla g J)^T)^2+(\nabla f_2 (\nabla g J)^T)^2,
\end{equation}
we compute
\begin{equation}\label{equ: C1}
\begin{aligned}
C_1&=\sqrt{2(h-V)}\big(s\cdot(\nabla VJ)\nabla^2f_1(\nabla VJ)^T+t\cdot(\nabla VJ)\nabla^2f_2(\nabla VJ)^T\big)\\
&\quad\ +(\nabla VJ)\nabla^2 V(\nabla VJ)^T.
\end{aligned}
\end{equation}
Now we shortly denote $f_{k,ij}=\partial_{x_i,x_j}f_k$, $V_i=\partial_{x_i}V$ and $V_{ij}=\partial_{x_i,x_j}V$. Then we obtain
$$\begin{aligned}
\det M&=2(h-V)\det(s\cdot\nabla^2f_1+t\cdot\nabla^2f_2)\\
&+\sqrt{2(h-V)}\cdot\big((s\cdot f_{1,11}+t\cdot f_{2,11})(V_{22}+|\partial_{x_2}F|^2)\\
&\qquad\qquad\quad \ \ -2(s\cdot f_{1,12}+t\cdot f_{2,12})(V_{12}+\partial_{x_1}F\cdot \partial_{x_2}F)\\
&\qquad\qquad\qquad\ +(s\cdot f_{1,22}+t\cdot f_{2,22})(V_{11}+|\partial_{x_1}F|^2)\big)\\
&+\det(\nabla^2 V+\nabla f_1\otimes \nabla f_1+\nabla f_2\otimes \nabla f_2),
\end{aligned}$$
where
$$\begin{aligned}
&\ \ \ \det(\nabla^2 V+\nabla f_1\otimes \nabla f_1+\nabla f_2\otimes \nabla f_2)\\
&=\det \nabla^2 V+(\nabla f_1J)\nabla^2V(\nabla f_1J)^T+(\nabla f_2J)\nabla^2V(\nabla f_2J)^T+(\nabla f_1(\nabla f_2J)^T)^2.
\end{aligned}$$
Using \eqref{equ: formula g}, we can also compute
$$\begin{aligned}
&\ \ \ -(\nabla f_1J)M(\nabla f_1J)^T-(\nabla f_2J)M(\nabla f_2J)^T+(\nabla f_1 (\nabla f_2J)^T)^2\\
&=-\sqrt{2(h-V)}\big((\nabla f_1J)(s\cdot \nabla^2 f_1+t\cdot \nabla^2 f_2)(\nabla f_1J)^T\\
&\qquad\qquad\qquad\ +(\nabla f_2J)(s\cdot \nabla^2 f_1+t\cdot \nabla^2 f_2)(\nabla f_2J)^T\big)\\
&\ \ \ -(\nabla f_1J)\nabla^2V(\nabla f_1J)^T-(\nabla f_2J)\nabla^2V(\nabla f_2J)^T-(\nabla f_1(\nabla f_2J)^T)^2.
\end{aligned}$$
The computation above gives
$$\begin{aligned}
C_2&=2(h-V)\det(s\cdot\nabla^2f_1+t\cdot\nabla^2f_2)\\
&+\sqrt{2(h-V)}\big(s\cdot(f_{1,11}V_{22}+ f_{1,22}V_{11}-2f_{1,12}V_{12})\\
&\qquad\qquad\quad\ \ +t\cdot(f_{2,11}V_{22}+f_{2,22}V_{11}-2f_{2,12}V_{12})\big)\\
&+\det \nabla^2V.
\end{aligned}$$
Therefore, we obtain
\begin{equation}\label{equ: C0}
\begin{aligned}
C_0&=4(h-V)^2\det(s\cdot\nabla^2f_1+t\cdot\nabla^2f_2)\\
&+\sqrt{2(h-V)}\big(2(h-V)\big(s\cdot(f_{1,11}V_{22}+f_{1,22}V_{11}-2f_{1,12}V_{12})\\
&\qquad\qquad\qquad\qquad\qquad +t\cdot(f_{2,11}V_{22}+f_{2,22}V_{11}-2f_{2,12}V_{12})\big)\\
&\qquad\qquad\qquad+s\cdot(\nabla VJ)\nabla^2f_1(\nabla VJ)^T+t\cdot(\nabla VJ)\nabla^2f_2(\nabla VJ)^T\big)\\
&+2(h-V)\det \nabla^2V+(\nabla VJ)\nabla^2 V(\nabla VJ)^T\\
&:=4(h-V)^2A_2+\sqrt{2(h-V)}(2(h-V)A_{13}+A_{11})+A_0,
\end{aligned}
\end{equation}
where
$$\begin{aligned}
A_2&:=s^2\cdot(f_{1,11}f_{1,22}-f_{1,12}^2)+t^2\cdot(f_{2,11}f_{2,22}-f_{2,12}^2)\\
&\quad\ +st\cdot(f_{1,11}f_{2,22}+f_{2,11}f_{1,22}-2f_{1,12}f_{2,12}),\\
A_{13}&:=s\cdot(f_{1,11}V_{22}+f_{1,22}V_{11}-2f_{1,12}V_{12})\\
&\quad\ +t\cdot(f_{2,11}V_{22}+f_{2,22}V_{11}-2f_{2,12}V_{12}),\\
A_{11}&:=s\cdot(f_{1,22}V_1^2+f_{1,11}V_2^2-2f_{1,12}V_1V_2)\\
&\quad\ +t\cdot(f_{2,22}V_1^2+f_{2,11}V_2^2-2f_{2,12}V_1V_2),\\
A_0&:=2(h-V)(V_{11}V_{22}-V_{12}^2)+V_{22}V_1^2+V_{11}V_2^2-2V_{12}V_1V_2.
\end{aligned}$$
Note that $C_0$ admits the same sign as $\det W$, which only depends on $V$ and $\nabla^2 f_i, i=1,2$. Denoting $r=\sqrt{2(h-V)}$, one can check that $C_0=D_1+D_2$, where
$$\begin{aligned}
D_1&=-4(h-V)^2(s\cdot f_{1,12}+t\cdot f_{2,12})^2-2((h-V)V_{12}^2+V_{12}V_1V_2)\\
&\quad\ -2\sqrt{h-V}(2(h-V)V_{12}+V_1V_2)(s\cdot f_{1,12}+t\cdot f_{2,12})\\
&=-\left(2(h-V)(s\cdot f_{1,12}+t\cdot f_{2,12})+\frac{2(h-V)V_{12}+V_1V_2}{\sqrt{2(h-V)}}\right)^2+\frac{V_1^2V_2^2}{2(h-V)}.
\end{aligned}$$
and
$$\begin{aligned}
D_2&=4(h-V)^2(s^2\cdot f_{1,11}f_{1,22}+t^2\cdot f_{2,11}f_{2,22}+st\cdot (f_{1,11}f_{2,22}+f_{2,11}f_{1,22}))\\
&+\sqrt{2(h-V)}\big\{2(h-V)\big(s\cdot (f_{1,11} V_{22}+f_{1,22} V_{11})+t\cdot (f_{2,11} V_{22}+f_{2,22} V_{11})\big)\\
&\qquad\qquad\qquad\ +s\cdot(f_{1,22}V_{1}^2+f_{1,11}V_{2}^2)+t\cdot(f_{2,22}V_{1}^2+f_{2,11}V_{2}^2)\big\}\\
&+2(h-V)V_{11}V_{22}+V_{22}V_{1}^2+V_{11}V_{2}^2\\
&=\left(2(h-V)(s\cdot f_{1,11}+t\cdot f_{2,11})+\frac{2(h-V)V_{11}+V_{1}^2}{\sqrt{2(h-V)}}\right)\cdot\\
&\qquad \left(2(h-V)(s\cdot f_{1,22}+t\cdot f_{2,22})+\frac{2(h-V)V_{22}+V_{2}^2}{\sqrt{2(h-V)}}\right)-\frac{V_{1}^2V_{2}^2}{2(h-V)}.
\end{aligned}$$

We thus obtain
$$C_0=D_1+D_2= u_{11}(s,t)u_{22}(s,t)-u_{12}(s,t)^2=\det U_W.$$
where
$U_W=(u_{ij})_{2\times 2}$ and $u_{ij}(s,t)=r^2(s\cdot f_{1,ij}+t\cdot f_{2,ij})+rV_{ij}+r^{-1}V_i V_j$.

If $(y,x)$ is such that $x\in\partial \mathcal{H}$, i.e. $y_F=0$, then $X_1$ vanishes. Hence, as explained before, we consider the simpler frame
$$
X_{1b}=(1,0,0,0),\quad X_{2b}=(0,1,0,0),\quad X_{3b}=(0,0,V_2,-V_1).
$$
The tangent Hessian in this frame gives the following $3\times 3$ matrix
$$
W_b=(X_{ib}\nabla^2H X_{jb}^T)_{3\times 3}=\begin{pmatrix}
1 & 0 & \nabla f_1(\nabla VJ)^T\\
0 & 1 & \nabla f_2(\nabla VJ)^T\\
\nabla f_1(\nabla VJ)^T & \nabla f_1(\nabla VJ)^T & (\nabla VJ)M(\nabla VJ)^T
\end{pmatrix}
$$
Notice that $W_b$ is positive-definite if and only if $\det W_b>0$. A direct computation shows that
$$
\det W_b=(\nabla VJ)M(\nabla VJ)^T-(\nabla f_1( \nabla V J)^T)^2-(\nabla f_2 (\nabla V J)^T)^2=C_1,
$$
where $C_1$ was obtained above. Since $2(h-V)=0$ along $\partial \mathcal{H}$, we further obtain from \eqref{equ: C1}
$$\det W_b=(\nabla VJ)\nabla^2 V(\nabla VJ)^T=V_{22}V_{1}^2+V_{11}V_{2}^2-2V_{12}V_{1}V_{2},$$
which is independent of the magnetic field $F$.

We conclude in both cases $y_F\neq 0$ and $y_F=0$ that $W$ is positive-definite if and only if
\begin{equation}\label{equ: criterion formula}
C_0|_{\mathcal{M}}=C_1+|y_F|^2C_2=\det U_W|_{\mathcal M}>0,
\end{equation}
which also implies that
$C_0|_{\partial \mathcal{H}}=C_1|_{\partial \mathcal H}=\det W_b|_{\partial \mathcal{H}}>0$. Hence, the theorem follows.
\end{proof}

\begin{rem}\label{rem: decouple}
If $F$ and $V$ are decoupled, then $D_1=0$ and $C_0=u_{11}(s,t)u_{22}(s,t)-r^{-2}V_1^2V_2^2$.
If $F$ vanishes, then $C_0=A_0$ and the criterion boils down to the mechanical case considered in \cite{Sa1}.
\end{rem}

From \eqref{equ: U}, we observe that $\nabla V\otimes \nabla V$ admits two eigenvalues $\{\lambda_1=|\nabla V|^2,\lambda_2=0\}$, which correspond to the eigenvectors $\xi_1=\nabla V$ and $\xi_2=\nabla VJ=(V_2,-V_1)$, respectively. If
$$U_1:=r(s \nabla^2f_1+t\nabla^2f_2)+\nabla^2 V>0, $$
then $U_W$ is positive definite for $(y,x)\in \mathcal{M}$ with $x\in \mathcal{H}\setminus \partial \mathcal H$. Moreover, $\det U_W>0$ on $\mathcal M$. If $\nabla^2 F=0$, then the strict convexity of $V$ implies $\det U_W>0$.

If $F$ and $V$ are decoupled, then $U_1=\mathrm{diag}(c(x_1,x_2),d(x_1,x_2))$, where
\begin{equation}\label{equ: c and d}
\begin{aligned}
c(x_1,x_2)&:=r(x_1,x_2)(sf_{1,11}(x_1)+tf_{2,11}(x_1))+V_{11}(x_1),\\
d(x_1,x_2)&:=r(x_1,x_2)(sf_{1,22}(x_2)+tf_{2,22}(x_2))+V_{22}(x_2).
\end{aligned}
\end{equation}
If $c,d>0$, then $U_1$ is positive definite and $\det U_W>0$ naturally holds. In general, we have
$$\det U_W=r^{-2}\big((V_1^2+r^2c)(V_2^2+r^2d)-(V_1V_2)^2\big)=dV_1^2+cV_2^2+r^2cd.$$
Then $\mathrm{det}U_W>0$ is equivalent to $dV_1^2+cV_2^2+r^2cd>0$. In summary, we have the following corollary.

\begin{cor}\label{cor: a special case}
The following properties hold:
\begin{itemize}
\item[(i)] If $U_1:=r(s \nabla^2f_1+t\nabla^2f_2)+\nabla^2 V>0,$ then $\det U_W>0$ on $\mathcal M$. In particular, if $\nabla^2 F=0$, then $\nabla^2V|_{\mathcal H}>0$ implies $\det U_W>0$ on such points.
\item[(ii)] If $F$ and $V$ are decoupled, then
\begin{equation*}\label{equ: det U}
\det U_W=dV_1^2+cV_2^2+r^2cd.
\end{equation*}
where $c=r(sf_{1,11}+tf_{2,11})+V_{11}$ and $d=r(sf_{1,22}+tf_{2,22})+V_{22}$.
\end{itemize}
\end{cor}

\subsection{Elliptic-hyperbolic coordinates}\label{sec: elliptic coordinates and strict convexity}
We study the convexity of the critical level for $\mu=1/2$. From the generalized criterion proved in the previous subsection, we have reduced the local convexity problem to that of checking the positivity of a certain smooth function on the Hill region. In this section we prove that this function is positive for the critical energy $L_1(1/2)=-2$ and also for every lower values of energy.

Recall that the Hamiltonian of the circular restricted three-body problem as
$$
H_\mu(p,q)=\frac{1}{2}((p_1-q_2)^2+(p_2+q_1)^2)-\frac{\mu}{|q-(1-\mu)|}-\frac{1-\mu}{|q+\mu|}-\frac{1}{2}|q|^2.
$$
Shifting the $q_1$-variable in a way that the primaries stay at $\pm \frac{1}{2}$, and replacing $p_2$ by $p_2-1/2+\mu$, we obtain the equivalent Hamiltonian
$$\begin{aligned}
\bar H_\mu(p,q)=\frac{1}{2}((p_1-q_2)^2+(p_2+q_1)^2)-\frac{\mu}{|q-\frac{1}{2}|}-\frac{1-\mu}{|q+\frac{1}{2}|}-\frac{1}{2}\big|\frac{1}{2}-\mu+q\big|^2.
\end{aligned}$$
Consider the symplectic transformation given by elliptic coordinates
$$p_1=a_1y_1+b_1y_2,\quad p_2=a_2y_1+b_2y_2,\quad q_1=\frac{1}{2}\cosh x_1\cos x_2,\quad q_2=\frac{1}{2}\sinh x_1\sin x_2,$$
where $a_i,b_i, i=1,2,$ satisfy
$$\begin{aligned}
\begin{pmatrix}
a_1 & b_1 \\ a_2 & b_2
\end{pmatrix}&=
2\begin{pmatrix}
\sinh x_1 \cos x_2 & \cosh x_1 \sin x_2 \\
-\cosh x_1 \sin x_2 & \sinh x_1 \cos x_2
\end{pmatrix}^{-1}\\
&=\frac{2}{\cosh^2 x_1-\cos^2 x_2}
\begin{pmatrix}
\sinh x_1 \cos x_2 & -\cosh x_1 \sin x_2 \\
\cosh x_1 \sin x_2 & \sinh x_1 \cos x_2
\end{pmatrix}.
\end{aligned}$$
The regularized Hamiltonian $\hat H = \hat H_{\mu, h}$ in coordinates $(y,x)\in \R^2 \times (\R \times \R/2\pi \Z)$ becomes
$$
\begin{aligned}
\hat H(y,x)&=\frac{1}{4}(\cosh^2 x_1-\cos^2 x_2)(\bar H_\mu(p(y,x),q(x))-h)\\
&=\frac{1}{2}\left(\left(y_1+\frac{\sin 2x_2}{8}\right)^2+\left(y_2+\frac{\sinh 2x_1}{8}\right)^2\right)+V(x)+(1-2\mu)\hat V(x),
\end{aligned}
$$
where
$$
\begin{aligned}
V(x)&=-\frac{h\cosh^2x_1}{4}-\frac{\cosh x_1}{2}-\frac{\sinh^2(2x_1)}{128}+\frac{h\cos^2x_2}{4}-\frac{\sin^2(2x_2)}{128},\\
\hat V(x)&=\frac{\cos x_2}{2}-\frac{1}{16}(\frac{1}{2}-\mu+\cosh x_1 \cos x_2)(\cosh^2 x_1-\cos^2 x_2).
\end{aligned}$$
Notice that $\hat H$ smoothly depends on $\mu$ and $h$, and the regularized dynamics of $\bar H_\mu^{-1}(h)$ corresponds to the dynamics of $\hat H^{-1}(0)$ under a double covering map that identifies $(y,x)\sim -(y,x)$.

\subsection{Strict convexity} From now on we fix $\mu=1/2$. Then
$$
\begin{aligned}
\hat H(y,x)=\frac{1}{2}\left((y_1+f_1(x))^2+(y_2+f_2(x))^2\right)+V(x).
\end{aligned}
$$
The magnetic field is denoted $F:=(f_1,f_2)=\frac{1}{8}(\sin(2x_2),\sinh(2x_1))$ and the potential function  is decoupled
$$
V(x) = W_1(x_1) + W_2(x_2),
$$
where
$$
W_{1}(x_1)=-\frac{h\cosh^2x_1}{4}-\frac{\cosh x_1}{2}-\frac{\sinh^2(2x_1)}{128},\quad W_{2}(x_2)=\frac{h\cos^2x_2}{4}-\frac{\sin^2(2x_2)}{128}.
$$

For every fixed $h\leq -2$, a simple analysis shows that the critical points of $V$ are given by two nondegenerate minima at $(0,0)$ and $(0,\pi)$ with value $-\frac{1}{2}$, two saddle-points at $(0,\pm \pi/2)$ with value $\hat v_0=-h/4-1/2 \geq 0$, two saddles at $(\hat x_1,0)$ and $(\hat x_1, \pi)$ with value $\hat v_1> \hat v_0>0$ and two maxima $(\hat x_1, \pm \pi/2)$ with value $\hat v_2 > \hat v_1>0$. Notice that $V_1 = \partial_{x_1}W_1>0$ on $(0,\hat x_1)$. Hence, for $h=-2$, the Hill region $\mathcal{H} = \mathcal{H}^{-2}$ of the bounded subset of the regularized critical level restricted to $\R \times [-\pi/2,\pi/2]$ is bounded by the graphs of $x_1 = \pm f(x_2)$, where $f:[-\pi/2,\pi/2]\to [0,\bar x_1]$ is a smooth function on $[-\pi/2,\pi/2]$, and $\bar x_1 =f(0) < \hat x_1$ is the maximum of $f$, see Figure \ref{fig:critical level}. The boundary of $\mathcal{H}$ has singularities at $(0,\pm \pi/2)$. This subset corresponds to the component around the primary at $q=1/2$. Since for $\mu=\frac{1}{2}$ the components of the energy surface around the primaries are symmetric and thus admit the same dynamics, we only consider this component. Now observe that $\partial_h V = -\frac{1}{4}(\cosh^2 x_1-\cos^2 x_2)<0$ for every $(x_1,x_2)\in [-\bar x_1, \bar x_1] \times [-\pi/2, \pi/2] \setminus \{(0,0)\}$, we see that for $h<-2$, the Hill region $\mathcal{H}^h$ is an embedded disk contained in the interior of $\mathcal{H}^{-2}$ and, more generally,  $\mathcal{H}^{h_2} \subset \text{int}(\mathcal{H}^{h_1})$ for every $h_2 < h_1\leq -2$.

Due to the symmetry, we restrict the Hill region $\mathcal{H}= \mathcal{H}^h$ to the first quadrant $x_1, x_2\geq 0$. We denote this subset of $\mathcal{H}$ by $\mathcal{H}_1=\mathcal{H}_1^h\subset [0,\bar x_1] \times [0,\pi/2]$, where $\bar x_1(h):=\max \{x_1:(x_1,x_2)\in \mathcal{H}\}$ satisfies $V(\bar x_1,0)=0$. Hence $0<\bar x_1(h_2) \leq \bar x_1(h_1), \forall h_2<h_1\leq -2$. Notice that $W_{1}(x_1)\geq0$, $W_{2}(x_2)\leq 0$ for every $(x_1, x_2) \in[0,\bar x_1]\times [0,\frac{\pi}{2}]$. Also, $V_{1}=\partial_{x_1} W_1, V_{2}=\partial_{x_2}W_2> 0$ in the interior $[0,\bar x_1]\times [0,\frac{\pi}{2}]$. Let $r=r(x):=\sqrt{-2V(x)}$ be defined for $x\in \mathcal{H}$. Then $r$ decreases with both $x_1$ and $x_2$.

Let $(s_1,s_2):=(\cosh x_1,\cos x_2)$. Then we write $W_{1},W_{2}$ as functions of $(s_1,s_2)$
$$
W_{1}(s_1)=-\frac{s_1}{32} \left(16 + (-1 + 8 h) s_1 + s_1^3\right),\quad W_{2}(s_2)=\frac{s_2^2}{32}((-1 + 8 h) + s_2^2).
$$
In coordinates $(s_1,s_2)$, $\mathcal H_1$ lies in $[1,\bar s_1]\times [0,1]$, where $\bar s_1:=\cosh \bar x_1$.

If $h=-2$, then $\mathcal{H}_1$ contains the saddle-point $S_+:=(0, \pi/2)$ corresponding to $l_1(1/2)$. Moreover, $\bar x_1\approx 1.19954$ and $\bar s_1\approx 1.80996<1.9$ since $V(s_1=1.9, s_2=1)=19379/320000>0$.
\begin{figure}[hbpt]
    \centering
    \includegraphics[width=0.4\linewidth]{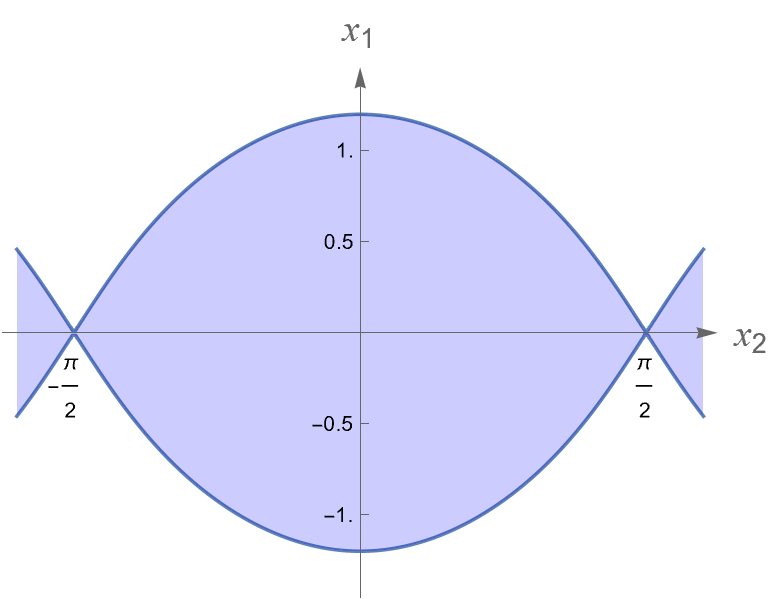}
    \caption{The Hill region $\mathcal{H} = \mathcal{H}^{-2}$ of the subset $\mathcal{M}= \mathcal{M}^{-2}$ of the regularized critical energy surface. }
    \label{fig:critical level}
\end{figure}

Let $\mathcal M=\mathcal{M}^h$ be the subset of $\hat H^{-1}(0)$ projecting to $\mathcal{H}^h$ for every $h\leq -2$.  We compute
\begin{equation}\label{equ: data of Copenhagen problem}
\begin{aligned}
r^2&=\frac{h}{2}(\cosh^2x_1-\cos^2x_2)+\cosh x_1+\frac{1}{64}(\sinh^2(2x_1)+\sin^2(2x_2)),\\
\nabla^2f_1&=\mathrm{diag}\big(0, -\sin x_2\cos x_2\big),\quad \nabla^2f_2=\mathrm{diag}\big(\sinh x_1 \cosh x_1, 0\big),\\
\nabla V &=-\frac{1}{32}\begin{pmatrix}
(16+\cosh x_1+16h\cosh x_1+\cosh(3x_1))\sinh x_1\\
(8h+\cos(2x_2))\sin(2x_2)
\end{pmatrix},\\
\nabla^2V&=-\frac{1}{16}\mathrm{diag} \big(
8\cosh x_1+8h\cosh(2x_1)+\cosh(4x_1), 8h\cos(2x_2)+\cos(4x_2)\big).
\end{aligned}
\end{equation}
Let $y_F:=y+F=re^{i\theta}$ and $(s,t):=(\cos \theta,\sin \theta)$. By Theorem \ref{thm: convexity formula}, the  local convexity of $\mathcal M$ is equivalent to $\mathrm{det} U_W|_{\mathcal M}>0$. The symmetry condition \eqref{equ: symmetries} implies that the matrix
$$U_W(\theta,x_1,x_2)=r^2(s\nabla^2f_1+t\nabla^2 f_2)+r\nabla^2V+r^{-1}\nabla V\otimes \nabla V,
$$
admits the following symmetries
\begin{equation}\label{equ: symmetry}
U_W(\theta,x_1,x_2)=U_W(\pi-\theta,x_1,-x_2)\quad \text{and}\quad U_W(\theta,x_1,x_2)=U_W(-\theta,-x_1,x_2).
\end{equation}
Therefore, it suffices to check the local convexity in $\mathcal M_1:=\{(\theta,x_1,x_2)\in \mathcal M: x\in \mathcal H_1\}$.

Our main goal is to prove the following theorem on the local convexity of $\mathcal{M}$.

\begin{thm}\label{thm: Copenhagen problem}
The following properties hold:
\begin{itemize}
\item[(i)] If $h=-2$, then $\det U_W|_{\mathcal M\setminus S}>0$, where $S\subset \mathcal{M}$ is formed by the singularities $S_\pm$ corresponding to $l_1(1/2)$, and projecting to $\bar S_\pm=(0,\pm\pi/2)$ in the $x$-plane, respectively.

\item[(ii)] If $h<-2$, then $\det U_W|_{\mathcal M}>0$.

\end{itemize}
\end{thm}

From \eqref{equ: data of Copenhagen problem}, we see that $f_{1,11}=f_{2,22}=0$ and $-f_{1,22},f_{2,11}\geq 0$ on $\mathcal M_1$. The quantities $c,d$ in \eqref{equ: c and d} are reduced to
$$
c_t(x):=r\cdot tf_{2,11}(x_1)+V_{11}(x_1),\quad
d_s(x):=r\cdot sf_{1,22}(x_2)+V_{22}(x_2).
$$
Hence, from Corollary \ref{cor: a special case}-(ii), we have
\begin{equation}\label{equ: det U1}
\det U_W=d_s(V_{1}^2+r^2c_t)+c_tV_{2}^2=c_t\big(V_{2}^2+d_s(V_{1}^2/c_t+r^2)\big).
\end{equation}

\begin{rem} \label{rem_non-convexity}
For every $0< \mu <1$, $\mu \neq \frac{1}{2}$, it is possible to show that at least one of the critical subsets $\mathcal M^e_{\mu,L_1(\mu)}$ and $\mathcal M^m_{\mu,L_1(\mu)}$ is not strictly convex in elliptic coordinates. Recall that the regularized Hamiltonian is $\hat H=\frac{1}{2}((y_1+f_1(x))^2+(y_2+f_2(x))^2)+V_\mu(x)$, where $V_\mu=V+(1-2\mu)\hat V$ and $f_1,f_2,V$ do not depend on $\mu$. We compute the determinant of $U_W=r^2(s\nabla^2f_1+t\nabla^2f_2)+r\nabla^2V_\mu+r^{-1}\nabla V_\mu\otimes \nabla V_\mu$ as
$$\det U_W=r^2(c_{\mu,t}d_{\mu,s}-V_{\mu,12}^2)+c_{\mu,t}V_{\mu,2}^2+d_{\mu,s}V_{\mu,1}^2-2V_{\mu,12}V_{\mu,1}V_{\mu,2},$$
where $c_{\mu,t}=r\cdot tf_{2,11}+V_{\mu,11}$ and $d_{\mu,s}=r\cdot sf_{1,22}+V_{\mu,22}$. In elliptic coordinates, the first Lagrange point $\hat l_1=(1/2-r_1,0)\in \C$ is given by $(0,\arccos(1-2r_1))$. Recall also that $\mu=\mu(r_1)$ and $L_1=L_1(\mu(r_1))$ for every $r_1\in(0,1)$. Restricting $\det U_W$ to $\{x_1=0\}$, we compute near $x_*:=\arccos(1-2r_1)$ that
$$
\det U_W(\theta,0,\hat x_2+x_*)=G(r_1,\cos \theta)\hat x_2^3+O(\hat x_2^4),
$$
where $G(r_1,s)=G_1(r_1)G_2(r_1,s)$ is such that
$$
\begin{aligned}
G_1(r_1)&=\frac{((1 - r_1) r_1)^{\frac{11}{2}}(1 - 2 r_1) (3 - 3r_1 + r_1^2) (1 + r_1 + r_1^2) ((3 - r_1)^2 + 4 r_1^2 + (2 - r_1) r_1^3)}{ 2 ((1 - r_1)^2 + (2 - r_1) r_1^3)^3} \\
G_2(r_1,s)&= 4((1 - r_1)^2 + (2 - r_1)r_1^3)^{\frac{1}{2}} ((3-r_1)^2 + 4r_1^2 + (2-r_1)r_1^3)^{\frac{1}{2}}s\\
&\quad -(15 - 14r_1 + 11r_1^2 + 6 r_1^3 - 3 r_1^4),\quad \forall s\in [-1,1].
\end{aligned}
$$
We easily check that $G_1$ admits only one zero $r_1=1/2$ in $(0,1)$. Moreover, we compute
$$\begin{aligned}
&(15 - 14r_1 + 11r_1^2 + 6 r_1^3 - 3 r_1^4)^2- 16((1 - r_1)^2 + (2 - r_1)r_1^3)((3-r_1)^2 + 4r_1^2 + (2-r_1)r_1^3)\\
&=\frac{19433}{256} + \frac{123 \nu^2}{16} + \frac{307 \nu^4}{8} + 51 \nu^6 - 7 \nu^8>0, \quad \forall r_1=\nu+\frac{1}{2}\in (0,1).
\end{aligned}$$
Hence, $G_2<0$ for every $(r_1,s)\in (0,1)\times [-1,1]$. We conclude that for every $0<\mu<1$, $\mu\neq \frac{1}{2}$, $G$ never vanishes. In particular, $\det U_W(\theta,0,x)$ has opposite signs on each side of $(0,x_*)$. Hence, for every $\mu\neq 1/2$, $\mathcal M^e_{\mu,L_1(\mu)}$ and $\mathcal M^m_{\mu,L_1(\mu)}$ cannot be both strictly convex. 
\end{rem}

To prove Theorem \ref{thm: Copenhagen problem}, we start proving the following lemma.

\begin{lem}\label{lem: c>0}
For every $h\leq -2$ and $t\in [-1,1]$, we have
$c_t\geq c_{-1}=-rf_{2,11}+V_{11}>0$ on $\mathcal H_1$. As a consequence, $\det U_W|_{\mathcal M\setminus S}>0$ if and only if $I:=V_{2}^2+d_s(V_{1}^2/c_t+r^2)>0$ on $\mathcal M_1\setminus S$.
\end{lem}

\begin{proof}
We first consider $h=-2$. Since $f_{2,11}(s_1)=s_1\sqrt{s_1^2-1}$ is positive on $(1,+\infty)$ and $r=r(s_1,s_2)$ is increasing with $s_2$, we see for a fixed $s_1\geq 1$ that $c_t(s_1,s_2)\geq c_{-1}(s_1,s_2)\geq c_{-1}(s_1,1)$. Hence, it is enough to prove that
$c_{-1}(s_1,1)=\frac{1}{16}(-17 - 8 s_1 + 40 s_1^2 - 8 s_1^4 - 4 s_1 (s_1^2 - 1)^{1/2} (16 + 16 s_1 - 17 s_1^2 + s_1^4)^{1/2})>0$ for every $s_1\in [1, \bar s_1]$. Indeed, let $v=s_1-1\in [0,\bar s_1-1]\subset [0,1]$. We compute
$$\begin{aligned}
&\quad\ (-17 - 8 s_1 + 40 s_1^2 - 8 s_1^4)^2-16s_1^2(s_1^2 - 1)(16 + 16 s_1 - 17 s_1^2 + s_1^4)\\
&=49 + 48 v + 16 v^2 \big( 41 (1 - v)^4 + 7 v (1 - v)^5 + 25 (1 - v)^4 v^2 + 29 v^6 \\
&\quad\  + 13 v (3 - 4 v)^2 + (1 - v) v^3 (14 - 45 v + 44 v^2)\big)>0,\quad \forall v\in[0,1].
\end{aligned}$$
In particular, we see that $c_{-1}(1,s_2)=7/16>0$ for every $s_2\in[0,1]$. Therefore, $c_{-1}(s_1,s_2)\geq c_{-1}(s_1,1)>0$ for every $(s_1,s_2)\in \mathcal H_1\setminus S$. Notice that $c_{-1}$ is well-defined for $(s_1,s_2) = (1,0)$ corresponding to $\bar S_+$, and is equal to $7/16>0$ at this point. Hence, from the symmetries \eqref{equ: symmetry} and the expression \eqref{equ: det U1}, the proof is complete for $h=-2$.

For $h<-2$, we change the previous notation to $\mathcal H^h,\mathcal H^h_1,r_h,V_{i,h},V_{ii,h},c^h_t,d^h_s$ for $i=1,2$. Notice that $\mathcal H^h\subset \text{int}(\mathcal H^{-2})$ for every $h< -2$. From \eqref{equ: data of Copenhagen problem}, we observe that $r_h^2(x)$ increases with $h$ for every $x\in \mathcal H^h$, $\nabla^2f_1,\nabla^2f_2$ are independent of $h$, and $V_{1,h},V_{2,h},V_{11,h}$ decrease with $h$ on $\mathcal H_1^h$. Hence, $c_{t}^h(x)\geq c_{-1}^h(x)=-r_h(x)f_{2,11}(x)+V_{11,h}(x)\geq c_{-1}^{-2}(x)>0$ for every $t\in[-1,1]$ and $x\in \mathcal H^h_1$. Finally, using \eqref{equ: symmetry} and \eqref{equ: det U1}, the lemma follows.
\end{proof}

Next we prove Theorem \ref{thm: Copenhagen problem} using Theorem \ref{thm: convexity formula}, Lemmas \ref{lem: c>0}, \ref{lem: I0>0}, and a monotonicity argument.

\begin{figure}[hbpt]
    \centering
    \includegraphics[width=0.4\linewidth]{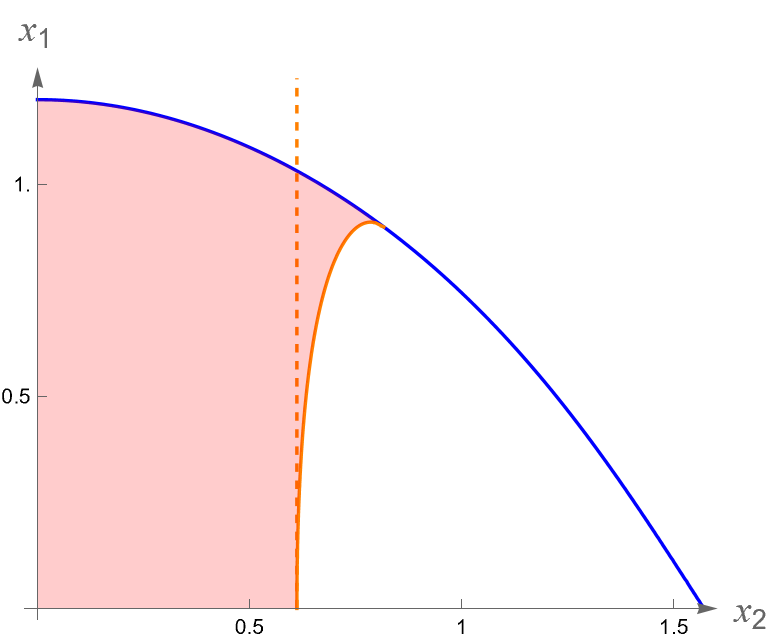}
    \caption{The boundary curve $\partial \mathcal H\cap \mathcal H_1$ (blue), the region $\{d_1>0\}\cap \mathcal H_1$ (red) for $h=-2$, the set $\{d_1=0\}$ (orange) and the set $\{D=0\}$ (orange dashed). }
    \label{fig:dzero}
\end{figure}

\begin{proof}[Proof of Theorem \ref{thm: Copenhagen problem}] We keep the previous notations $\mathcal M^h,\mathcal M^h_1,\mathcal H^h$ and $\mathcal H^h_1$. Recall that  $(s,t)=(\cos\theta,\sin \theta)$ and $I(\theta,x)=V_2^2(x_2)+d_s(x)(V_1^2(x_1)/c_t(x)+r^2(x))$. Here, $(\theta,x)\in \mathcal{M}_1 \setminus S$. By Lemma~\ref{lem: c>0}, we know that if $d_s(x)>0$, then $I(\theta,x)>0$ and thus $\det U_W(\theta,x)>0$. See the region $\mathcal H_1^{-2}\cap \{d_1> 0\}$ in Figure \ref{fig:dzero}. We thus restrict to $\{(\theta,x)\in \mathcal M_1^h|d_s(x)\leq 0\}$.
In this region, $I\geq V_2^2+d_s(V_1^2/c_{-1} +r^2)$ since $c_t \geq c_{-1}>0$ by Lemma \ref{lem: c>0}. Moreover, since $f_{1,22}\leq 0$ and $r(x)$ decreases with $x_1$ for every fixed $x_2$, we have $\hat D(x_2,h):=d_1(0,x_2)\leq d_1(x)\leq d_s(x)$ on $\mathcal H_1^h$ for every $s\in[-1,1]$. Notice that $\hat D$ depends only on $\cos x_2$ and $h$. Since $W_0(x):=V_1^2/c_{-1}+r^2>0$ on $\mathcal H_1^h\setminus S$, it is sufficient to show that $I_0^h(x):=V_2^2+\hat D\cdot W_0>0$ on $(\mathcal H_1^h\setminus S)\cap \{\hat D(x_2,h) <0\}$ for every $h\leq -2$. Let $(s_1,s_2):=(\cosh x_1,\cos x_2)$ as before. We compute
$$
16\hat D=-1 + 8 s_2^2 - 8 s_2^4 + h (8 - 16 s_2^2) - 4 s_2 (1 - s_2^2)^{\frac{1}{2}} (16 + s_2^2 - s_2^4 + 8 h (1-s_2^2))^{\frac{1}{2}}.
$$
Clearly, $I_0^h|_{\mathcal H_1^h\setminus S}>0$ whenever $\hat D\geq 0$, since $\hat D|_{\{x_2=0\}}=(-1-8h)/16>0$ for every $h\leq -2$.

Now we compute
$$
\begin{aligned}
\partial_h W_0(s_1,s_2)&=\frac{1}{512c_{-1}^2}\bigg(\left(\frac{s_1(s_1^2 - s_2^2)}{2r(s_1,s_2)}+(s_1^2-1)^\frac{1}{2}\right)(s_1^2-1)^{\frac{3}{2}}(8 + (-1 + 8 h) s_1 + 2 s_1^3)^2\\
&\qquad\qquad\ \ + (s_1^2-1) \big( s_1 (8 + (-1 + 8 h) s_1 + 2 s_1^3) + 16c_{-1}\big)^2\bigg) + \frac{1}{2}(1 - s_2^2).
\end{aligned}
$$

We see that $\partial_h W_0\geq 0$ for every $(s_1,s_2)\in \mathcal H_1^h$. Moreover, from \eqref{equ: data of Copenhagen problem}, we see that $r^2(x)$ increases with $h$ for every $x\in \mathcal H^h$, $\nabla^2f_1,\nabla^2f_2$ are independent of $h$, and $V_1,V_2,V_{11}$ all decrease with $h$ on $\mathcal H_1^h$. Since $V_1,V_2\geq 0$ on $\mathcal H_1^h$ for every $h\leq -2$, we see that $V_1^2,V_2^2$ also decrease with $h$ on $\mathcal H_1^h$. Hence, for every $(s_1,s_2)\in \Omega_-^h:=\{(s_1,s_2)\in \mathcal H_1^h: \hat D(s_2,h)\leq 0, \partial_h \hat D(s_2,h)\leq 0\}$, we have
$$
\partial_h I_0^h=\partial_h (V_{2}^2) + \partial_h \hat D\cdot W_0 + \hat D \cdot \partial_hW_0\leq 0.
$$

Now we aim to clarify the region $\partial_h \hat D \leq 0$ in terms $s_2$ and $h$. Write $V=V(s_1,s_2)$. We solve $V(1,s_2)=0$ for $h$ to obtain
$$\underline{h}(s_2):=-\frac{16 + s_2^2 - s_2^4}{8(1 - s_2^2)},$$
which is decreasing on $[0,1]$ and satisfies $\underline h(0)=-2$. Therefore, $\mathcal H^h_1\subset \{(s_1,s_2)\in[0,\bar s_1]\times[\bar s_2,1]\}$, where $\bar s_1(h)$ solves $V(\bar s_1,1)=0$ and $\bar s_2(h)$ is the inverse of $\underline{h}(s_2)$. Let $f(s_2,h):=-32V(1,s_2)=16 + s_2^2 - s_2^4 + 8 h (1-s_2^2)$. We see that $f>0$ for every $s_2 \in (\bar s_2, 1]$ and $h\leq -2$.  Then we compute $\partial_h \hat D(s_2,h)=\frac{1}{2} - s_2^2 - s_2 (1 - s_2^2)^\frac{3}{2}f^{-\frac{1}{2}}$. We solve $\partial_h \hat D(s_2,h)=0$ for $h$ to obtain
$$
\bar h(s_2)=\frac{16 - 67 s_2^2 + 71 s_2^4 - 4 s_2^6}{8 (s_2^2-1) (-1 + 2 s_2^2)^2},\quad \forall s_2\in[0,1)\setminus\{2^{-\frac{1}{2}}\}.
$$
In particular, $\bar h(0)=\bar h(s_{20})=-2$, where $s_{20}:= \frac{1}{2} (\frac{1}{30} (57 - \sqrt{129}))^{\frac{1}{2}} \approx0.616726<5/8$. Since $\partial_{hh}^2\hat D(s_2,h)=4s_2 (1 - s_2^2)^\frac{5}{2}f^{-\frac{3}{2}}>0$, we know that $\partial_h \hat D$ is increasing with $h\in [\underline h(s_2),-2]$. Moreover, since $f(s_2,\underline h(s_2))=0$ and $\partial_h \hat D(s_{20},-2)=0$, we obtain $\partial_h \hat D(s_2,\underline{h}(s_2))=-\infty$ and $\partial_h \hat D(s_2,-2)=\frac{1}{2} - s_2^2 - (1 - s_2^2)^\frac{3}{2}(17 - s_2^2)^{-\frac{1}{2}}$, which is positive on $(0,s_{20}]$ and negative on $[s_{20},1]$. This implies that for every fixed $s_2\in (0, s_{20}]$,  $\partial_h \hat D(s_2,h) < 0$ if and only if
$h \in (\underline{h}(s_2), \bar h(s_2))$. Moreover,  $\partial_h \hat D(s_2,h) > 0$ if and only if $h\in (\bar h(s_2),-2]$. For every fixed $s_2\in [s_{20}, 1]$, $\partial_h \hat D(s_2,h) < 0$ if and only if $h\in (\underline{h}(s_2), -2].$  See Figure \eqref{fig: hatD} for the regions $\{\partial_h \hat D< 0\}$ and $\{\partial_h \hat D > 0\}$.
\begin{figure}[hbpt]
    \centering
    \includegraphics[width=0.4\linewidth]{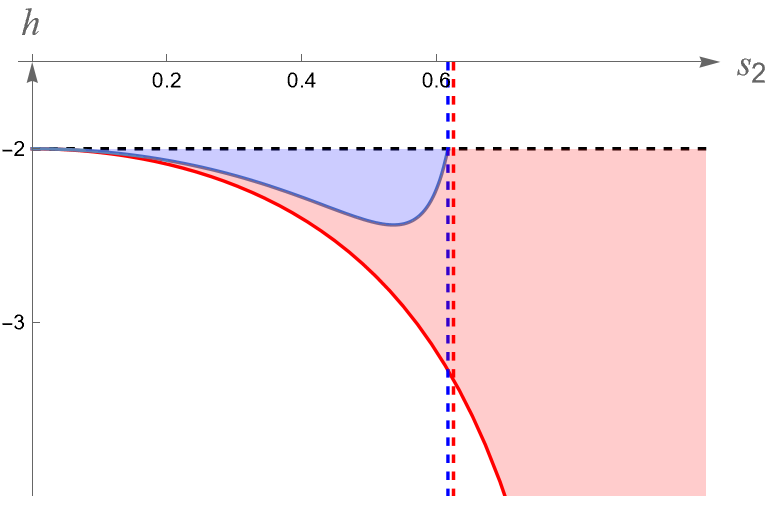}
    \caption{The regions $\{\partial_h \hat D<0\}$ (red), $\{\partial_h \hat D>0\}$ (blue), in which the blue and red curve denote the graph of $\bar h$ on $[0,s_{20}]$ and of $\underline{h}$ on $[0,1]$. The blue and red dashed lines are $s_2=s_{20}$ and $s_2=5/8$, respectively.}
    \label{fig: hatD}
\end{figure}

Now we consider the region $\Omega_+^h:=\{(s_1,s_2)\in\mathcal H^h_1| \hat D\leq 0,\partial_h \hat D\geq 0\}$, where $h\in [\bar h(s_2),-2]$ for every $s_2\in [0,s_{20}]$. We compute
$$
\begin{aligned}
\partial_h I_0^h&=\partial_h (V_{2}^2) + \partial_h \hat D\cdot W_0 + \hat D \cdot \partial_hW_0\\
&=\partial_h (V_{2}^2)+\frac{1}{2}((1 - s_2^2) \hat W_0 - s_2^2 W_0)-s_2 (1 - s_2^2)^{\frac{3}{2}}f^{-\frac{1}{2}} W_0\\
&\quad\ + \frac{\hat D}{512c_{-1}^2}\bigg((\frac{s_1(s_1^2 - s_2^2)}{2r(s_1,s_2)} + (s_1^2-1)^{\frac{1}{2}})(s_1^2-1)^{\frac{3}{2}} (8 + (-1 + 8 h) s_1 + 2 s_1^3)^2 \\
&\qquad\qquad\qquad\ +(s_1^2-1) (s_1 (8 + (-1 + 8 h) s_1 + 2 s_1^3) + 16 c_{-1})^2\bigg).
\end{aligned}$$
where $\hat W_0(s_1,s_2,h):=W_0(s_1,s_2,h) + \hat D(s_2,h)$. We see that all the terms are non-positive except $\hat W_0$ on $\Omega_+^h$. Since $\partial_hW_0\geq 0$ and $\partial_h \hat D\geq 0$ on $\{\partial_h \hat D\geq 0\}$, we have $\hat W_0(s_1,s_2,h)\leq \hat W_0(s_1,s_2,-2)$ for every $\bar h(s_2)\leq h\leq -2$. We shall prove that  $\hat W_0(s_1,s_2,-2)<0$ on $\Omega_+^h$.

We first compute that
$$\begin{aligned}
\partial_{s_2}\hat W_0(s_1,s_2,-2)&=
\frac{s_2}{8} \bigg(37 - 18 s_2^2 + (20 - 4 (1 - s_2^2)^{\frac{1}{2}} (17 - s_2^2)^{\frac{1}{2}}) + \frac{4 s_2^2 (9 - s_2^2)}{(1 - s_2^2)^{\frac{1}{2}} (17 - s_2^2)^{\frac{1}{2}}}\\
&\quad\ + \frac{s_1 (-1 + s_1^2)^{\frac{3}{2}}}{512r(s_1,s_2)c_{-1}^2}(8 - 17 s_1 + 2 s_1^3)^2 (17 - 2 s_2^2)\bigg)\geq 0,
\end{aligned}$$
for every $(s_1,s_2)\in\mathcal H_1^{-2}\subset [1,\bar s_1] \times [0,1]$. Then $\hat W_0(s_1,s_2,-2)\leq \hat W_0(s_1,s_{20},-2)\leq \hat W_0(s_1,\frac{5}{8},-2)$ on $\Omega_+^h$. Now we aim to show that $\hat W_0(s_1,5/8,-2)<0$ for every $s_1\in [1,\hat s_{1}]$, where $\hat s_1\approx 1.37116$ solves $V(\hat s_1,5/8,-2)=0$. We first compute
\begin{equation}\label{equ: E 5/8}
\begin{aligned}
&\quad\ 2^{20} c_{-1}(s_1,5/8)\hat W_0(s_1,5/8,-2)=E_0(s_1) +4 s_1 (100 \sqrt{41457} - 20361) J(s_1)\\
&+4 s_1 (4418 - 65536 s_1 + 69632 s_1^2 - 4096 s_1^4) (J(s_1) - (5/4 - 8 (s_1 - 5/4)^2))\\
&+ (100 \sqrt{41457} - 20360) (17 + 8 s_1 - 40 s_1^2 + 8 s_1^4).
\end{aligned}
\end{equation}
where
$J(s_1):=4(s_1^2-1)^{\frac{1}{2}}r(s_1,5/8)=(s_1^2-1)^{\frac{1}{2}}\left(\frac{26575}{4096} + 16 s_1 - 17 s_1^2 + s_1^4 \right)^{\frac{1}{2}}\geq 0$ for every $s_1\in[1,\hat s_1]$, and
$$
\begin{aligned}
E_0(s_1)&=-187055 - 163474 s_1 + 2863736 s_1^2 - 6584384 s_1^3 + 6310408 s_1^4 - 2469888 s_1^5 \\
&\quad\ + 98304 s_1^6 + 131072 s_1^7 - 16384 s_1^8.
\end{aligned}
$$

We see that $20360<100 \sqrt{41457}\approx 20360.99<20361$, $17 + 8 s_1 - 40 s_1^2 + 8 s_1^4 = -7 (3 - 2 s_1) -
41 (s_1 - 1) - (s_1 - 1) (3 - 2 s_1) (-1 + 10 s_1 + 4 s_1^2)<0$ for every $s_1\in [1,3/2]$ and $4418 - 65536 s_1 + 69632 s_1^2 - 4096 s_1^4 = 4418 +
4096 (s_1 - 1) s_1 (16 - s_1 - s_1^2)>0$ for every $s_1\in [1,2]$. Moreover,
$$\begin{aligned}
&\quad\ (5/4 - 8 (s_1 - 5/4)^2)^2-J(s_1)^2
=(s_1-1)^4\left(\frac{1549675}{4096} - \frac{1048977 s_1}{2048} + \frac{713883 s_1^2}{4096}\right) \\
&+ (2 - s_1)^2\left( \frac{28749}{4096} (6 - 5s_1)^2 (s_1 - 1)^2  + \frac{2993}{2048} (7 - 6s_1)^2 (s_1 - 1) + \frac{9}{64} (15 - 13s_1)^2\right) \\
&+ \frac{44897}{2048} (2 - s_1) (s_1-1)^3 (5 - 4s_1)^2 >0,\quad \forall s_1\in[1,2].
\end{aligned}$$
We rewrite $E_0$ as
$$
\begin{aligned}
E_0(1+v)&=-1024 v^3 (1 - 2 v)^3 (5 - v) (13 + 2 v) - 189054 v^2 (1 - 2 v)^3\\
&\quad - 185596 v^3 (1 - 2 v)^2- 5817 v (1 - 2 v)^2 - 29434 v^2 (1 - 2 v) \\
&\quad - 17665 (1 - 2 v) (1 - 6 v)^2 - 162199 v (4 v - 1)^2<0,\quad v\in[0,1/2].
\end{aligned}$$

From the estimates above, we see that all the terms in \eqref{equ: E 5/8} are negative in $[1,\hat s_1]\subset [1,3/2]$. We conclude that $\hat W_0(s_1,s_2,h)<\hat W_0(s_1,\frac{5}{8},-2)<0$ on $\Omega_+^h$. Combining both cases $\Omega_+^h \cup \Omega_-^h$, we obtain $\partial_h I^h_0\leq 0$ in $\{\hat D\leq 0\}$.

Let $h_*(s_2,h):=\min\{h_0\in [h,-2]:\hat D(s_2,h_0)\geq 0\}$, with the convention that $h_*(s_2,h)=-2$ if $\hat D(s_2,h_0)<0$ for every $h_0\in [h,-2]$. We know that $I_0^h>0$ on $\{\hat D\geq 0\}\setminus \{\bar S_+\}$ and $\partial_h I_0 \leq 0$ on $\{\hat D \leq 0\}$. Moreover, by Lemma \ref{lem: I0>0} below, $I^{-2}_0>0$ on $\mathcal H_1^{-2}\setminus S$. Hence, if $\hat D(s_2,h) <0$, then $I_0^h(s_1,s_2) \geq  I_0^{h_*}(s_1,s_2)>0$. Notice that in this case, we have $\mathcal{H}_1^h \subset \mathcal{H}_1^{h_*}$. We conclude that $I^h_0(x)>0$ for every $x\in \mathcal H_1^h\setminus S$. The proof of Theorem \ref{thm: Copenhagen problem} is complete.
\end{proof}

\begin{lem}\label{lem: I0>0}
Fix $h=-2$, and let
$I_0(x):=V_2(x_2)^2+D(x_2)(V_1(x_1)^2/c_{-1}(x)+r^2(x))$ on $\mathcal H_1$, where
$D(x_2):=d_{1}(0,x_2)=r(0,x_2)f_{1,22}(x_2) + V_{22}(x_2)$. Then $I_0(x)>0$ for every $x\in \mathcal H_1\setminus S$. Moreover, $I_0(x)=O(\cos^4 x_2)$ uniformly in $x_1$ near $S$.
\end{lem}

\begin{proof}
If $x=(x_1,x_2)\in \mathcal{H}_1 \setminus S$ and $D(x_2)\geq 0$, then $I_0(x)>0$ naturally holds, see Lemma \ref{lem: c>0}. In fact, $V_2(x_2)$ only vanishes at $x_2=0$, where $D=15/16$. Therefore, it is sufficient to consider $x\in \mathcal H_1\cap \{D(x)<0\}$. See Figure \ref{fig:dzero} for the curve $\{D(x)=0\}$. Let $(s_1,s_2)=(\cosh x_1,\cos x_2)$ as before and write $D=D(s_2)$, where $s_2\in [0,1]$. We compute
\begin{equation}\label{equ: D}
\begin{aligned}
D(s_2)=\frac{1}{16}(-17 - 8 s_2^4 + 40s_2^2 - 4s_2^2(1-s_2^2)^{1/2}(17-s_2^2)^{1/2}).
\end{aligned}
\end{equation}
Hence,
$$
D'(s_2)=\frac{s_2\big(8+(1-s_2^2)^{1/2}((10 - 4 s_2^2)(17-s_2^2)^{1/2} - (25 - 2s_2^2) (1-s_2^2)^{1/2})\big)}{2(1-s_2^2)^{1/2}(17-s_2^2)^{1/2}}.
$$
Since $ (10 - 4 s_2^2)^2(17 - s_2^2) - (25 - 2 s_2^2)^2(1 - s_2^2) = 1075 - 735 s_2^2 + 248 s_2^4 - 12 s_2^6>0$ for every $s_2\in [0,1]$, we have $D'(s_2)>0$, and thus $D(s_2)$ is an increasing functions with a unique zero $\hat s_2\in (0,1)$, since $D(0)=-17/16$ and $D(1)=15/16$. We also know that $D(s_2)<0, \forall s_2 \in[0,\hat s_2)$. It will be important later to know that $\hat s_2 < 0.82$ since $D(0.82)>0$ as one easily checks. Hence we only need to prove that $I_0(s_1,s_2)>0$ in $\Omega:=\{(s_1,s_2)\in[1,\bar s_1]\times[0,\hat s_2],V(s_1,s_2)\leq 0\}\subset \mathcal H_1$. Notice that $\Omega$ corresponds to the region of $\mathcal{H}_1$, where $D\leq 0.$

From $V(s_1,s_2)=\frac{1}{32}(-s_1(16 - 17 s_1 + s_1^3) + s_2^4 -17 s_2^2)=0$, we see that
$$
s_1=1+t_0(s_1,s_2)^2s_2^2,\quad t_0=\sqrt{\frac{17 - s_2^2}{s_1 (16 - s_1 - s_1^2)}},
$$
where $t_0(s_1,s_2)\leq\sqrt{17/14}<6/5$ on $[1,\bar s_1]\times[0,1]$. Notice that $s_1(16-s_1-s_1^2)$ is increasing in $s_1\in [1,\bar s_1]$. Consider the change of coordinates $(s_2,c)\mapsto (s_1,s_2)$, where $s_1=1+c^2s_2^2$, for every $(s_2,c)\in (0,1]\times [0,\sqrt{17/14}]$. In new coordinates $(s_2,c)$, we have $\mathcal H_1\setminus S \subset (0,1]\times[0,\sqrt{17/14}]$. From the expression of $t_0$ above, we see that the curve $V=0$ in these coordinates, are given by $(s_2,t_0),$ where $t_0 < \sqrt{17/14}$. The constant curve $c(s_2)=\sqrt{17/14}, \forall s_2\in [0,1],$  in $x$-coordinates, is depicted in Figure \ref{fig:enlarged hill region}. We compute in new coordinates $(s_2,c)$
\begin{equation}\label{equ: remove deg}
64c_{-1}(s_2,c)I_0(s_2,c)=s_2^4E_1(s_2,c),
\end{equation}
where
\begin{equation}\label{E1}\begin{aligned}
E_1(s_2,c)&:= j_1(s_2) c_{-1}(s_2,c) - c^3 D(s_2) W_3(s_2,c),\\
16c_{-1}(s_2,c)&:=j_0(s_2, c) - 4 c s_2^2 (1 + c^2 s_2^2) j_2(s_2, c), \\
W_3(s_2,c)&:= j_3(s_2,c)j_2(s_2,c) + c j_4(s_2,c),
\end{aligned}
\end{equation}
and
$$
\begin{aligned}
j_0(s_2,c)&:=7 + 40 c^2 s_2^2 - 8 c^4 s_2^4 - 32 c^6 s_2^6 - 8 c^8 s_2^8,\\
j_1(s_2)&:=(85 - 26 s_2^2 +
    s_2^4 - (17 - s_2^2) (1-s_2^2)^{1/2}(17-s_2^2)^{1/2}),\\
j_2(s_2,c)&:=\big((2+c^2s_2^2)(-32s_2^{-2} V(1+c^2s_2^2,s_2)\big)^{1/2}\\
&:=\big((2 + c^2 s_2^2) (17 - 14 c^2 - s_2^2 - 11 c^4 s_2^2 + 4 c^6 s_2^4 + c^8 s_2^6)\big)^{1/2},\\
j_3(s_2,c)&:=(-14 + 3 c^2 s_2^2 + c^4 s_2^4) (1 + c^2s_2^2)^2,\\
j_4(s_2,c)&:=70 + 18 c^2 s_2^2 - 105 c^4 s_2^4 - 74 c^6 s_2^6 + 2 c^8 s_2^8 + 8 c^{10} s_2^{10} + c^{12} s_2^{12}.
\end{aligned}
$$
We have to prove that $E_1(s_2,c)>0$ for every $(s_2,c) \in \mathcal{H}_1\setminus S$ so that $D(s_2) < 0$. We start computing $E_1(0,0)=\frac{119}{16}(5-\sqrt{17})>0$. Since $E_1$ is continuous on $\mathcal H_1$, we see that $I_0=O(s_2^4)$ uniformly in $c$ near $(s_2,c)=(0,0)$.
\begin{figure}[hbpt]
\centering
\begin{minipage}{0.4\linewidth}
\includegraphics[width=\linewidth]{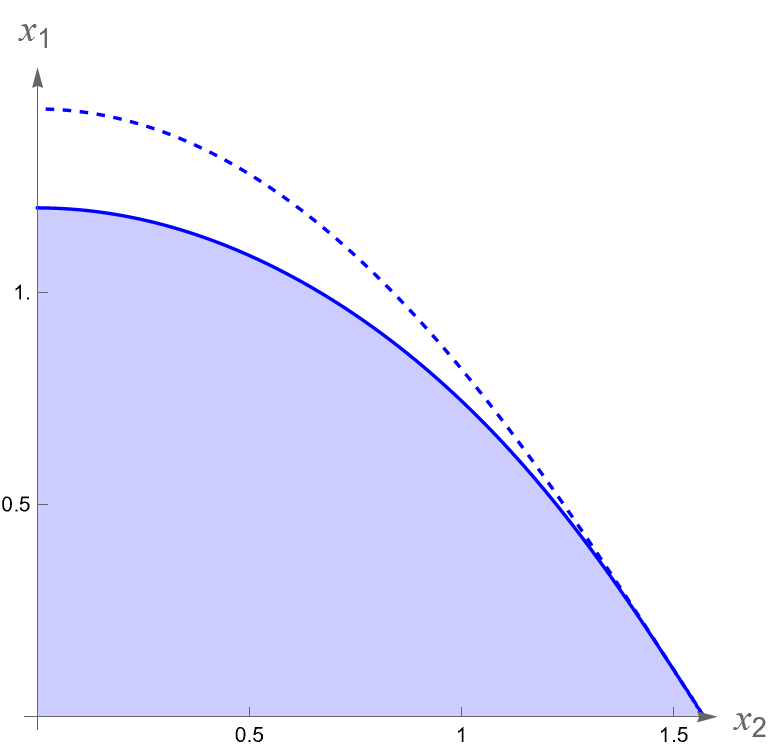}
\end{minipage}
\hspace{5mm}
\begin{minipage}{0.4\linewidth}
\includegraphics[width=\linewidth]{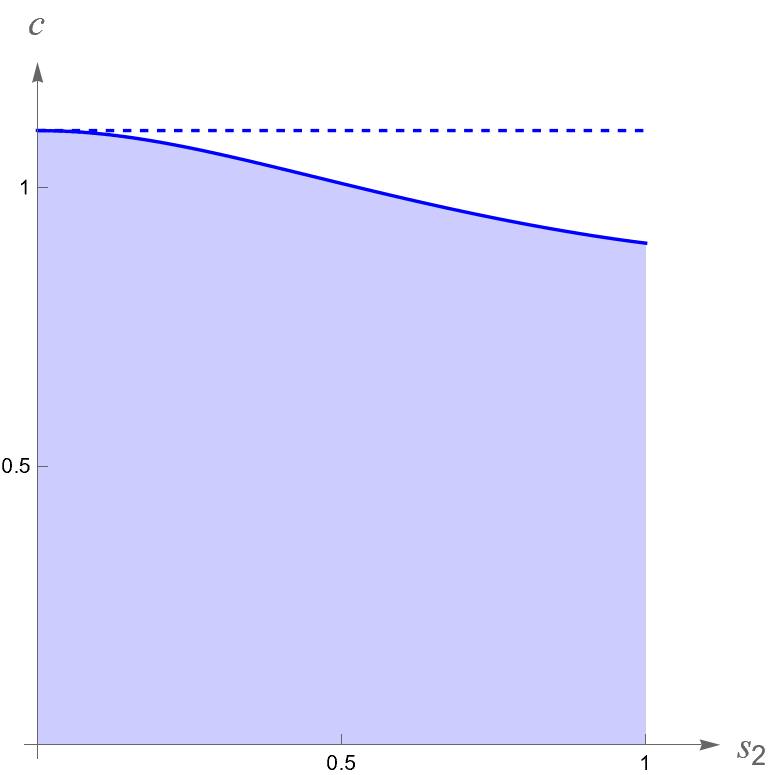}
\end{minipage}
\centering
\caption{The blue dashed curve is $s_1=1+\frac{17}{14}s_2^2$ (or $c=\sqrt{17/14}$) and the blue region is $\mathcal H_1$ in different coordinates.}
\label{fig:enlarged hill region}
\end{figure}

Before we continue with the proof of Lemma \ref{lem: I0>0}, we prove the following lemma.

\begin{lem}\label{lem: estimates of fi}
The following estimates hold:
\begin{itemize}

\item[(i)] $j_0(s_2,c)>0$ and $j_3(s_2,c)<0$ for every $(s_2,c)\in \mathcal{H}_1$.

\item[(ii)] $j_1(s_2)\geq 85 - 17^{3/2}>59/4$ and $j_4(s_2,c)>0$ for every $(s_2,c)\in \mathcal H_1\cap\{0\leq s_2\leq\hat s_2\}$.

\item[(iii)] $0\leq |j_2(s_2,c)|< j_{2,0}(s_2,c)$ for every $(s_2,c)\in [0,1]\times[0,\sqrt{17/14}]$, where
$$
j_{2,0}(s_2,c)=\frac{169}{32} - \left(c^2-\frac{1}{4}\right)\left(\frac{5}{2}+\frac{s_2^2}{4}+c^2\right) - \frac{c^4 s_2^2}{2}.
$$

\item[(iv)] $D_{-}(s_2)\leq 16D(s_2)\leq D_+(s_2)$ for every $s_2\in[0,1]$, where
$$\begin{aligned}
D_-(s_2)&= -17 - 8 s_2^4 + 40 s_2^2 - 4s_2^2 (33/8 - (2 s_2^2 + s_2^4)),\\
D_+(s_2)&= -17 - 8 s_2^4 + 40 s_2^2 - 4s_2^2 (1 - s_2^2) (4 + 2 s_2^2 + s_2^4).
\end{aligned}$$
\end{itemize}
\end{lem}
\begin{proof}
First, we rewrite $j_0$ as
$j_0(s_2,c)=g_0(1+c^2s_2^2)$, where
$g_0(s_1)=-17 - 8 s_1 + 40 s_1^2 - 8 s_1^4$. Since $g_0''(s_1)=80-96s_1^2<0$ for every $s_1\geq1$, we see that $g_0$ is a concave function on $[1,+\infty)$. Therefore, $g_0(1)=7$ and $g_0(1.9)=9929/1250>0$ which implies that $g_0(\bar s_1)>0$ since $1<\bar s_1 < 1.9$. We conclude that  $g_0(s_1)>0$ for every $s_1\in [1,\bar s_1]$ since $g_0$ is concave on the interval. Hence, $j_0>0$ on $\mathcal H_1$. Moreover, $j_3<0$ on $\mathcal H_1$, since in that region $s_2 \in [0,1]$ and $c\in [0, \sqrt{17/14}].$ This proves (i).

To prove (ii), we rewrite and estimate $j_1$ as follows
$$
\begin{aligned}
j_1(s_2)&=85 - 17^{3/2} +
 s_2^2 \left(-26 + s_2^2 + \frac{5780 - 918 s_2^2 + 52 s_2^4 - s_2^6}{17^{3/2} + (1 - s_2^2)^{1/2} (17 - s_2^2)^{3/2}}\right)\\
 &\geq 85 - 17^{3/2} +
 s_2^2 \left(s_2^2 + \frac{5780-26\cdot 2\cdot 17^{3/2} - 918 s_2^2 + 52 s_2^4 - s_2^6}{2\cdot 17^{3/2}}\right)\\
 &\geq 85 - 17^{3/2} +
 s_2^2 \left(s_2^2 + \frac{2000 - 918 s_2^2 + 52 s_2^4 - s_2^6}{2\cdot 17^{3/2}}\right)\\
 &\geq 85 - 17^{3/2}>\frac{59}{4},\quad \forall s_2\in[0,1].
\end{aligned}
$$

As for $j_4$, we rewrite it as $j_4(s_2,c)=g_4(1+c^2s_2^2)$, where
$$
g_4(s_1):=16 + 32 s_1 + 64 s_1^2 - 22 s_1^3 - 23 s_1^4 + 2 s_1^5 + s_1^6.
$$
We shall prove first that the points $(s_1,s_2)$ in $\mathcal{H}_1$ for which $D(s_2)<0$ are such that $0\leq s_1 < 1.6$. Then we show that $g_4$ is positive in this interval. Since $V(s_1,s_2)$ increases with $s_1\in [1,\bar s_1]$, decreases with $s_2\in[0,1]$ and $V(1.6,\hat s_2)>V(1.6,0.82)=0.012116305>0$, we see that $\max\{s_1:V(s_1,s_2)\leq 0, \forall 0\leq s_2\leq \hat s_2\}=\max\{s_1:V(s_1,\hat s_2)\leq 0\}<1.6$. Hence, it is sufficient to prove that $g_4(s_1)>0$ for every $s_1\in[1,1.6]$. We first see that $g_4''(v+1)=-210 - 444 v + 24 v^2 + 160 v^3 + 30 v^4<0$ for every $v\in[0,0.6]$, which means that $g_4(s_1)$ is a concave function on $[1,1.6]$. Using that $g_4(1)=70$ and $g_4(1.6)= 436624/15625>0$, we obtain $g_4(s_1)>0$ for every $s_1\in[1,1.6]$ as desired. Hence, (ii) follows.

To prove (iii), we first show that $j_{2,0}^2 - j_2^2>0$ on $\mathcal H_1$. Denote by $(l_2,c_1):=(s_2^2,c^2)$. Recall that $(s_2,c)\in [0,1] \times [0,\sqrt{17/14}]$. Hence $(l_2,c_1)\in [0,1] \times [0,17/14]$. We compute
$$
\begin{aligned}
256(j_{2,0}^2-j_2^2)&=\frac{905}{4} + 364 c_1 - 1728 c_1^2 + 1152 c_1^3 + 256 c_1^4 + (701 - 5180 c_1 + 7960 c_1^2  \\
&\quad\ + 704 c_1^3 + 256 c_1^4) l_2 + (1 + 248 c_1 + 832 c_1^3 + 64 c_1^4) l_2^2 - 1536 c_1^4 l_2^3 - 256 c_1^5 l_2^4\\
&=\frac{l_2(1 - l_2)}{2} \big(l_2(1 - l_2) j_{21} + 4 (1 - l_2)^2 j_{22} + 4 l_2^2 j_{23}\big) + (1 - l_2)^4 j_{24} + l_2^4 j_{25},
\end{aligned}
$$
where
$$\begin{aligned}
j_{21}&=6923 - 26216 c_1 + 27024 c_1^2 + 19712 c_1^3 + 4736 c_1^4,\\
j_{22}&=803 - 1862 c_1 + 524 c_1^2 + 2656 c_1^3 + 640 c_1^4,\\
j_{23}&=1505 - 6794 c_1 + 8484 c_1^2 + 4192 c1^3 + 192 c_1^4,\\
j_{24}&=905/4 + 364 c_1 - 1728 c_1^2 + 1152 c_1^3 + 256 c_1^4,\\
j_{25}&=3713/4 - 4568 c_1 + 6232 c_1^2 + 2688 c_1^3 - 960 c_1^4 - 256 c_1^5.
\end{aligned}$$
Let $v=14c_1/17\in[0,1]$. After manipulating these functions, we obtain
$$\begin{aligned}
j_{21}(v)&= \frac{1 - v}{49} \left(2 v (96528 - 580911 v + 1346231 v^2) + 12000 (3 - 10 v)^2  + 231227 (1 - v)^3\right)\\ &  + \frac{145322731}{2401} v^4,\\
j_{22}(v)&=\frac{v(1 - v)}{343} (326193 - 1061368 v + 1671464 v^2) + 803 (1 - v)^4 + \frac{13113085}{2401} v^4,\\
j_{23}(v)&=\frac{v(1 - v)}{343}(58359 - 325948 v + 3078524 v^2) + 100 (3 - 10 v)^2 + 605 (1 - v)^4 + \frac{21099315}{2401} v^4,\\
j_{24}(v)&=\frac{v^2}{9604} (4183326 - 8388128 v + 7465427 v^2) + 100v (3 - 5 v)^2 + 447v (1 - v)^3 + \frac{905}{4} (1 - v)^4,\\
j_{25}(v)&=\frac{v^2(1 - v)^2}{98} (125841 + 401512 v) + \frac{16643}{28} v(1 - v)^4 + \frac{67796525}{9604} v^4(1 - v)  + 100 (3 - 10 v)^2\\
&\quad\ + \frac{113}{4} (1 - v)^5 + \frac{115642367}{67228} v^5,\quad \forall v\in[0,1].
\end{aligned}$$
One can check that $j_{21},j_{22}, j_{23}, j_{24}$ and $j_{25}$ are positive functions on $[0,1]$. Therefore, $j_{2,0}^2-j_2^2>0$ for every $(s_2,c)\in[0,1]\times [0,\sqrt{17/14}]$. This proves (iii).

Finally, we prove (iv). From \eqref{equ: D} and the definition of $D_\pm$, it is sufficient to prove that
$$
d_+(s_2):=(1-s_2^2)(4+2s_2^2+s_2^4)\leq (1 - s_2^2)^{1/2}(17 - s_2^2)^{1/2}\leq d_-(s_2):=\frac{33}{8}-(2s_2^2+s_2^4),
$$
for every $s_2\in[0,1]$.
Taking the square of the non-negative expressions above, we see that
$$
\begin{aligned}
(1 - s_2^2)(17 - s_2^2)-d_+(s_2)^2&=(1 - s_2^4) (1 - 2 s_2^2 + 6 s_2^4 + 2 s_2^6 + s_2^8)>0,\\
d_-(s_2)^2-(1 - s_2^2)(17 - s_2^2)&=\frac{s_2^4}{4} (71 - 287 s_2^2 + 293 s_2^4) + \frac{3s_2^2}{4} (11 - 20 s_2^2)^2\\
&\quad\  +\frac{37s_2^2}{4} (1 - s_2^2)^3 + (1 - 2 s_2^2)^2>0,\quad \forall s_2\in[0,1].
\end{aligned}
$$
Item (iv) follows and the proof of Lemma \ref{lem: estimates of fi} is complete.
\end{proof}

Assume that $D(s_2)<0$. By Lemma \ref{lem: estimates of fi}, we obtain from \eqref{E1} the following estimates
\begin{equation}\label{equ: E2}
\begin{aligned}
16E_1 & = 16j_1c_{-1}-16c^3D\cdot W_3 \\ &\geq \frac{59}{4}\cdot  16c_{-1}-16c^3D\cdot W_3\\
&=\frac{59}{4}\big(j_0 - 4 c s_2^2 (1 + c^2 s_2^2)j_{2}\big) - 16 c^3 D\cdot j_3 j_{2} - 16 c^4 D\cdot j_4\\
&\geq \frac{59}{4}(j_0 - 4 c s_2^2 (1 + c^2 s_2^2)j_{2,0})  -  c^3 D_- \cdot j_3 j_{2,0} -  c^4D_+ \cdot j_4\\
&=:E_2(s_2,c),\quad (s_2,c)\in \Omega,
\end{aligned}
\end{equation}
where $j_{2,0}$ and $D_\pm$ are as in Lemma \ref{lem: estimates of fi}-(iii) and (iv).
The first inequality follows from Lemmas \ref{lem: c>0} and \ref{lem: estimates of fi}-(ii). The last inequality follows from Lemma \ref{lem: estimates of fi}-(iv) and the relations $16 c^3 D j_3 j_{2}\leq 16 c^3 D j_3 j_{2,0}\leq 16 c^3 D_- j_3 j_{2,0}$ on $\Omega$ due to Lemma \ref{lem: estimates of fi}-(i), (ii) and (iii). Finally, by Lemma \ref{lem: E2>0}, we have $16E_1(s_2,c)\geq E_2(s_2,c)>0,\forall (s_2,c)\in \Omega$. Notice that in coordinates $(s_2,c)$, $\Omega$ corresponds to the points in $\mathcal{H}_1$ lying in $[0,\hat s_2] \times [0,\sqrt{17/14}]$, or equivalently, $D\leq 0$. We conclude that $I_0>0$ on $\Omega$. As mentioned before, it follows from \eqref{equ: remove deg} that $I_0(s_2,c)=O(s_2^4)$ uniformly in $c$ near $s_2=0$. The proof of Lemma \ref{lem: I0>0} is completed.
\end{proof}

We still need to show that the positivity of sectional curvatures of $\mathcal M^e_{\mu,E}$ and $\mathcal M^e_{\mu,E}$ implies that both subsets are strictly convex. This local-to-global property is proved in the next section.

\subsection{Sectional curvatures and strict convexity}
In this section, we discuss the local-to-global relation between positive sectional curvatures and strict convexity of a topologically embedded three-sphere $M$ in $\R^4$, possibly with a finite set $S\subset M$ of singularities. In particular, $\R^4\setminus M$ bounds a compact set $B_M\subset \R^4$ and an unbounded open set $C_M=\R^4 \setminus B_M$. Assume that all sectional curvatures of $M\setminus S$ are positive. In particular, the principal curvatures of $ M\setminus S$ have the same sign, and we may assume they are all positive. If $M$ admits no singularity, i.e., $M$ is embedded, then \cite[Theorem 13.5]{Thorpe} implies that $M$ is strictly convex. This means that $B_M$ is convex and every hyperplane tangent to $M$ at a regular point has contact of order $1$, i.e., if a curve $\alpha(t)\in M \setminus S$ satisfies $\alpha'(0)\neq 0$, then $\alpha''(0)$ has a non-trivial projection to the normal direction to $M$ at $\alpha(0)$. The case $\# S =1$ was considered in \cite{Sa1}. Let us consider the case $\# S=2$.

\begin{prop}\label{prop_section curvature implies convexity}
Assume that $\#S =2$, i.e., $S=\{S_+, S_-\}$. If all sectional curvatures of $M\setminus S$ are positive, then $M$ is strictly convex.
\end{prop}

\begin{proof}
The convex hull of $S$ is a line segment. Hence, we find a sufficiently large sphere $S_R$ which contains $M$ in its interior and is tangent to $M$ at a regular point $p_0\in M\setminus S$. We can choose a continuous normal vector $N$ along $M\setminus S$ pointing towards $C_M$. In a rectangular system of coordinates near $p_0$, $M$ is locally the graph of a strictly convex function $f_{p_0}$ defined near the origin $0\in T_{p_0}M$, with positive values in the direction of $-N$, and so that $B_M$ corresponds to points above the graph of $f_{p_0}$. This is true since the same holds for $S_R$. This property is locally preserved and thus it must hold for every point in $M\setminus S$ since all sectional curvatures of $M\setminus S$ are positive and $M\setminus S$ is connected.

To prove that $B_M$ is convex, it is enough to show that given $x,y\in M$, the line segment $xy$ connecting $x$ to $y$ is contained in $B_M$. Fix $x\in M \setminus S$. Then $xy \subset B_M$ for every $y$ sufficiently close to $x$ by the local property of $M$ near $x$ proved before. Now notice that the two lines connecting $x$ to $S_+$ and $S_-$ intersect $M$ in at most a countable subset $G\subset M$. This follows from the fact that $M\setminus S$ has positive curvature. It follows that $M\setminus G$ is connected, and we can take a continuous curve $\gamma(t)\in M\setminus G, t\in [0,1],$ connecting $x$ to $y$. Then $x\gamma(t)$ does not intersect $S$. We claim that for each $t$, the line segment $x\gamma(t)$ is contained in $B_M$. In fact, we show that the set $I\subset [0,1]$ of $t$ for which this property holds is open and closed in $[0,1]$. Notice that $I$ is clearly closed in $[0,1]$ since $B_M$ is compact. It must also be open since $x\gamma(t), t\in I,$ must be transverse to $M$ at both $x$ and $\gamma(t)$ since otherwise, $x\gamma(t)$ contains points in $C_M$, a contradiction. Hence $xy\subset B_M$. The other cases for $x$ and $y$ follow from continuation.  Hence, $B_M$ is convex. Since the curvatures are positive, the intersection of each hyperplane tangent to $M\setminus S$ with $S$ is a single point, and the contact has order $1$.
\end{proof}

Theorems \ref{thm: convexity formula} and \ref{thm: Copenhagen problem} imply that all sectional curvatures of $\mathcal{M}_{1/2,E}^e$ and $\mathcal{M}_{1/2,E}^m$ are positive for every $E\leq -2$. This local condition combined with Proposition \ref{prop_section curvature implies convexity} implies that both $\mathcal{M}_{1/2,E}^e$ and $\mathcal{M}_{1/2,E}^m$ are strictly convex for every $E\leq -2$. This completes the proof of Theorem \ref{thm_3bp2}.

\section{Proof of Theorem \ref{thm_Birkhoff_conj}}

 First we claim that for $(\mu,E)$ sufficiently close to $(-1/2,-2)=(-1/2,L_1(1/2))$, with $E< L_1(\mu)$, the $\R P^3$-components $\mathcal{M}_{\mu,E}^e$ and $\mathcal{M}^m_{\mu,E}$ are dynamically convex. Indeed, we know from Theorem \ref{thm_muRSl1} that there exists a neighborhood $\U_3\subset \R^2 \times (\R \times \R / 2\pi \Z)$ of the singularities $S_\pm(1/2)$ corresponding to $l_1(1/2)$ so that, for $(\mu,E)$ sufficiently close to $(1/2,-2)$ and any contractible periodic orbit $P'\subset \mathcal{M}_{\mu,E}^e \cup \mathcal{M}_{\mu,E}^m$, that is not a cover of the Lyapunov orbit near $l_1(\mu)$ and intersects $\U_3$, the Conley-Zehnder index of $P'$ is $>3$. Because $\mathcal{M}_{1/2,-2}^e$ and $\mathcal{M}_{1/2,-2}^m$ are strictly convex, we also know that for $(\mu,E)$ sufficiently close to $(1/2,-2)$, with $E<L_1(\mu)$, any periodic orbit in $(\mathcal{M}_{\mu,E}^e \cup \mathcal{M}_{\mu,E}^m) \setminus \U_3$ has index $\geq 3$.

 Now since $\mathcal{M}_{1/2,E}^e$ and $\mathcal{M}_{1/2,E}^m$ are strictly convex for every $E< -2=L_1(1/2)$, we find an open neighborhood $\V \subset \R^2$ of $\{1/2\} \times (-\infty,-2)$ so that both $\mathcal{M}_{\mu,E}^e$ and $\mathcal{M}_{\mu,E}^m$ are strictly convex for $(\mu,E) \in \V$. For $E \ll 0$, the neighborhood can be taken uniform since both $\mathcal{M}_{\mu,E}^e$ and $\mathcal{M}_{\mu,E}^m$ are uniformly strictly convex. We conclude that there exists $\epsilon_0>0$ so that for every $|\mu-1/2| < \epsilon_0$ and $E<L_1(\mu)$, both $\mathcal{M}_{\mu,E}^e$ and $\mathcal{M}_{\mu,E}^m$ are dynamically convex. This proves (i). Items (ii) and (iii) directly follow from the main results in \cite{hryn2, HLS, HS2, HSW}.

\appendix

\section{\texorpdfstring{Non-negative paths in $\mathrm{Sp}(2n)$}{Non-negative paths in Sp(2n)}}\label{sec: positive paths}
In this section, we prove the following statement, which is used to obtain a lower bound on the index of periodic orbits of Hamiltonians with a magnetic term in the kinetic energy.

\begin{prop}\label{prop_non_negative_path}Given $M \in \mathrm{Sp}(2n)$, there exists a smooth path $\beta:[0,1] \to \mathrm{Sp}(2n)$ with $\beta(0) = I_{2n}, \beta(1) = M$ and $-J\dot \beta(t) \beta(t)^{-1} \geq 0,\ \forall t\in [0,1]$. Moreover, $n \leq \mu_{\text{RS}}(\beta) \leq 2n.$
\end{prop}

\begin{proof} Define $M_1\diamond M_2$ as the $(n_1+n_2)\times (n_1+n_2)$ matrix
$$M_1\diamond M_2=\begin{pmatrix}
A_1 & 0 & B_1 & 0 \\
0 & A_2 & 0 & B_2 \\
C_1 & 0 & D_1 & 0 \\
0 & C_2 & 0 & D_2
\end{pmatrix},\quad \mbox{ where} \quad  M_i=\begin{pmatrix}A_i & B_i \\ C_i & D_i \end{pmatrix}_{n_i\times n_i},\ i=1,2.$$
Since $M$ is generically semi-simple (diagonalizable), Theorem 1.7.3 from \cite{Long02} implies the existence of $P\in\mathrm{Sp}(2n)$ such that
$M=P(N_1\diamond \cdots\diamond N_k)P^{-1},$
where $N_i\in\mathrm{Sp}(2n_i), \sum_{i=1}^kn_i=n$, is one of the following matrices in normal form
$$
R(\theta)=\begin{pmatrix} \cos\theta & -\sin\theta \\ \sin\theta & \cos\theta \end{pmatrix},\quad
D(\lambda)=\begin{pmatrix} \lambda & 0 \\ 0 & \lambda^{-1} \end{pmatrix},\quad
D(\rho,\theta)=\begin{pmatrix} \rho R(\theta) & 0 \\ 0 & \rho^{-1}R(\theta) \end{pmatrix},
$$
with $\theta\in \R/2\pi\Z$, $\lambda\in\mathbb{R}\setminus\{0\}$ and $\rho>0$.
If $N_i=R(\theta),\ 0<\theta\leq 2\pi$, we define
$N_{i,t}=R(t\theta),t\in[0,1].$
If $N_i=D(\lambda)$, we first define
$$N_{i,t}=\left\{\begin{aligned}
&D(\lambda)R(\pi+\pi t)\quad &\text{if}\ \lambda>0,\\
&D(-\lambda)R(t\pi)\quad\quad &\text{if}\ \lambda<0,
\end{aligned}\right.\quad \forall t\in[1/2,1].$$
In particular, the eigenvalues of $N_{i,1/2}$ are $\{i,-i\}$. Then there exists $Q_\pm\in\mathrm{Sp}(2)$ such that
$N_{i,1/2}=Q_\pm(\mp J)Q_\pm^{-1}$ with $\lambda=\pm|\lambda|$. We further define
$$N_{i,t}=\left\{\begin{aligned}
&Q_+ R(3\pi t)Q_+^{-1}\quad &\text{if}\ \lambda>0,\\
&Q_- R(\pi t)Q_-^{-1}\quad  &\text{if}\ \lambda<0,
\end{aligned}\right.\quad \forall t\in[0,1/2].$$
If $N_i=D(\rho,\theta),\theta\in[0,2\pi)$, we first define
$$N_{i,t}=\begin{pmatrix}
\rho I_2 & 0 \\ 0 & \rho^{-1}I_2
\end{pmatrix}\begin{pmatrix}
R((2t-1)\theta) & 0 \\ 0 & R((2t-1)\theta)
\end{pmatrix},
\quad \forall t\in[1/2,1],$$
where $N_{i,1/2}=D(\rho)\diamond D(\rho)=:D(\rho)^{\diamond 2}$. Then we define
$$
N_{i,t}=\left\{\begin{aligned}
& (D(\rho)R((2t+1)\pi))^{\diamond 2},\quad \text{if}\ t\in[1/4,1/2],\\
& (QR(6\pi t)Q^{-1})^{\diamond 2}, \quad\quad\quad  \text{if}\ t\in[0,1/4],
\end{aligned}\right.
$$
where $Q\in \mathrm{Sp}(2)$ satisfies $D(\rho)R(3\pi/2)=Q(-J)Q^{-1}$.

Since $R(at)$ satisfies $-J\frac{d}{dt} R(at) R(at)^{-1}=-aJ\cdot JR(at)R(-at)=aI_2>0$ for any $a>0$, we conclude that $N_{i,t}$ is a piecewise positive regular path. The Robbin-Salamon index satisfies $1\leq\mu_{\text{RS}}(N_{i,t}|_{t\in[0,1]})\leq 2$ for $N_{i,1}=R(\theta),D(\lambda)$ and $2\leq \mu_{\text{RS}}(N_{i,t}|_{t\in[0,1]})\leq 4$ for $N_{i,1}=D(\rho,\theta)$. After reparametrizing each path  $N_{i,t}, t\in [0,1],$ at $t=1/4,1/2$ by slowing down the parametrization at these points, and perturbing the end at $t=1$, see \cite{Long02}, we obtain a smooth non-negative regular path $\beta(t):=P(N_{1,t} \diamond \cdots \diamond N_{k,t}) P^{-1}, t\in [0,1],$ from $I_{2n}$ to $M$ satisfying $-J \dot \beta(t) \beta(t)^{-1} \geq 0$ and whose Robbin-Salamon index satisfies $n\leq \mu_{\text{RS}}(\beta)\leq 2n$.
\end{proof}

\section{Supplementary proofs}\label{sec: E2>0}
This section is devoted to the proof of the following lemma.

\begin{lem}\label{lem: E2>0}
The function $E_2=E_2(s_2,c)$, given in \eqref{equ: E2}, is positive on $[0,1]\times [0,\sqrt{17/14}]$.
\end{lem}

\begin{proof}
Let $l_2:=s_2^2\in [0,1]$, and rewrite $E_2$ as
\begin{equation}\label{equ: E2 1}
\begin{aligned}
64E_2(l_2,c) & = E_3 l_2^3 (1 - l_2) + h_1 l_2^7 (1 - l_2) + h_2 l_2^7 + h_3 l_2^6 + h_4 l_2^2 (1 - l_2)^4 + h_5 l_2 (1 - l_2)^5 \\
& + h_6 (1 - l_2)^6 + 10000 c^4 \big(3 - 5 l_2)^2 (31(4 - 5 c)^2 l_2^4/50 + (3 - 4c)^2 (1 - l_2) l_2^2 \\
& + (1 - 2 c)^4 (1 -l_2) l_2  + c^4 l_2^4\big) + 80 c^9 l_2^6 (1 - l_2) +  256 c^{16} l_2^8 (1 - l_2^2),
\end{aligned}
\end{equation}
for every $(l_2,c)\in [0,1] \times [0,\sqrt{17/14}]$, where
\begin{align*}
E_3(l_2,c)&:=k_1(c)l_2^2+k_2(c)l_2+k_3(c),\\
h_1(l_2,c)&:=16 c^{11} (1 + 32 c - 4 c^2 + 256 c^3 - 8 c^4 + 32 c^5 + 16 c^3 (8 + c^2) l_2)>0,\\
h_2(c)&:=8 c^{10} (2368 - 151 c - 128 c^2 + 160 c^3 - 512 c^4 + 48 c^5 - 256 c^6)>0,\\
h_3(c)&:=6608 - 22538 c + 37760 c^2 + 43056 c^3 - 471552 c^4 + 1083859 c^5 - 667488 c^6\\
&\quad\ - 22371 c^7 + 53248 c^8 - 57015 c^9 + 52096 c^{10} - 1459 c^{11} - 896 c^{12} + 5440 c^{13} \\
&\quad\ - 3584 c^{14} + 624 c^{15} + 1088 c^{16},\\
h_4(c)&:=99120 - 111746 c + 188800 c^2 - 709972 c^3 - 272752 c^4 + 2377365 c^5 \\
&\quad\ - 2809728 c^6 + 2257746 c^7 - 1074240 c^8 + 129136 c^9 + 7616 c^{11},\\
h_5(c)&:=39648 - 22302 c + 37760 c^2 - 407878 c^3 + 259440 c^4 + 725190 c^5 \\
&\quad\ - 2140416 c^6 + 3019152 c^7 - 1440000 c^8 + 27200 c^9,\\
h_6(c)&:=6608 - 89964 c^3 + 76160 c^4 + 34272 c^5 + 15232 c^7.
\end{align*}
and
$$
\begin{aligned}
k_1(c)&:=72688 - 111982 c + 188800 c^2 - 349002 c^3 + 1552416 c^4 - 3429694 c^5 + 2538200 c^6\\
&\quad\  - 768657 c^7 + 642128 c^8 + 53009 c^9 + 33152 c^{10} + 4267 c^{11} - 896 c^{12} + 3784 c^{13}\\
&\quad\  + 8704 c^{14} - 128 c^{15},\\
k_2(c)&:=-165200 + 223492 c - 377600 c^2 + 991280 c^3 - 1742336 c^4 + 3920764 c^5, \\
&\quad\ - 5570360 c^6  + 5195743 c^7 - 2587792 c^8 - 319833 c^9 + 33152 c^{10} - 17470 c^{11}\\
&\quad\ + 2176 c^{12} - 3088 c^{13} - 1088 c^{15},\\
k_3(c)&:=132160 - 223964 c + 377600 c^2 - 539456 c^3 + 779872 c^4 - 2021770 c^5 + 3175040 c^6, \\
&\quad - 3205287 c^7 + 1464320 c^8 + 236090 c^9 - 80512 c^{10} + 11504 c^{11} - 5440 c^{13}.
\end{aligned}
$$

The functions $h_1(l_2,c)$ and $h_2(c)$ are clearly positive for $0\leq l_2\leq 1$ and $0\leq c\leq \sqrt{17/14}$. Let $\beta:=6/5-c.$. Then $\beta>0$ for every $c\in[0,\sqrt{17/14}]\subset [0,6/5]$. We rewrite the expressions for $h_3, h_4, h_5$ and $h_6$ to obtain
\begin{align*}
h_3&= \bigg(\frac{17638696}{625} +\frac{5496536}{125} c + \frac{4054798}{25} c^2  + \left(\frac{487824}{5} + 634791 \beta \right) c^3\bigg) (\frac{2}{5} - c)^2\beta\\
&\quad\ \frac{1799104}{3125} c\beta + \frac{2 (46460648 - 2460625 c)}{78125} + (332348 + 145849 c + 50716 c^2) c^4 (1 - c)^4 \\
&\quad\ + (1380 + 1138 c) c^{10} (1 - c)^2 + 163 c^{11} + (3584 \beta + \frac{6}{5}) c^{13} + 624 c^{15} + 1088 c^{16},\\
h_4&=c(\frac{7}{10} - c)^2 \bigg(\frac{32498403269}{125000} + \beta\bigg(\frac{55874294}{3125} + \frac{183617569}{250} c + \frac{114118354}{125} c^2 + \frac{1436406}{25} c^2 \beta \\
&\quad\ + \left(\frac{2917616}{5} + 129136 \beta \right) c^4 \bigg)\bigg) + \left(99120 - \frac{3120663099669}{12500000}c + \frac{99820440049}{625000}c^2\right) + 7616 c^{11},\\
h_5&=c (\frac{13}{20} - c)^2 \bigg(\frac{3703539883}{40000} + \beta\bigg(\frac{92744181}{1000} + \frac{52681943}{100} c + \frac{3077836}{5} c^2  + 503940 c^3 \\
&\quad\ + 1404640 c^4\bigg) + 27200 c^6\bigg)  + \frac{22592124459}{80000000} + \frac{489006553 c}{3200000}  +\frac{3744646701}{50000} (\frac{29}{40}-c)^2,\\
h_6&= c^2 (\frac{3}{5} - c)^2 \left(\frac{16305856}{125} + \frac{1268064}{25} c + \frac{91392}{5} c^2 + 15232 c^3\right) + \frac{6039012}{125} c (\frac{13}{20} - c)^2\\
&\quad\   + \left(6608 - \frac{255148257}{12500}c + \frac{49515186}{3125}c^2\right),\quad c\in [0,6/5].
\end{align*}
One can easily check that $h_3, h_3, h_5$ and $h_6$ are positive on $[0,6/5]$. Therefore, it is sufficient to prove that $E_3(l_2,c)$ is positive on $[0,1]\times [0,6/5]$.

Now, we rewrite $k_1$ and $k_3$ and check that both are positive on $[0,6/5]$
\begin{align*}
k_1(c)&= \left(\frac{581229138}{15625} + 500000 c^3 + \frac{63625694}{25} c^4\right) (\frac{7}{10} - c)^2  + \frac{473182277c}{125} (\frac{3}{10} - c)^2 (\frac{9}{10} - c)^2\\
&\quad\ + \left(\frac{948554347}{6250} + \frac{136241809}{1250} \beta \right) (\frac{2}{5} - c)^2 \beta + \left(642128 \beta^4  + \frac{19273323c}{5} (\frac{3}{5} - c)^2\right) c^2 \beta^2\\
&\quad\  + c^9 (53009 + 33152 c^{1} + 4267 c^{2} - 896 c^{3} + 3784 c^{4}  + 8704 c^{5} - 128 c^{6}) + \frac{32652259}{156250} \\
&\quad\  + \frac{15311007}{10000}c >0,\\
k_3(c)&=c^2 (\frac{17}{20} - c)^2 \bigg(\frac{2355466141019}{4000000} +
\frac{170920469037}{400000} c + \left(\frac{3936935827}{5000}  + \frac{401995653}{1000} \beta\right) c^2 \\
&\quad\  + \frac{42279081}{25} c^4 + \left(\frac{13026}{5} + 80512 \beta\right) c^5\bigg) + \frac{24462722037599}{160000000} c (\frac{4}{5} - c)^2 + 11504 c^{11} \\
&\quad\ - 5440 c^{13} + \left(132160  - \frac{80453722037599 c}{250000000} + \frac{314833837847093 c^2}{1600000000}\right) >0.
\end{align*}
Computing $E_3$ and the line segment $s_2=1$ and $c\in [0,6/5]$, we obtain
\begin{align*}
E_3(1,c)&=k_1(c)+k_2(c)+k_3(c)\\
&=\bigg(39648 - \frac{28903204136204 c}{244140625} + \frac{5470057737716 c^2}{48828125}\bigg) +\frac{1773112492}{78125} c^3 \beta^2 \\
&\quad\  + c^3 \beta^3 \bigg(\frac{16976668754}{78125} +\frac{11890235417}{15625} c +  \frac{1721158044}{3125} c^2\bigg) \\
& \quad\  + c^8 \beta \bigg(\frac{36159674}{625} + \frac{2825154}{125} c + \frac{174859}{25} c^2 +  \frac{22064}{5} c^3 + 4744 c^4\bigg) \\
&\quad\ + c (\frac{3}{5} - c)^2 \bigg(\frac{160957143606}{9765625} + \frac{104765224976c}{390625} \bigg) \\
&\quad\ + 8704 c^{14} - 1216 c^{15}>0,\quad \forall  c\in [0,6/5].
\end{align*}
Then we compare $E_3(l_2,c)-k_3(c)(1-l_2)^2-E_3(1,c)l_2^2=l_2(1-l_2)k_5(c)$, where
$$\begin{aligned}
k_5(c)&:=99120 - 224436 c + 377600 c^2 - 87632 c^3 - 182592 c^4 - 122776 c^5 + 779720 c^6 \\
&\quad\ - 1214831 c^7 + 340848 c^8 + 152347 c^9 - 127872 c^{10} +  5538 c^{11} + 2176 c^{12}\\
&\quad\ - 13968 c^{13} - 1088 c^{15}.
\end{aligned}$$
We compute
$$\begin{aligned}
-k_5'(c)&=\left(224436 - \frac{39166298761 c}{40000} + \frac{8584984187 c^2}{8000}\right) + c^2 (\frac{1}{2} - c)^2 \bigg(\frac{296223209}{400} \\
&\quad\ + \frac{29197057}{20} c + \frac{2934884627}{2500} c^2 + \beta \left(\frac{433202652}{125} c^2 + \frac{5077647}{25} c^3\right) + \frac{29200428}{25} c^5\bigg) \\
&\quad\ + \frac{110596281c}{100} (\frac{9}{20} - c)^2 + c^9 \beta \left(\frac{461262}{5} + 26112 c\right) + 181584 c^{12} + 16320 c^{14}>0,
\end{aligned}
$$
for every $c\in [0,6/5].$
Therefore, $k_5(c)$ is a decreasing function on that interval. Since $k_5(4/5)=14698.5>0$, we conclude that $k_5(c)>0$ for every $c\in [0,4/5]$. This implies that $E_3(l_2,c)>0$ for every $(l_2,c)\in [0,1]\times [0,4/5]$. This implies that $E_3(l_2,c)$ is positive for every $(l_2,c)\in [0,1] \times [0,4/5].$ It remains to show that $E_3$ is positive on $[0,1] \times [4/5,6/5]$.

Since $k_1(c)>0$ for every $c\in  [0,6/5]$, $E_3(l_2,c)$ is a convex function in $l_2$. For every fixed $c\in[4/5,6/5]$, the minimum value of $E_3(\cdot, c)$ is given by
$$
E_3\left(-\frac{k_2}{2k_1},c\right)=\frac{1}{4k_1(c)}(4k_1(c)k_3(c)-k_2(c)^2)=:\frac{1}{4k_1(c)}D_0(c).
$$
We shall prove that $D_0(c)>0$ for every $c\in[4/5,6/5]$. We consider $c=2v/5+4/5$, where $v\in [0,1]$. Notice that $D_0(2v/5+4/5)$ is a polynomial in $v$ of degree $30$ whose coefficients are all rational numbers. Replacing all such coefficients $a$ by $\lfloor a\rfloor$, we obtain a new polynomial $D_1(v)$ satisfying
$$
\begin{aligned}
&\quad\ D_0(2v/5+4/5)\geq D_1(v)\\
:&=1171163506 - 7471707255 v + 16174395360 v^2 + 8922925732 v^3 - 44809560461 v^4 \\
& - 33799370597 v^5  + 45096619475 v^6 + 78595290662 v^7 + 37015742898 v^8 - 14346554736 v^9 \\
&  - 29900538237 v^{10} - 19491958537 v^{11} - 6963317252 v^{12} - 1208863632 v^{13} + 45920428 v^{14} \\
&  + 44662945 v^{15} - 18645624 v^{16} - 19608465 v^{17} - 8327827 v^{18} - 2388872 v^{19} - 519800 v^{20} \\
&  - 89861 v^{21} - 12669 v^{22} - 1482 v^{23} - 145 v^{24} - 12 v^{25} - v^{26} - v^{27} - v^{28} - v^{29} - v^{30},
\end{aligned}
$$
for every $v\in[0,1]$. We can rewrite $D_1$ as
$$
\begin{aligned}
D_1(v)&=\frac{24063486356113}{320000} + \frac{5194158621777 v}{80000} \\
&+ (\frac{1}{2} - v)^2 \big(\frac{14476579533567}{16000} +  \frac{2258156665941}{32} (\frac{1}{5} - v)^2\big) \\
& + (\frac{2}{5} - v)^2 \big(\frac{820472989093}{800} (1 - v) + \frac{587799303261}{8} v (\frac{3}{5} - v)^2\big) \\
& + v^4 (\frac{1}{2} - v)^2 \big(\frac{419040213759}{8} + (1 - v) \big(\frac{147822902909}{2}  + \frac{299960460267}{2} v\\
& + 129475910052 v^2 + 57564677658 v^3 \big)\big) + v^{10} (1 - v) (27664139421 \\
&  + 8172180884 v + 1208863632 v^2 + 4931817 v^4 + 49594762 v^5  + 30949138 v^6\\
& + 11340673 v^7 + 3012846 v^8 + 623974 v^9 + 104174 v^{10} + 14313 v^{11})+ 40988611 v^{14} \\
& + (1 - v) v^{22} (1644 + 162 v + 17 v^2 + 5 v^3 + 4 v^4 + 3 v^5 + 2 v^6 + v^7)>0,\quad \forall v\in [0,1].
\end{aligned}
$$
Therefore, $D_0(c)>0$ for every $c\in [4/5,6/5]$ and thus $E_3(l_2,c)>0$ on $[0,1]\times [4/5,6/5]$. Together with the estimate for every $c\in[0,4/5]$, we conclude that $E_3(l_2,c)>0$ for every $(l_2,c)\in [0,1]\times [0,6/5]$. Finally,  \eqref{equ: E2 1} and the previous estimates for $h_1,\ldots, h_6$ show that $E_2(l_2,c)>0$ for every $(l_2,c)\in [0,1]\times [0,\sqrt{17/14}]$. This finishes the proof of the lemma.
\end{proof}

\end{document}